 \documentclass{amsart}
 \usepackage[top=2.74cm,bottom=2.74cm,left=2.5cm,right=2.5cm,marginparwidth=2cm, marginparsep=3mm]{geometry}

\usepackage{MC_macros}

\usepackage[font=scriptsize,labelfont=bf]{caption}
\usepackage[textsize=tiny]{todonotes}
\usepackage[shortlabels]{enumitem}

\usepackage[hidelinks]{hyperref}
\usepackage[nameinlink]{cleveref}
\usepackage{stmaryrd} 

\declaretheorem[
name=Theorem,
refname={Theorem,Theorems},
Refname={Theorem,Theorems},
numberwithin=section,
]{theorem}
\declaretheorem[
name=Proposition,
refname={Proposition,Propositions},
Refname={Proposition,Propositions},
sibling=theorem,
]{proposition}
\declaretheorem[
name=Lemma,
refname={Lemma,Lemmas},
Refname={Lemma,Lemmas},
sibling=theorem,
]{lemma}
\declaretheorem[
name=Corollary,
refname={Corollary,Corollaries},
Refname={Corollary,Corollaries},
sibling=theorem,
]{corollary}

\declaretheoremstyle[
  spaceabove=2pt, spacebelow=2pt,
  bodyfont=\normalfont\itshape,
 postheadspace=1em
]{claimnonum}
\declaretheorem[
name=Claim,
numbered=no,
style=claimnonum
]{claim*}

\declaretheoremstyle[
  spaceabove=2pt, spacebelow=2pt,
  bodyfont=\normalfont,
 postheadspace=1em
]{opennonum}
\declaretheorem[
name=Open Problem,
style=opennonum
]{openproblem}

\declaretheorem[
name=Definition,
refname={Definition,Definitions},
Refname={Definition,Definitions},
sibling=theorem,
style=definition,
]{definition}

\declaretheorem[
name=Remark,
refname={Remark,Remarks},
Refname={Remark,Remarks},
sibling=theorem,
style=remark,
]{remark}
\declaretheorem[
name=Remark,
style=remark,
numbered=no,
]{remark*}

\newcommand\crev[1]{{#1}}
\newcommand\cadd[1]{{#1}}

\usepackage[textsize=tiny]{todonotes}

\title{On the structure and theory of M\MakeLowercase{c}Carthy algebras}
\author{Stefano Bonzio 
    \and 
    Gavin St.\,John
    }
\address{Department of Mathematics and Computer Science, University of Cagliari, Italy}
\email{stefano.bonzio@unica.it \and gavinstjohn@gmail.com}
\date{}


\begin{document}
\begin{abstract}
We provide a structural analysis for McCarthy algebras, the variety generated by the three-element algebra defining the logic of McCarthy (the non-commutative version of Kleene three-valued logics). 
Our analysis will be conducted in a very general algebraic setting by introducing McCarthy algebras as a subvariety of unital bands (idempotent monoids) equipped with an involutive (unary) operation $\nc{}$ satisfying $\nnc{x}\approx x$; herein referred to as {\iname}s. 
Prominent (commutative) subvarieties of {\iname}s include Boolean algebras, ortholattices, Kleene algebras, and involutive bisemilattices, hence {\iname}s provides an algebraic common ground for several non-classical logics. 
Our main contributions consist in providing for McCarthy algebras: reduced and equivalent axiomatizations; a semilattice decomposition theorem; and representations as certain decorated posets from which the algebraic structure can be uniquely determined.  
\end{abstract}
\maketitle
\section{Introduction}

McCarthy algebras consist of algebraic structures playing the role of algebraic counterpart to McCarthy logic, a three-valued logic \cite{Mc59} based on a non-commutative conjunction. McCarthy logic is used to interpret the lazy evaluation of partial predicates, particularly attractive for computing purposes and adopted in several programming languages, such as Haskell and OCaml (the function is currently supported also in non purely functional programming languages, as Java and Python).
More in general, this formalism proves to be very useful for the process-algebraic treatments of computational processes affected by errors \cite{BP98, AldiniMezzina}. 
Beyond the usual operations on processes, including non-deterministic choice, sequential and parallel composition, the process algebra formalisms studied in \cite{BP98} and \cite{AldiniMezzina} are enriched with a constant, designed to model the error process and conditional guard statements of the form $\varphi :\to P$, where $\varphi$ is a logical formula and $P$ a process, whose reading is: in case $\varphi$ holds (in a certain logic) then execute $P$. 
While the two mentioned proposals disagree on the choice of the logic taking care of conditional guards, they both agree that it should necessarily include McCarthy conjunction, as the correct logical tool for properly modeling the sequential composition of conditional guards, witnessed by the axiom $\varphi:\to (\psi:\to P) = \varphi~\mathsf{and}~ \psi:\to P$, where $\mathsf{and}$ is interpreted by McCarthy conjunction (denoted by $\mc$ in \Cref{fig:M3}). 
\crev{More specifically, the conditional guard ``if $\varphi\mc\psi $, then do $P$'' will
execute $P$ if both formulas yield the value
true, it will be skipped in case $\varphi$ yields the value false or if $\varphi$ yields the values true
and $\psi$ yields the value false, while computation will lead to a failure (or an error)}\footnote{The process algebras' formalisms introduced in \cite{BP98} and \cite{AldiniMezzina} are enriched with a construct modeling the error process.} \crev{if $\varphi$ yields the third value or if $\varphi$ yields the value true but $\psi$ the third value.} 

\begin{figure}[ht]
    \centering
    $\begin{array}{c|c}
        \nc{} &  \\ \hline
         1&0\\
         0&1\\
         \varepsilon&\varepsilon
    \end{array}
    \qquad
    \begin{array}{c|ccc}
        \jc{} & 1 & 0 & \varepsilon  \\ \hline
         1&1&1&1\\
         0&1&0&\varepsilon\\
         \varepsilon&\varepsilon&\varepsilon&\varepsilon
    \end{array}
    \qquad
    \begin{array}{c|ccc}
        \mc{} & 1 & 0 & \varepsilon  \\ \hline
         1&1&0&\varepsilon\\
         0&0&0&0\\
         \varepsilon&\varepsilon&\varepsilon&\varepsilon
    \end{array}$
    \caption{The tables of the 3-element algebra $\McA = \str{\{0,1,\varepsilon\}, \jc,\mc,\nc{}, 0,1}$. 
    The operations $\mc,\jc$ denote McCarthy conjunction and disjunction, respectively.} \label{fig:M3}
\end{figure}

Despite its indisputable usefulness in computer science, not much is known about McCarthy logic (see e.g., \cite{FittingChildren,Indrzejczak_Petrukhin_2024}), here understood (in accordance with \cite{Kow96}) as the logic induced by the logical matrix $\langle\McA , \{1\}\rangle$, where $\McA$ is the 3-element algebra (of truth-tables) given in \Cref{fig:M3}. 
This translates to saying that 
\[\Gamma\vdash\varphi \quad\iff\quad \forall h\in \mathrm{Hom}(\mathbf{Fm},\McA),\; h[\Gamma]\subseteq\{1\} \;\Rightarrow\; h(\varphi) = 1, 
\]
where $\mathbf{Fm}$ is the absolutely free algebra generated by the algebraic language of McCarthy logic $\langle\mc, \jc, \nc{}\rangle$.

As defined by Konikowska in \cite{Kow96}, an algebra $\m A = \str{A,\jc,\mc,\nc{},0,1}$ is called a \defem{McCarthy algebra} if $\m A$  \emph{``satisfies all the equational tautologies of a Boolean algebra that hold in''} the algebra $\McA$. 
Restating this with the parlance of universal algebra and utilizing the celebrated result of Birkoff, $\m A$ is a McCarthy algebra iff it is a member of the variety of algebras generated by $\McA$. In this way, we can define the \defem{variety of McCarthy algebras} $\MC$ via
$$\MC :=  \mathsf{V}(\McA). $$

Note that the restriction of the operations to $\{0,1\}$ defines the two element Boolean algebra $\m 2$, i.e., the subalgebra $\m B_0$ of $\McA$ generated by $\{0,1\}$ is isomorphic with $\m 2$. 
As the variety of Boolean algebras is generated by the two-element Boolean algebra $\m 2$, which is isomorphic to a subalgebra of $\McA$, it is immediate that $\BA$ is a (proper) subvariety of $\MC$.

\cadd{Despite Konikowska \cite{Kow96} left open the problem of finding an equational basis for the variety of McCarthy algebras, we only recently realized that the problem was actually previously solved by Guzm{\'a}n and Squier in \cite{Guzman-Squier},}\footnote{We thank the colleagues P. Graziani, P. Jipsen, and U. Rivieccio for suggesting (independently) the reading of \cite{Guzman-Squier} after reading the draft of the present work.}
\cadd{ where (the constant-free reduct of) McCarthy algebras are referred to as Conditional algebras (or \emph{$C$-algebras}). In this work we prefer to call them McCarthy algebras (assuming also that the language contains constants).}
They are defined over the algebraic language $\str{\jc, \mc, \nc{},0,1}$ (of type $(2,2,1,0,0)$), however $+$ is just the De~Morgan dual of $\mc{}$ and $0$ the negation of $1$, thus term-definable in the language $\str{\mc, \nc{},1}$ (the choice of including the constant symbols will become clear later). 
Observe moreover that, in this restricted language, $\McA$ is an idempotent monoid (with $1$ as unit) equipped with an involution.

\cadd{In this paper, we introduce} a very general class (variety) of algebras: 
unital bands (or rather, idempotent monoids) with an involution, herein referred to as \emph{{\iname}s},
\cadd{which allow for a unified algebraic treatment of McCarthy and (strong and weak) Kleene logics, of which the former can be seen as the non commutative version.} 
It is worth mentioning that {\iname}s differ from (the monoid expansion of) involution band semigroups introduced in \cite{DolinkaInvolutionBands}, as we are only requiring involution to satisfy the identity $\nnc{x} \approx x$, and not additionally $\nc{(x\cdot y)} \approx \nc{y}\cdot\nc{x}$. 
The involution-free reduct of {\iname}s coincides with the variety of unital bands (i.e., idempotent monoids), whose lattice of subvarieties is characterized in \cite{VarietiesBandMonoids} (see also \cite{AllVarofBands}). 
Interestingly, and aside from their obvious generalization of Boolean algebras and classical propositional logic, {\iname}s are general enough to provide a common root for a large class of other well-studied algebras related to non-classical logics, such as ortho(modular) lattices (related to the foundation of quantum logic), De~Morgan algebras \crev{(related to paraconsistent as well as mathematical fuzzy logic)}, and in particular also for Kleene 3-valued logics (see \cite{KleeneBook,Bonziobook}), showing that the logic of McCarthy can be seen as their non-commutative companion (see also \cite{FittingChildren}).

\cadd{The main advantage derived by introducing {\iname}s is that of allowing a comparative study of algebras related to many logical formalisms in a non-commutative setting. Relative to McCarthy logic, our main contributions consist in providing reduced and equivalent axiomatizations (\cref{t:equivAxioms}), providing a decomposition theorem in terms of a semilattice direct system of Boolean algebras (\cref{t:Decomp}), and characterizing representations of McCarthy algebras as certain decorated posets (\cref{t:isopo}).}

The paper is structured as follows: Section \ref{Sec: preliminaries} introduces the basic preliminary notions \crev{used in this paper}.
In Section \ref{sec: variety of iname}, the main algebraic structures that will be developed in the whole paper are introduced: {\iname}s, i.e., any structure $\str{M,*,\nc{},\miden}$ consisting of an idempotent monoid reduct $\str{M,*,\miden}$ (a band with unit) and unary involutive operation $\nc{}$ (satisfying $\nnc{x} \approx x$). 
After discussing some relevant three-element examples of {\iname}s (whose relevance will become clear later), we will focus on the class of subclassical {\iname}s, i.e., where each member has the two-element Boolean algebra as a subalgebra, \crev{which turns out to be a quasivariety.}
\cadd{In \Cref{sec:findingidens}, we discuss relations between different identities holding in  {\iname}s, to guide the reader towards the first main result of the paper which consists of reduced and equivalent axiomatizations for the variety $\mathsf{M}$ of McCarthy algebras.
\Cref{sec:axiomMC} is devoted to proving this result, passing through the understanding of the structure of any McCarthy algebra based on a (largely unexplored) partial order induced by addition. In \Cref{sec:decomp}, we capitalize on the analysis of this structure to provide the other main results of this paper, a semilattice decomposition theorem for McCarthy algebras and their representations as decorated posets. 
Finally, in \Cref{sec: subvarieties} we prove some results on subvarieties of {\iname}s which allow us to provide a sketch of the lower levels of the lattice of all subvarieties of {\iname}s. }

\section{Preliminaries}\label{Sec: preliminaries}

\subsection*{Universal Algebraic background}

We will make use of several concepts from universal algebra throughout our discussion: we recall here some basic definitions and results that will be used, we refer to \cite{BergmanLibro, BuSa00, mckenzie} for more detailed treatments. 
We denote the class operators of closure under homomorphic images, subalgebras, and direct
products by $\mathsf{H}$, $\mathsf{S}$, and $\mathsf{P}$, respectively. 
A class of (similar) algebras is said to be a \emph{variety} if it is closed under $ \mathsf{H}$, $ \mathsf{S}$, and $ \mathsf{P}$. 

By an \emph{identity} $s\approx t$, we refer to the universally quantified formula $\forall \vec{x}[s(\vec{x})= t(\vec{x})]$, where $s$ and $t$ are terms in the absolutely free algebra in a given (fixed) signature over a set of formal variables $X=\{x_i\}_{i\in \N}$, and $\vec{x}=(x_1,\ldots,x_n)$ are those variables present in the terms $s$ and $t$. 
In this way, an algebra $\m A$ \emph{satisfies} an identity $s\approx t$ iff the equation $s(\vec{a})=t(\vec{a})$ holds for every instantiation of the variables by $\vec{a}\in A^n$. 
Here, we will often use the letters $x,y,z,\ldots$ to denote formal variables.
Birkhoff's Theorem establishes that any given variety coincides with the class of all algebras satisfying some set of identities (see e.g., \cite[Thm. II.11.9]{BuSa00}).
Similarly, a \emph{quasi-identity} is any universally quantified formula $\forall\vec{x} [\bigwedge_{i=1}^n s_i=t_i \Rightarrow s_0=t_0]$ for some list of terms $s_0,t_0,\ldots,s_n,t_n$; note that an identity is simply a quasi-identity in which the antecedent is the empty conjunction (i.e., when $n=0$). 
The notion of algebras/classes satisfying quasi-identities is essentially the same. 
We note that a class of algebras is called a \emph{quasivariety} if it can be axiomatized a set of quasi-identities.\footnote{Quasivarieties also correspond to those classes closed under isomorphic images, subalgebras, products, and ultraproducts \cite[Cf., Thm.~2.25]{BuSa00}; but this result is not necessary for our purposes.}

Varieties are closed under (arbitrary) intersections, thus for any variety $\mathcal{V}$ the collection of its subvarieties forms a complete lattice (under inclusion). 
Given an arbitrary class $\mathcal{K}$ of (similar) algebras, we indicate by $\mathsf{V}(\mathcal{K})$ the smallest variety containing $\mathcal{K}$, i.e., the variety \emph{generated} by $\mathcal{K}$. 
A known result states that $\mathsf{V}(\mathcal{K}) = \mathsf{H}\mathsf{S} \mathsf{P}(\mathcal{K})$ (see e.g., \cite[Thm. 3.43]{BergmanLibro}). 
\cadd{For an algebra $\m A$, we write $\mathsf{V}(\m A)$ to denote $\mathsf{V}(\{\m A\}).$} 

Given a homomorphism $f\colon \mathbf{A}\to\mathbf{B}$, let $f[\mathbf{A}]$ be the subalgebra of $\m B$ with universe $f [A]$ and write $\mathbf{A}\leq \mathbf{B} $ to indicate that $ \mathbf{A}$ is a subalgebra of $\mathbf{B}$. 
Then an algebra $ \mathbf{A}$ is a \emph{subdirect product} of a family $ \{\mathbf{B}_i \;:\;i \in I\} $ when $\mathbf{A}\leq  \prod_{i\in I} \mathbf{B}_i $ and for every $i\in I$ the projection map $\pi_{i}\colon \mathbf{A}\to\mathbf{B}_i $ is surjective. 
Similarly, an embedding $f\colon \mathbf{A}\to\prod_{i\in I}\mathbf{B}_i$
is called \emph{subdirect} when $f [\mathbf{A}]\leq  \prod_{i\in I}\mathbf{B}_i $ is a subdirect product. 
An algebra $\m A$ is said to be \emph{subdirectly irreducible} if for every subdirect embedding $f\colon\mathbf{A}\to \prod_{i\in I} \mathbf{B}_i$ with $\{\mathbf{B}_i: i \in I\}$ there exists $i \in I $ such that $\pi_{i}\circ f \colon \mathbf{A}\to \mathbf{B}_i$ is an isomorphism. 

Subdirectly irreducible algebras are particularly relevant in the understanding of varieties. 
More precisely, every algebra $\mathbf{A}$ in a variety $\mathcal{V}$ is isomorphic to a subdirect product of subdirectly irreducible algebras in $\mathcal{V}$. 
Therefore every variety is determined by its subdirectly irreducible members (see \cite[Thm.~3.44, Cor.~3.45]{BergmanLibro}), in particular, for two varieties $\mathcal{V}$ and $\mathcal{W}$, $\mathcal{V}\subseteq\mathcal{W}$ \crev{if and only if} the subdirectly irreducible members of $\mathcal{V}$ are included in those of $\mathcal{W}$.   

Recall that a congruence over an algebra $\mathbf{A}$ (in a class $\mathcal{K}$) is an equivalence relation which preserves all the fundamental operations of $\mathbf{A}$. 
The set of congruences over $\mathbf{A}$ forms a complete lattice, as congruences are closed under arbitrary intersections. 
An algebra $\m A$ being subdirectly irreducible can be checked using the following convenient equivalent characterization related to the lattice $\mathrm{Con}(\mathbf{A})$ of the congruences of $\mathbf{A}$: an algebra $\mathbf{A}$ is subdirectly irreducible if and only if \crev{the identity congruence} $\Delta_{\mathbf{A}}$ is a completely intersection-irreducible element in $\mathrm{Con}(\mathbf{A})$ (see e.g., \cite[Thm.~3.23]{BergmanLibro}). 
If a variety contains constants in its signature, i.e., nullary operations, then it must contain \emph{trivial algebras}, i.e., (isomorphic) algebras with exactly one element. 
\crev{Note that, by definition of subdirectly irreducible algebras ($\Delta_{\mathbf{A}}$ is a complete intersection-irreducible element), trivial algebras are not subdirectly irreducibles.}

\subsection*{Basic structures}
Recall an algebra $\str{S,*}$ is called a \emph{semigroup} if $*$ is an associative operation over $S$. 
A \emph{band} is a semigroup in which the operation $*$ is idempotent, i.e., one satisfying the identity
\begin{equation}
    x * x \approx x .\label{eq:idem}\tag{idempotency}
\end{equation}
On the other hand, a \emph{monoid} is a unital semigroup, namely an expansion of a semigroup with a designated constant stipulated to be a (two-sided) unit for $*$; 
i.e., a structure $\str{M,*,\miden}$ where $\str{M,*}$ is a semigroup with $\miden \in M$ satisfying $\miden * x \approx x\approx x * \miden$. 
By a \defem{unital band} (or simply \defem{uband}) we refer to any idempotent monoid.

As usual, any of these structures will be called \defem{commutative} so long as the underlying semigroup is commutative; i.e., satisfying $x*y\approx y*x$.
We also recall that a semigroup is called \defem{left-regular} if it satisfies the following identity:
\begin{equation}\label{eq:leftreg}\tag{left-regularity}
    x*y*x \approx x* y . 
\end{equation}
It is called \defem{right-regular} if it satisfies the mirrored identity  $x*y*x \approx y* x$.
\begin{remark}\label{fact:RRLR}
    It is easily verified that a band is commutative iff it is both left-regular and right-regular. 
\end{remark}

Note that the operation $*$ is idempotent in any left-regular monoid, and therefore its semigroup reduct $\str{M,*}$ is a left-regular band. 
It is well known that, over any left-regular band $\str{M,*}$, the relation $\leq_*$ defined via
$$ a\leq_* b \iff a * b = b $$
is a partial order. 
Moreover, a simple consequence of left-regularity is that this order is compatible with $*$ from the left, i.e., 
$$a\leq_* b \implies c*a \leq_* c*b;$$
however $\leq_*$ is not generally compatible with the operation $*$ from the right. 
Of course, the mirrored notions also hold for right-regular structures.

\begin{proposition}\label{t:leftregEquiv}
Let $\m M$ be a monoid. Then the following are equivalent:
\begin{enumerate}
    \item $\leq_*$ is a partial order.
    \item $\m M$ is idempotent and $\leq_*$ is antisymmetric. 
    \item $\m M$ is left-regular. 
\end{enumerate} 
\end{proposition}

\begin{proof}
It is easily verified that that reflexivity of $\leq_*$ implies idempotency of $*$, which in turn implies $x*y \leq_* x*y*x$ and $x*y*x\leq_* x*y$. 
\end{proof}

\section{The variety of {\iname}s}\label{sec: variety of iname}

One of the principal motivations in the study of non-classical logics is the investigation of generalizations, or weakenings, of (propositional) \emph{classical logic}, which finds as its semantics the class of \emph{Boolean algebras}.
Recall that an algebra $\m B = \str{B,\vee,\land,\neg, 0,1}$ is a \defem{Boolean algebra} if it is a distributive ortholattice, i.e., $\str{B, \vee, \land,0,1}$ is a bounded distributive lattice and $\neg$ is an orthocomplement, meaning it is an antitone involution satisfying $x\land \neg x = 0$ (dually, $x\vee \neg x = 1$). 
By $\BA$ we denote the variety of Boolean algebras.

As Boolean algebras satisfy the De~Morgan laws, the variety $\BA$ is term equivalent to one in a reduced signature, e.g., $\str{\vee,\neg,0}$ or $\str{\land,\neg,1}$. 
In either case, a Boolean algebra, as viewed in one of these signatures, is simply an instance of a monoid further expanded with an \emph{involution} (satisfying additional identities relating the involution with the semigroup operation).

We will call a \defem{monoid with involution} (or simply \defem{\inv monoid}) any structure $\str{M,*,\nc{},\miden}$ consisting of a monoid reduct $\str{M,*,\miden}$ and unary involutive operation $\nc{}$, i.e., satisfying $\nnc{x} \approx x$. 
Any \inv monoid $\m M = \str{M,*,\nc{},\miden}$ induces another \inv monoid $\m M^{\partial}:=\str{M,\nc{*},\nc{},\nc{\miden}}$ where $x~\nc{*}~y:=\nc{(\nc{x} * \nc{y})}$, called its \defem{De~Morgan dual}. 
It is readily verified that any \inv monoid $\m M$ is isomorphic to its De~Morgan dual $\m M^\partial$ via the map $x\mapsto \nc{x}$, and moreover $(\m M^\partial)^\partial = \m M$. 
In this way, any \inv monoid $\m M$ is term-equivalent to an algebra over a richer signature, namely $\str{M,\jc,\mc,\nc{},0,1}$ where $\m M \cong \str{M,\mc,\nc{},1}^\partial  = \str{M,\jc,\nc{},0}$, where we use the constant $1$ to denote the conjunctive/multiplicative unit, and the constant $0$ to denote the disjunctive/additive unit. 
By definition. $0:=\nc{1}$ and/or $1:=\nc{0}$. 
We will mainly use these constant symbols (in place of $\miden$ and $\nc{\miden}$) when the unit of an \inv monoid is not necessarily an involution fixed-point; i.e., when $\miden \neq \nc{\miden}$.

For a term $t$ in the language $\mathcal{L} = \str{\jc,\mc, \nc{},0,1}$, let $t^\partial$ denote the term in $\str{\mc,\jc,\nc{},1,0}$ obtained by uniformly and simultaneously swapping each instance of a symbol from $\mathcal{L}$ by its dual. 
E.g., if $t=t(x,y) := \nc{(x \jc 1)}\mc y$ then $t^\partial=t^\partial(x,y) =\nc{(x \mc 0 )}\jc y $.
Clearly then for any term $t$ over $\mathcal{L}$, an i-monoid satisfies the identity $t(x_1,\ldots,x_n) \approx \nc{t^\partial(\nc{x_1},\ldots,\nc{x_n})}$. 
The following proposition is readily verified via structural induction.

\begin{proposition}[De~Morgan Dual Equivalence]\label{t:DMduals}
An \inv monoid $\m M$ satisfies some identity if and only if it satisfies the dual identity, i.e., $\m M\models s\approx t$ iff $\m M\models s^\partial\approx t^\partial$ for any terms $s,t$ in the language $\str{\jc,\mc,\nc{},0,1}$.
\end{proposition}

Of course, as Boolean algebras are also instances of lattices, the operations $\land$ and $\lor$ are \emph{idempotent}, i.e., $x\land x \approx x$ and $x\lor x\approx x$ holds. 
This leads us to the following definition for the structures that underlie most of the analysis for this article. 
\begin{definition}\label{def:imonoids}
By an \defem{{\iname}} we refer to a unital band with involution (i.e., an idempotent \inv monoid).
\end{definition}

In general, we may preface the name of any algebraic structure with ``\inv'' to indicate that class of algebras expanded with an involution.

\begin{remark}\label{fact:noinverses}
    In any {\iname}, due to the interplay of idempotency with associativity, no element other than the unit $\miden$ can have a (left or right) \emph{inverse}, i.e., $x*y = \miden$ implies $x=y=\miden$. 
    Indeed, if $x*y = \miden$ then $x = x*\miden = x*(x*y) = (x*x)*y=x*y = \miden$, and similarly, $y =x*y*y = \miden$.
\end{remark}

Aside from Boolean algebras, {\iname}s generalize a number of familiar structures as a result of the De~Morgan laws. In particular, the following are all (term-equivalent to) subvarieties of commutative {\iname}s: bounded involutive lattices, and hence also ortholattices, modular ortholattices, orthomodular lattices; De~Morgan algebras; Kleene algebras; and involutive bisemilattices. Other relevant examples of {\iname}s  appear as (the multiplicative and involution) reducts of richer structures, e.g., involutive idempotent residuated lattices, Sugihara monoids, as well the non-commutative variants of Sugihara monoids.

The simplest (nontrivial) examples of {\iname}s are those over a two-element set. 
There are exactly two such algebras up to isomorphism---they share the same operation table for $*$ (i.e., isomorphic monoidal reducts) which is uniquely determined when stipulating it be idempotent and contain a unit constant $\miden$, which so happens to satisfy associativity---but there are exactly two different possible involutions; the ``trivial'' identity map and the one swapping the two elements. 
The former results in the (expansion by a unary identity operation of the) two-element bounded semi-lattice, here denoted by $\m C_2$, and the latter is the two-element Boolean algebra $\m 2$. 
We will denote the respective varieties of {\iname}s they generate by $\mathsf{SL}$ and $\BA$; as these are the varieties of (bounded) \emph{semi-lattices} and \emph{Boolean algebras}, respectively. 
It is easily verified that they are atoms in the lattice of subvarieties for {\iname}s; indeed, $\m 2$ is a subalgebra of any algebra in which the unit is not an involution fixed-point; while, on the other hand, if a non-trivial {\iname} $\m M$ has its unit being a fixed-point, $\m C_2$ is isomorphic to a quotient via \cref{fact:noinverses} (as the relation $(M\setminus\{\miden\}\times M\setminus\{\miden\})\cup\{(\miden,\miden) \}$ is a congruence over $\m M$).

\subsection{The 3-element {\iname}s}
There are exactly three ubands (i.e., idempotent monoids) of cardinality $3$ up to isomorphism, see Figure~\ref{f:3elemIdM} for the operation tables. Indeed, suppose $\str{\{\miden,a,b\},*,\miden}$ is an idempotent monoid. 
As a consequence of $*$ being an idempotent operation with a (two-sided) unit $\miden$, the only values that are free to be determined are $a*b$ and $b*a$, and moreover, they must differ from $\miden$ (see \cref{fact:noinverses}). 
So there are exactly four possibilities. 
Note that the two cases in which $a*b$ and $b*a$ coincide result in isomorphic structures (simply renaming $a$ by $b$ and vice versa), and are indicated by the (commutative) operation $*_\mathsf{c}$ below where $a*b=b=b*a$. 
The other two possibilities, when these values differ, are given by $*_{\ell}$ and $*_{r}$. That all three operations are associative is easily verified by the reader. 
That no two result in isomorphic monoids is witnessed by the fact that $*_\mathsf{c}$ is the only commutative operation, while $*_{\ell}$ is left-regular and $*_{r}$ is right-regular, and thus must be non-isomorphic as neither are commutative (cf. \cref{fact:RRLR}). 
\begin{figure}[H]
\centering
$
\begin{array}{c c c}
\begin{array}{c|ccc}
*_\mathsf{c} & \miden & a  & b  \\\hline
\miden & \miden  & a  & b  \\
a  & a  & a  & b  \\
b  & b  & b  & b 
\end{array}
\quad&\quad
\begin{array}{c|ccc}
*_{\ell} & \miden  & a  & b  \\\hline
\miden  & \miden  & a  & b  \\
a  & a  & a  & a  \\
b  & b  & b  & b 
\end{array}
\quad&\quad
\begin{array}{c|ccc}
*_{r}& \miden  & a  & b  \\\hline
\miden  & \miden  & a  & b  \\
a  & a  & a  & b  \\
b  & b  & a  & b  \\
\end{array}
\end{array}
$
\caption{The four idempotent monoids of cardinality 3, up to isomorphism.}\label{f:3elemIdM}
\end{figure}

There are exactly 4 possible involutions over a three element set; one being the identity map and the other three determined by which element is the (unique) involution fixed-point. 
We label these possibilities over the set $\{\miden,a,b\}$ below:
\[
\begin{array}{cccc}
    \begin{array}{c|c}
        &\prime_\mathrm{id}   \\ \hline
         \miden &\miden\\
         a & a\\
         b & b
    \end{array}
    \quad&\quad
    \begin{array}{c|c}
        &\prime_\miden   \\\hline
         \miden &\miden\\
         a & b\\
         b & a
    \end{array}
    \quad&\quad
    \begin{array}{c|c}
        &\prime_a   \\\hline
         \miden &b\\
         a & a\\
         b & \miden
    \end{array}
    \quad&\quad
    \begin{array}{c|c}
       & \prime_b   \\\hline
         \miden &a\\
         a & \miden\\
         b & b
    \end{array}
\end{array}
\]
Consequently, these result in 12 possible combinations for 3-element {\iname}s when paired with the operations from Figure~\ref{f:3elemIdM}. 
In \cref{t:3elem} below, we verify that this reduces to 10 non-isomorphic models in total.

Note that the algebras with the identity involution $\prime_\mathsf{id}$ are simply the idempotent monoids (expanded by the identity function) from Figure~\ref{f:3elemIdM}; we will denote these {\iname}s by $\m C_3$, $\m L_3$ and $\m R_3$, respectively. 
Note that $\m C_3$ is simply the $3$-element bounded semilattice (chain).

For the structures with involution $\prime_{\miden}$ fixing only the unit $\miden$, let us call them $\m{C}^\mathsf{s}_3$, $\m{L}^\mathsf{s}_3$, $\m{R}^\mathsf{s}_3$, denoting the {\iname}s respectively corresponding to $\str{*_\mathsf{c},\prime_{\miden}}$, $\str{*_{\ell},\prime_{\miden}}$, and $\str{*_{r},\prime_{\miden}}$. 
The algebra $\m{C}^\mathsf{s}_3$ is isomorphic to the {\iname}-reduct of the 3-element (odd) Sugihara monoid $\m S_3$. 
The monoidal operation for this algebra is often easily described using natural order $\leq$ of the integers over the set $\{ -1,0,1\}$ (here, $\miden\mapsto 0$, $a\mapsto 1$, and $b\mapsto -1$): the involution $\nc{}$ is $-$ and a product $x\mc y$ is dependent upon the size of the absolute values $|x|, |y|$ for the arguments $x,y$: the product results in (i) the value of whichever argument has the larger absolute when they differ, or else (ii) the value of the least argument when they have the same absolute value (i.e., the minimum ``breaks ties''); in this way $0$ is the unit for $\mc$.  
The operations for the algebras $\m{L}^\mathsf{s}_3$ and $\m{R}^\mathsf{s}_3$ can be described similarly:
$\nc{}$ coincides with $-$, and $\mc$ again follows the same rule (i) but, instead of (ii), ties are broken by always taking the value of the left-most (resp., right-most) argument in $\m{L}^\mathsf{s}_3$ (resp., $\m{R}^\mathsf{s}_3$); which in either case still has $0$ the unit for $\mc$. 
\begin{remark}\label{rem:+=*}
    By computing the tables for $\jc$ (the De~Morgan dual of $\mc$), the reader may readily verify that the algebras $\m{L}^\mathsf{s}_3$ and $\m{R}^\mathsf{s}_3$ satisfy the curious identity $x\mc y\approx x\jc y$, while this identity fails in $\m{C}^\mathsf{s}_3$ (indeed 
    $a *^{\prime}_\mathsf{c} b :=
    \nc{(\nc{a}*_\mathsf{c}\nc{b})} = 
    \nc{(b*_\mathsf{c} a)} = 
    \nc{b} = a
    \neq b = a*_\mathsf{c} b$). 
    We point out that any {\iname} satisfying $x\mc y\approx x\jc y$ also satisfies $\miden \approx\nc{\miden}$; indeed 
    $\nc{\miden} \approx \nc{\miden}* \miden \approx \nc{\miden}*^\prime \miden:=\nc{(\nnc{\miden}*\nc{\miden})}\approx\nc{(\miden*\nc\miden)}\approx \nc{(\nc{\miden})}=\nnc{\miden}\approx\miden$.
\end{remark}

The remaining 4 algebras are each isomorphic to one in Figure~\ref{f:3elem} below.
The {\iname} corresponding to $\str{*_\mathsf{c},\prime_a}$ is isomorphic with $\m{SK}$ ($a\mapsto\ee$), while $\str{*_\mathsf{c},\prime_b}$ is isomorphic with $\m{WK}$ ($b\mapsto\ee$), \crev{both introduced in Figure~\ref{f:3elem} below.} 
On the other hand, it is easily checked that $\str{*_{\ell},\prime_{a}}$ and $\str{*_{\ell},\prime_{b}}$ result in isomorphic structures (taking $a\mapsto b$ and simply observing that elements $a,b$ are both absorbing from the left); these are isomorphic to the 3-element McCarthy algebra $\McA$ (cf. Figure~\ref{f:3elem}). 
Similarly,  $\str{*_{r},\prime_{a}}$ and $\str{*_{r},\prime_{b}}$ result in isomorphic structures and coincide with $\McA^\mathrm{op}$, the mirror of $\McA$ (cf. \Cref{f:3elem}). 
We will explore these algebras in more detail in the following \Cref{sec: subclassical}.

\begin{proposition}\label{t:3elem}
    There are exactly ten {\iname}s of cardinality 3 up to isomorphism (see \Cref{f:ten-3elem}). 
\end{proposition}
\begin{figure}[h]
\centering
$
\begin{array}{r|c|c|c|}
    &\mbox{$x\approx \nc{x}$} & \mbox{$\miden \approx \nc{\miden}$} & \mbox{$\miden\neq \nc{\miden}$}   \\ \hline 
    \mbox{commutative}& \m{C}_3 & \m{C}^\mathsf{s}_3 & \begin{array}{c}
          \m{WK} \\ 
          \m{SK}
    \end{array}\\ \hline
    \phantom{\begin{array}{c}|\\|\end{array}}\mbox{left-regular} & \m{L}_3 & \m{L}^\mathsf{s}_3 & \phantom{{}^\mathsf{op}}\McA\phantom{{}^\mathsf{op}} \\ \hline
    \phantom{\begin{array}{c}|\\|\end{array}}\mbox{right-regular} & \m{R}_3 & \m{R}^\mathsf{s}_3 & \phantom{{}^\mathsf{op}}\McA^\mathsf{op}\\ \hline
\end{array}
$
\caption{The ten non-isomorphic {\iname}s of cardinality 3.}\label{f:ten-3elem}
\end{figure}

\begin{proof}
    That no two algebras from different rows can be isomorphic follows from \cref{fact:RRLR}. That no two algebras from different columns is immediate from the involutions. 
    It remains to verify the (well-known) fact that $\m{WK}$ and $\m{SK}$ are non-isomorphic, which is readily established by observing that the identity $x\mc 0\approx 0 $ (recalling $0:=\nc{1}$ with constant $1$ the unit for $\mc$) holds in $\m{SK}$ but not in $\m{WK}$.
\end{proof}

We will revisit the corresponding varieties generated by each of these algebras, their relationship to each other and, more generally, to the lattice of subvarieties of {\iname}s in \Cref{sec: subvarieties}.

\subsection{Subclassical {\iname}s}\label{sec: subclassical}

The identity $1\approx 0$ (i.e., $\miden\approx \nc{\miden}$) and the Boolean algebra $\m 2$ form a \emph{splitting pair} in lattice of subvarieties for {\iname}s. 
That is, given any {\iname} $\m A$, either $\m A$ satisfies the identity $1\approx 0$ (known as the \emph{splitting equation}), or else $\m 2$ is a subalgebra of $\m A$. 
We call an {\iname} a \defem{subclassical algebra}, or simply \emph{subclassical}, if it contains a copy of the Boolean algebra $\m 2$ as a subalgebra (adopting the same terminology of \cite{MicheleSubclassical}); or equivalently, one in which the unit $\miden$ is not a involution fixed-point (i.e., $\miden\neq \nc{\miden}$). 

The following is an immediate corollary to \cref{t:3elem}.
\begin{corollary}\label{t:3elemAlg}
    There are exactly four subclassical {\iname}s of cardinality 3 up to isomorphism (see \Cref{f:3elem}).
\begin{figure}[H]
\centering
$
\begin{array}{c | c c c c}
\begin{array}{c|c}
        &\nc{}   \\ \hline
         1&0\\
         0&1\\
         \ee&\ee
    \end{array}
\quad&\quad   
\begin{array}{c|ccc}
\land_{\mathsf{wk}} & 1 & 0 & \ee \\\hline
1 & 1 & 0 & \ee \\
0 & 0 & 0 & \ee \\
\ee & \ee & \ee & \ee \\
\end{array}
\quad&\quad
\begin{array}{c|ccc}
\land_{\mathsf{sk}} & 1 & 0 & \ee \\\hline
1 & 1 & 0 & \ee \\
0 & 0 & 0 & 0 \\
\ee & \ee & 0 & \ee \\
\end{array}
\quad&\quad
\begin{array}{c|ccc}
{}\mc_{\mathsf{m}} & 1 & 0 & \ee \\\hline
1 & 1 & 0 & \ee \\
0 & 0 & 0 & 0 \\
\ee & \ee & \ee & \ee \\
\end{array}
\quad&\quad
\begin{array}{c|ccc}
{}\mc_{\mathsf{m}}^\mathsf{op}& 1 & 0 & \ee \\\hline
1 & 1 & 0 & \ee \\
0 & 0 & 0 & \ee \\
\ee & \ee & 0 & \ee \\
\end{array}
\\ & \m{WK}=\str{X,\land_{\mathsf{wk}},\nc{},1} & \m{SK}=\str{X,\land_{\mathsf{sk}},\nc{},1} & \McA=\str{X,\mc_{\mathsf{m}},\nc{},1} & \McA^\mathsf{op}=\str{X,\mc_{\mathsf{m}}^\mathsf{op},\nc{},1} 
\end{array}
$
\caption{The four subclassical {\iname}s of cardinality 3 over the set $X=\{1,0,\ee \}$ and signature $\str{\mc,\nc{},1}$.
}\label{f:3elem}
\end{figure}
\end{corollary}
\begin{remark}
Differently from the approach in \cite{MicheleSubclassical} (which assumes types do not contain constant symbols), the algebra $\m{C}^\mathsf{s}_3$ (the {\iname} reduct of the 3-element Sugihara monoid) is not a subclassical {\iname}. 
Indeed, $\{a,b\}$ should be the universe of a Boolean subalgebra, which is not, due to the fact that $a$ is not the unit of the monoidal $*_\mathsf{c}$ (corresponding to meet) operation ($\miden$ is the unit).    
\end{remark}

The first operations in Figure \ref{f:3elem} are the conjunctions for the 3-element involutive bisemilattice ($\m{WK}$) and Kleene lattice ($\m{SK}$), respectively, defining weak and strong Kleene logics, respectively. The third operation is the conjunction of the 3-element \emph{McCarthy algebra} $\McA$. 
Note that the last algebra has the \emph{opposite} monoid relation ($x*^{\mathsf{op}}y:= y * x$), so we will denote it by $\McA^\mathsf{op}$.
Obviously $\m{WK}$ and $\m{SK}$ are left/right-regular as they are commutative and idempotent. 
It is not difficult to check that, while not commutative, $\McA$ is left-regular (and therefore $\McA^\mathsf{op}$ is right-regular): 
they are therefore all partially ordered (cf. \cref{t:leftregEquiv}). 
Below are the Hasse Diagrams for their dual operation:
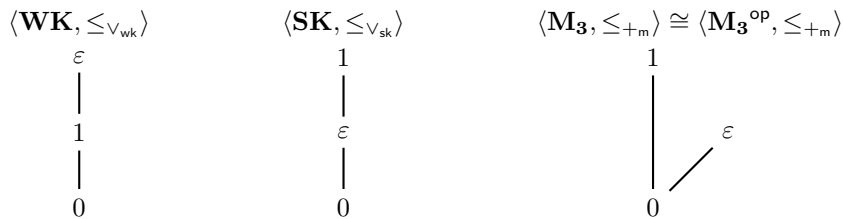
\begin{figure}[ht]
    \centering
    \begin{tikzpicture}[scale = 1.2, every node/.style={inner sep=.5pt}, every label/.style={black}]
        \node[label = {[xshift=-.5em,yshift = -.5em]\scriptsize $0$}] (0) at (0,0) {\scriptsize$\bullet$};
        \node[label = {[xshift=-.5em,yshift = -.5em]\scriptsize $1$}] (a) at (0,.5) {\scriptsize$\bullet$};
        \node[label = {[xshift=-.5em,yshift = -.5em]\scriptsize $\ee$}] (b) at (0,1) {\scriptsize$\bullet$};
        \draw (0)--(a)--(b);
    \end{tikzpicture}
    \qquad\qquad\qquad
    \begin{tikzpicture}[scale = 1.2, every node/.style={inner sep=.5pt}, every label/.style={black}]
        \node[label = {[xshift=-.5em,yshift = -.5em]\scriptsize $0$}] (0) at (0,0) {\scriptsize$\bullet$};
        \node[label = {[xshift=-.5em,yshift = -.5em]\scriptsize $\ee$}] (a) at (0,.5) {\scriptsize$\bullet$};
        \node[label = {[xshift=-.5em,yshift = -.5em]\scriptsize $1$}] (b) at (0,1) {\scriptsize$\bullet$};
        \draw (0)--(a)--(b);
    \end{tikzpicture}
    \qquad\qquad\qquad
    \begin{tikzpicture}[scale = 1.2, every node/.style={inner sep=.5pt}, every label/.style={black}]
        \node[label = {[xshift=-.5em,yshift = -.5em]\scriptsize $0$}] (0) at (0,0) {\scriptsize$\bullet$};
        \node[label = {[xshift=-.5em,yshift = -.5em]\scriptsize $1$}] (a) at (-.5,1) {\scriptsize$\bullet$};
        \node[label = {[xshift=-.5em,yshift = -.5em]\scriptsize $\ee$}] (b) at (.5,1) {\scriptsize$\bullet$};
        \draw (0)--(a);
        \draw (0) -- (b);
    \end{tikzpicture}
    \caption{The Hasse diagrams (left-right) for posets  $\str{\m{WK},\leq_{\vee_{\mathsf{wk}}}}$, $\str{\m{SK},\leq_{\vee_{\mathsf{sk}}}}$, and $\str{\McA,\leq_{\jc_{\mathsf{m}}}}$ ($\cong \str{\McA^\mathsf{op},\leq_{\jc_{\mathsf{m}}}}$).}
    \label{fig:WK/SK/MKposets}
\end{figure}

Recall that the algebras $\mathbf{WK}$ and $\mathbf{SK}$ have a logical meaning, namely they define weak and strong Kleene logics, respectively (see \cite{Bonziobook} for more details). 
Moreover $\mathbf{WK}$ and $\mathbf{SK}$ generate the variety $\mathsf{IBSL}$ of \emph{involutive bisemilattices} \cite{Bonzio16SL} and the variety $\mathsf{KA}$ of \emph{Kleene algebras} \cite{kalman58involution}, respectively. 
Involutive bisemilattices consist of the regularization of (the variety of) Boolean algebras, i.e., they satisfy all and only the regular identities holding in Boolean algebras, where an identity $\varphi\approx \psi$ is regular if the variables actually occurring in both the terms $\varphi$ and $\psi$ are the same. 
Despite being usually introduced in the richer language $\str{\land,\lor,\lnot,0,1}$, in our language, an involutive bisemilattice is a \emph{commutative {\iname}} satisfying $x\mc y\approx x\mc (\nc{x}\jc y)$ (see \cite{Bonzio16SL,Bonziobook}). 
A slightly different (equivalent) axiomatization can be found in \cite[Example 1]{Plonka1984nullary}.

Kleene algebras, sometimes also referred to as Kleene \emph{logic} algebras to distinguish from Kleene-star algebras \cite{Kleene56, KozenKleenetest}, are defined in the same language as involutive bisemilattices and consist of a De~Morgan algebra (i.e., a bounded distributive lattice equipped with an involutive negation $\neg$ satisfying the De~Morgan identities) satisfying the further identity $\mathsf{k}:x\land\neg x\leq y\vee\neg y$, which we will refer to as the \emph{Kleene axiom}, whose intuitive reading is that any contradiction is below any tautology.
As it turns out, Kleene algebras are term-equivalent to commutative {\iname}s satisfying $\mathsf{k}$ as well as distributivity and boundedness (see \Cref{sec:WD,sec:WA}).

We note that the class of all subclassical {\iname}s forms a quasivariety; i.e., it is the class of all {\iname}s satisfying the quasi-identity $\miden = \nc{\miden}\Rightarrow x = y$, and therefore need not be closed under homomorphic images.
\begin{definition}\label{def: subclassical variety}
A variety $\mathcal{V}$ of {\iname}s will be called a \defem{subclassical variety} if each \crev{non-trivial} member of $\mathcal{V}$ is subclassical.    
\end{definition}

    \begin{remark}
    As mentioned, a variety $\mathcal{V} $ is subclassical if and only if it satisfies the quasi-identity $\miden\approx\nc{\miden}\;\Rightarrow\; x\approx y$. As it turns out, Kleene algebras form a subclassical variety, while involutive bisemilattices do not ($\m C_2$ is a quotient of $\m{WK}$). 
    Later on (\cref{t:subclass}), we will show that the variety $\MC$ of McCarthy algebras is subclassical.
\end{remark}

\newpage

\section{Identities satisfied in McCarthy algebras}\label{sec:findingidens}
\cadd{A finite basis for McCarthy algebras was established by Guzm{\'a}n and Squier in \cite{Guzman-Squier}, where they are called $C$-algebras (with constants $T$ and $F$). Independently}
in \cite{Kow96}, Konikowska presents a list of identities claimed to hold for the algebra $\McA$ (see \Cref{f:KonAx} in \Cref{sec:MKdef}), but leaves open the question of whether they form an equational basis for the variety $\MC$ of McCarthy algebras. 
Instead of presenting them here, we begin an analysis with the motivation of relating $\McA$ to more familiar algebraic properties and structures in order to ``tease-out'' an equational basis for $\MC$. 
In doing so, we prove some general facts, interesting in their own right, for {\iname}s. 
Ultimately, in \Cref{sec:axiomMC}, we show Konikowska's postulates do indeed provide an equational basis, but also find reduced and equivalent alternatives as a byproduct of the analysis carried on in this section.
\subsection{Weak distributive laws}\label{sec:WD}
Let $\m S$ be an algebra that at contains two semigroups $\str{S,\jc}$ and $\str{S,\mc}$ as reducts. 
The algebra $\m S$ is called a \emph{semiring} if it satisfies both the left and right distributive laws: $x(y\jc z) \approx xy \jc xz$ and $(x\jc y) z \approx xz \jc yz$.
Both $\m{SK}$ and $\m{WK}$ are semirings, but $\McA$ fails the right-distributive law, e.g., $(1\jc \varepsilon)0 = 0 \neq \varepsilon = 0 \jc \varepsilon = (1\mc 0) \jc (\varepsilon \mc 0)$. 
However, the \hyperref[eq:leftdist]{left-distributive}  law is satisfied.
\begin{proposition}\label{t:leftsemiring}
    The algebra $\McA$ is left-distributive, i.e., it satisfies 
    \begin{equation}\label{eq:leftdist}\tag{left-distributivity}
    x(y\jc z)\approx xy \jc xz 
    \qquad\text{dually,}\qquad
    x\jc yz \approx (x\jc y)(x\jc z)
    \end{equation}
\end{proposition}
\begin{proof}
    In $\McA$, notice that the claim holds for $x=1$ since $1$ is a multiplicative unit, and it holds for $x= 0$ and $x=\varepsilon$ since they are both multiplicatively absorbing from the left, i.e., $0a=0$ and $\varepsilon a = \varepsilon$ for any $a\in \McSet$.
\end{proof}
\begin{remark}\label{t:LDqe}
    The following quasi-identity holds in any algebra with a pair of binary operations $\str{\mc,\jc}$ satisfying \ref{eq:leftdist} with a right-unit $1$ for $\mc$: $y\jc 1 = z\jc 1 \;\Rightarrow\; xy \jc x = xz \jc x $.
\end{remark}

Now, while the algebra $\McA$ does not satisfy right-distributivity in general, it does satisfy the following weakened instance of it.

\begin{proposition}\label{t:M3orthodistribitivity}
    The algebra $\McA$ satisfies the following identities:
    \begin{equation}\label{eq:orthodist}\tag{right-ortho\allowbreak distributivity}
        (x\jc \nc{x}) y \approx xy \jc \nc{x}y 
       \qquad\text{dually,}\qquad
       x \nc{x}\jc y \approx (x\jc y) (\nc{x}\jc y)
    \end{equation}
\end{proposition}
\begin{proof}
In $\McA$, it is readily verified that $x\jc\nc{x} = 1$ when $x\in \{0,1\}$, in either case the identity is valid since $1$ is a multiplicative unit. Other the other hand, when $x=\varepsilon$ the identity is valid since $\varepsilon=\varepsilon\jc \nc{\varepsilon}$ and $\varepsilon$ is absorbing from the left.
\end{proof}

Of course, there is the ``left'' variant of the identity above: $x(y\jc \nc{y})\approx xy \jc x\nc{y}$, which we will refer to as \emph{left-orthodistributivity}. 
An algebra satisfying both left- and right-orthodistributivity will simply be called \defem{orthodistributive}. 
Clearly then, $\McA$ is orthodistributive as it satisfies left-distributivity.
The following fact will be useful later on in \Cref{sec: subvarieties} for the algebras $\m C^\mathsf{s}_3$, $\m L^\mathsf{s}_3$, and $\m R^\mathsf{s}_3$ and can be verified by checking their respective operation tables.

\begin{proposition}\label{t:sIDENS}
    The algebras $\m C^\mathsf{s}_3$, $\m L^\mathsf{s}_3$, and $\m R^\mathsf{s}_3$ are orthodistributive but not distributive.
    However, the algebra $\m L^\mathsf{s}_3$ is left-distributive, $\m R^\mathsf{s}_3$ is right-distributive, and both algebras satisfy the identity $x\mc y\approx x\jc y$.
\end{proposition}

As it will be notationally convenient and evocative for what follows, let us define the following unary term-operations in the language of {\iname}s: 
\begin{align*}
    \Ic{x}&:=x\jc 1 &&\Oc{x}:= x\mc 0 \tag{the local constants}\\
    \tc{x}&:=x\jc \nc{x} &&\fc{x}:= x\mc \nc{x} \tag{the local extrema}
\end{align*}
\begin{remark}
    Of course, there are also the unary operations corresponding to the ``mirrored'' versions of the above. 
    On the one hand, due to $\nc{}$ being an involution, every {\iname} satisfies $\nc{x}\jc x \approx \tc{\nc{x}}$; i.e., the local extrema are inter-definable with the mirrored version. 
    However, this is not the case for $\Ic{x}/\Oc{x}$.
    As we are most interested in identities satisfied by $\McA$, 
    \crev{and $\McA$ satisfies $0\mc x \approx 0$ and $1\jc x \approx x$ (see \Cref{sec:WA}),}
    we will not consider them here. 
    We note that, when considering $\McA^\mathsf{op}$, such mirrored versions, and any identity invoking $\Ic{x}/\Oc{x}$, are to be used in place of the above. For this reason, guided by context, we refrain from the convention of prefacing such terms and identities with the words \emph{left} or \emph{right}.
\end{remark}

\begin{remark}
    Note that the orthodistributivity equations may be rewritten as $\tc{x}y \approx xy \jc \nc{x}y$ for the right case, and $x\tc{y}\approx xy \jc x\nc{y}$ for the left.
\end{remark}

\begin{lemma}\label{t:ICabs}
    Any \hyperref[eq:leftdist]{left-distributive}  {\iname} $\m M=\str{M,\jc,\mc,\nc{},0,1}$ satisfies the following quasi-identities:
    $$
    \Ic{y} = \Ic{z} \;\Rightarrow\; xy \jc x = xz
    \qquad\text{and}\qquad 
    \Oc{y} = \Oc{z} \;\Rightarrow\; (x\jc y)x = (x\jc z)x
    $$
    Moreover, the following identities also hold:
    \begin{equation}\label{eq:localunits}\tag{local units}
        \Ic{x}\mc x \approx x \approx x\mc \Ic{x} \qquad\text{dually,}\qquad \Oc{x}\jc x \approx x \approx x\jc \Oc{x}.
    \end{equation}
\end{lemma}
\begin{proof}
    \crev{Note that the following quasi-identity holds in any algebra with a pair of binary operations $\str{\mc,\jc}$ satisfying \ref{eq:leftdist} with a right-unit $1$ for $\mc$: $y\jc 1 = z\jc 1 \;\Rightarrow\; xy \jc x = xz \jc x $.  
    Since the reducts $\str{M,\jc,\mc,1}$ and $\str{M,\mc,\jc,0}$ are \hyperref[eq:leftdist]{left-distributive}  with units, respectively, the first claims are immediate. }
    For the last claim, fix $a\in M$. 
    The identity $x\mc \Ic{x}\approx x$ is nearly immediate via idempotency: $a(a\jc 1) = aa \jc a1 = a \jc a = a$.  
    The identity $\Ic{x}\mc x \approx x$ is immediate from the right-most quasi-identity and idempotency: by setting $y=1$ and $z=0$, we have $\Oc{1} = 1\mc 0 = 0 = 0\mc 0 =\Oc{0}$, and therefore $\Ic{a}\mc a := (a\jc 1)a = (a\jc 0)a = aa = a$.
\end{proof}

Recall that a bounded involutive lattice is called \emph{orthocomplemented} if $x\land \nc{x}\approx 0$ and $x\vee \nc{x}\approx 1$ hold. 
However, none of the algebras in \Cref{f:3elem} satisfy these identities, but some do satisfy \emph{local} versions of them.  
We will call an {\iname} \defem{locally-complemented} if the identity $\Ic{x}\approx x\jc \nc{x}$ or, equivalently, $\Oc{x}\approx x\mc \nc{x}$ is satisfied. Using our abbreviations, this is equivalently written below:
\begin{equation}\label{eq:localcomp}\tag{locally-complemented}
    \Ic{x}\approx \tc{x}\qquad\text{dually,}\qquad \Oc{x}\approx \fc{x}.
\end{equation}
The following is easily verified given the operation tables in Figure~\ref{f:3elem}.
\begin{proposition}\label{t:M3locallycomp}
    The algebras $\McA$ and $\m{WK}$ are locally-complemented.
\end{proposition}
\noindent
On the other hand, the algebra $\m{SK}$ is not locally-complemented. 
Note, however, that the Kleene axiom is equivalently written $\fc{x}\leq \tc{y}$.

We will say an {\iname} is \defem{left-divisible} it is satisfies the follow identity:\footnote{In the context of residuated structures, a residuated lattice is called \emph{divisible} if it satisfies $x\land y \approx x\cdot (x\to y)$ and $x\land y \approx (y\leftarrow x)\cdot x$. 
In Boolean algebras, the left-most identity is equivalent to our definition, as $\to$ is material implication and $\cdot$ and $\land$ coincide.} 
\begin{equation}\label{eq:divis}\tag{left-divisibility}
        xy \approx x(\nc{x}\jc y) 
        \qquad\text{dually,}\qquad 
        x\jc y \approx x\jc \nc{x}y
\end{equation}
\begin{lemma}\label{t:divisibility}
    Any \hyperref[eq:divis]{left-divisible} {\iname} is also \ref{eq:localcomp}. 
    Moreover, if $\m A$ is a \hyperref[eq:leftdist]{left-distributive}  {\iname}, then $\m A$ is \ref{eq:localcomp} iff it is \hyperref[eq:divis]{left-divisible}. 
\end{lemma}
\begin{proof}
    For the first claim, observe from \ref{eq:divis} we obtain
    $ \Oc{x}:=x0 \approx x(\nc{x} \jc 0) \approx x\mc \nc{x}.$ 
    For the second claim, we need only verify the forward implication.
    Fix $a,b\in A$. 
    We observe:
    \begin{align*}
        a(\nc{a}\jc b)
            &= a\nc{a} \jc ab && \text{\eqref{eq:leftdist}}\\
            &= a0 \jc ab && \text{\eqref{eq:localcomp}}\\
            &= a(0\jc b)=ab&& \text{\ref{eq:leftdist} \& $0$ is unital for $\jc$}\qedhere 
    \end{align*}
\end{proof}

\noindent It is easy to see that the identities for \ref{eq:orthodist} and \ref{eq:localcomp} can be  combined.
\begin{lemma}\label{t:localdecomp}
    An {\iname} satisfies \ref{eq:orthodist} and \ref{eq:localcomp} iff it satisfies:
    \begin{equation}\label{eq:localdecomp}\tag{left-decomposition}
    \Ic{x}\mc y\approx xy\jc \nc{x}y 
    \qquad \text{dually,}\qquad
    \Oc{x}\jc y \approx (x\jc y)(\nc{x}\jc y)
    \end{equation}
\end{lemma}
\begin{proof}
    The forward direction is immediate by replacing $x \jc \nc{x}$ by $\Ic{x}$ in the left-hand side of \ref{eq:orthodist} using the \ref{eq:localcomp} identity. 
    So suppose $\m A$ satisfies \ref{eq:localdecomp}. 
    Instantiating $y$ by $1$ yields the \ref{eq:localcomp} identity, and thus replacing $\Ic{x}$ by $x \jc \nc{x}$ in the identity for \ref{eq:localdecomp} yields \ref{eq:orthodist}.
\end{proof}
The following is immediate from the above by \cref{t:M3orthodistribitivity} and \cref{t:M3locallycomp}.
\begin{corollary}\label{t:M3localdecomp}
    The algebra $\McA$ satisfies \ref{eq:localdecomp}.
\end{corollary}
\cadd{
\begin{lemma}\label{t: LR2}
    An {\iname} satisfying \ref{eq:leftdist} and \ref{eq:localdecomp} is left-regular and satisfies: 
    \begin{equation}\label{eq:wk_terms}
        \Ic{x}yx\approx xy \jc yx
        \qquad
        \text{dually,}
        \qquad
        \Oc{x}\jc y \jc x \approx (x\jc y)(y\jc x)
    \end{equation}
\end{lemma}
\begin{proof}
    Let $\m A$ be such an algebra. 
    By \cref{t:localdecomp}, $\m A$ is  \ref{eq:localcomp}, and thus also \ref{eq:divis} by \cref{t:divisibility}. 
    Observe that \ref{eq:leftdist} and being \ref{eq:localcomp} yield $x\jc y\approx x\jc y\nc{x}$ since $a\jc b = a\jc b1 = (a\jc b)(a\jc 1) =(a\jc b)(a\jc \nc{a}) = a\jc b\nc{a}$. 
    Using this identity and \ref{eq:divis}, we find that $\m A$ satisfies \ref{eq:leftreg} since 
    $xy 
    \approx x(\nc{x}\jc y)
    \approx x(\nc{x}\jc yx)
    \approx x\mc yx
    $.
    Lastly, observe for $a,b\in A$,
    \begin{align*}
        \Ic{a}ba&=aba \jc \nc{a}ba && \eqref{eq:localdecomp}\\
        &= aba \jc \nc{a}ba \jc aba &&\eqref{eq:leftreg}\\
        &=aba \jc \Ic{\nc{a}}ba && \eqref{eq:localdecomp}\\
        &=ab \jc \Ic{\nc{a}}ba &&\eqref{eq:leftreg}\\
        &= ab \jc \nc{(ab)}\Ic{\nc{a}}ba
        := ab \jc (\nc{a}\jc\nc{b})(\nc{a}\jc 1)ba &&\eqref{eq:divis}\\
        &= ab \jc (\nc a \jc \nc{b}1)ab = ab \jc (\nc a \jc \nc{b})ab &&\eqref{eq:leftdist}\\
        &= ab \jc ba &&\eqref{eq:divis}. \qedhere 
    \end{align*}
\end{proof}
}
\cadd{Finally, the following distributive-like identity is observed in both \cite{Kow96} and \cite{Guzman-Squier} to hold in $\McA$:
\begin{align*}
    (x\jc y) z& \approx xz \jc \nc{x}yz
     &\mbox{dually, }&& xy \jc z& \approx (x\jc z) \mc (\nc{x} \jc y \jc z)
     \label{eq:paradist}\tag{right-para\allowbreak distributivity}
\end{align*}
The following is immediate by taking $y\mapsto \nc{x}$ and $z\mapsto 1$.
\begin{lemma}\label{t:paradist}
    Right-paradistributivity entails \ref{eq:divis} and \ref{eq:orthodist} in {\iname}s.
\end{lemma}
}

\subsection{Weak commutative laws}\label{sec:WC}
In stark contrast to $\m{SK}$ and $\m{WK}$, the most obvious property that fails in $\McA$ is that of commutativity. 
However, the algebra $\McA$ does satisfy commutativity for the local-constants.
\begin{proposition}\label{t:M3localunitcom}
    The algebra $\McA$ satisfies the following identity:
    \begin{equation}\label{eq:comlocalunits}\tag{local-unit commutativity}
        \Ic{x}\mc \Ic{y} \approx\Ic{y}\mc \Ic{x}
        \qquad \text{dually,}\qquad
        \Oc{x}\jc \Oc{y} \approx\Oc{y}\jc \Oc{x}
    \end{equation}
\end{proposition}
\begin{proof}
    It is readily verified in $\McA$ that $1=\Ic{1}=\Ic{0}$ and $\varepsilon = \Ic{\varepsilon}$. 
    For the non-redundant cases, the claim follows since $1$ is the multiplicative unit.
\end{proof}

As we will see below, left-distributivity in conjunction with \ref{eq:comlocalunits} is strong enough to yield a partially-ordered structure in {\iname}s, which we verify by establishing \ref{eq:leftreg}.
First we prove the following technical lemma for \hyperref[eq:leftdist]{left-distributive}  {\iname}s.

The identity below is related to Konikowska's postulate $(A3)$ (see \Cref{f:KonAx}) by replacing $\tc{x}$ by $\Ic{x}$
\begin{equation}\label{eq:A3}\tag{local commutativity}
    \Ic{y}\mc xy \approx \Ic{x}\mc yx 
        \qquad\text{dually,}\qquad
        \Oc{x}\jc y\jc x \approx \Oc{y}\jc x\jc y  
\end{equation}

\begin{lemma}\label{t:comAxs}
    Any {\iname} satisfying \ref{eq:A3} is left-regular and satisfies \ref{eq:comlocalunits}.
\end{lemma}
\begin{proof}
   Suppose $\m A$ is an {\iname} satisfying \ref{eq:A3}. Via the substitution $y\mapsto 1$, we find that $\Ic{x}\mc x \approx x$ holds:
   $$ \Ic{x}\mc x \approx \Ic{x}\mc 1x \approx \Ic{1}\mc x1 \approx \Ic{1}x := (1\jc 1)x \approx 1x \approx x.$$
   That \ref{eq:leftreg} holds is nearly immediate:
   $$xyx \approx \Ic{x}\mc xy\mc x \approx \Ic{xy} \mc x \mc xy \approx \Ic{xy}xy \approx xy. $$
   Next, observe that any {\iname} satisfies $\Oc{x} \approx \Oc{\Oc{x}}$ by idempotency since $\Oc{x}:=x\mc0$ and $\Oc{\Oc{x}}:=x0\mc 0\approx x0$. 
   Fix $a,b\in A$. 
   Finally for \ref{eq:comlocalunits}, we find
        \begin{align*}
    \Oc{a}\jc \Oc{b} 
        &= \Oc{a}\jc \Oc{b}  \jc \Oc{a} && \eqref{eq:leftreg}\\
        &= \Oc{\Oc{a}}\jc \Oc{b}  \jc \Oc{a} &&(\Oc{x}\approx\Oc{\Oc{x}})\\
        &= \Oc{\Oc{b}}\jc \Oc{b}  \jc \Oc{b} &&\eqref{eq:A3}\\
        &= \Oc{b}\jc \Oc{a}  \jc \Oc{b} &&(\Oc{x}\approx\Oc{\Oc{x}})\\
        &= \Oc{b}  \jc \Oc{a} && \eqref{eq:leftreg}.\qedhere 
    \end{align*}  
\end{proof}

\subsection{Weak absorption laws}\label{sec:WA}
     Let $\m L = \str{L,\jc,\mc}$ be an algebra such that both $\str{L,\jc}$ and $\str{L,\mc}$ are bands. 
     The algebra $\m L$ is called a \emph{skew-lattice} if it satisfies the \emph{absorption laws}:
    \begin{equation*}\label{eq:abs}
       \begin{array}{c | c |c}
        x\mc(x\jc y) \approx x & \text{left-absorption} & x\jc(x\mc y)\approx x\\ \hline
        (x\jc y)\mc y \approx y&\text{right-absorption} & (x\mc y)\jc y \approx y 
    \end{array} 
    \end{equation*}
A \emph{lattice} is simply a skew-lattice in which both operations are commutative.
Of the four algebras in Figure~\ref{f:3elem}, only the algebra $\m{SK}$ is a skew lattice, which is, in fact, a (distributive) lattice. 
However, the algebra $\McA$ is not skew as right-absorption fails, i.e., $(\varepsilon \jc 1)\mc 1 = \varepsilon\neq 1$. 
The algebra $\m{WK}$ satisfies none of the absorption identities, in general.

Recall that, in the theory of lattices, a structure is called \emph{$0$-bounded} if $0\land x \approx 0 \approx x\land 0$ holds. 
Call an {\iname} \defem{left-bounded} if its satisfies the identity
\begin{equation}\label{eq:leftann}\tag{left-bounded}
    0\mc x \approx 0 \qquad\qquad\text{dually, }\qquad\qquad 1\jc x \approx 1.
\end{equation}
\begin{lemma}\label{t:leftbdd}
    Any {\iname} satisfying left-absorption is also \ref{eq:leftann}. 
    Moreover, these these identities coincide in any \hyperref[eq:leftdist]{left-distributive}  {\iname}.
\end{lemma}
\begin{proof}
    On the one hand, \hyperref[eq:abs]{left-absorption} implies \ref{eq:leftann} on much weaker conditions. 
    Indeed, it only requires that $0$ is a unit for $\jc$: $0\mc x \approx 0\mc (0 \jc x) \approx 0$
    (and the dually for $\mc$ and $1$ using the dual \hyperref[eq:abs]{left-absorption} law).
    On the other hand, suppose an {\iname} $\m A$ is \ref{eq:leftann}. 
    Then $x\approx x\mc 1 \approx x \mc (1\jc y)$.
    If $\m A$ is left-distributive, then $x \mc (1\jc y) \approx (x\mc 1) \jc (x\mc y) \approx x \jc xy$.
\end{proof}
Since $\McA$ is clearly \ref{eq:leftann}, the following is immediate from \cref{t:leftbdd}.
\begin{proposition}\label{t:M3leftbdd}
    The algebra $\McA$ is \ref{eq:leftann} and therefore also satisfies: 
    \begin{equation}\label{eq:leftabs}\tag{left-absorption}
        x\mc (x\jc y)\approx x 
        \qquad\text{dually, }\qquad
        x \jc xy \approx x
    \end{equation}
\end{proposition}

\subsection{Focusing on left-distributivity}
Here we combine some of the previous mentioned notions when in the \hyperref[eq:leftdist]{left-distributive}  setting.
When in the presence \ref{eq:localdecomp} and a weak commutative law, we obtain the following.
\begin{lemma}\label{t:coherence}
    Let $\m A$ be a \hyperref[eq:leftdist]{left-distributive}  {\iname} satisfying \ref{eq:localdecomp} and \ref{eq:comlocalunits}. 
    Then $\m A$ satisfies the following identities:
    \begin{align*}
        \Ic{x}\approx \nc{\Oc{x}}
         &\approx \Ic{\nc{x}}
        &&\text{dually,}&
        \Oc{x} \approx \Oc{\nc{x}}&
        \approx \nc{\Ic{x}}
        \label{eq:unitcoh}\tag{unit-coherence}\\
        xy\nc{x} &\approx \Oc{xy}
        &&\text{dually,}&
        x\jc y \jc \nc{x} &\approx \Ic{x\jc y}
        \label{eq:dramcon}\tag{dramatic conjugation}\\
        \Oc{x}\jc y &\approx \Ic{x}\mc y
        \label{eq:coherence}\tag{left-coherence}
    \end{align*}
\end{lemma}
\begin{proof}
    By \cref{t:ICabs,t:divisibility,t: LR2,t:localdecomp}, 
    $\m A$ satisfies \ref{eq:localunits}, \ref{eq:leftreg}, \ref{eq:divis}, and the  \ref{eq:localcomp} identity, respectively.
    
    For \ref{eq:unitcoh}, first note that $\nc{\Oc{x}} \approx \Ic{\nc{x}}$ holds in any {\iname} by the De~Morgan laws, so we need only exhibit $\Ic{x}\approx\Ic{\nc{x}}$. 
    Moreover, it will suffice to verify the identity $\Ic{x} \mc \Ic{\nc{x}} \approx \Ic{x}$ holds, as the claim would then follow from \ref{eq:comlocalunits} via $\Ic{x} \approx \Ic{x} \mc \Ic{\nc{x}} \approx \Ic{\nc{x}} \mc \Ic{x} = \Ic{\nc{x}}$. 
    Fixing $a\in A$, we observe:
    \begin{align*}
        \Ic{a}\mc \Ic{\nc{a}} 
            & = a\Ic{\nc{a}} \jc \nc{a}\Ic{\nc{a}} && \text{\eqref{eq:localdecomp}}\\
            & = a\Ic{\nc{a}} \jc \nc{a} && \text{\eqref{eq:localunits}}\\  
            & = a\nc{a} \jc a \jc \nc{a} && \text{(\ref{eq:leftdist} since $\Ic{\nc{a}}:=\nc{a}\jc 1$)}\\
            & = \Oc{a} \jc a \jc \nc{a} = a \jc \nc{a} && \text{(\ref{eq:localcomp} \& \ref{eq:localunits})}\\
            & = \Ic{a} && \text{\eqref{eq:localcomp}}.   
    \end{align*}
For the remaining, fix $a,b\in A$. For \ref{eq:dramcon}, observe:
    \begin{align*}
        ab\nc{a} 
            &= ab\nc{a}a &&\text{\eqref{eq:leftreg}}\\
            &=ab\Oc{\nc{a}} &&\text{\eqref{eq:localcomp}}\\
            &=ab\Oc{a}=aba0 &&\text{(\ref{eq:unitcoh} \& def.~ $\Oc{a}:=a0$)}\\
            &=ab0 =\Oc{ab} &&\text{(\ref{eq:leftreg} \& def.~ $\Oc{ab}:=ab0$).}
    \end{align*}
For \ref{eq:coherence}, observe:
    \begin{align*}
        \Oc{a} \jc b 
            &= \Oc{a} \jc \nc{\Oc{a}}b && \text{\eqref{eq:divis}}\\
            &= \Oc{a} \jc \Ic{a}b && \text{\eqref{eq:unitcoh}}\\
            & = (\Oc{a} \jc \Ic{a})(\Oc{a}\jc b) && \text{\eqref{eq:leftdist}}\\
            & = (\Oc{a} \jc a\jc 1)(\Oc{a}\jc b) && \text{(by definition of $\Ic{a}$)}\\
            & = (a\jc 1)(\Oc{a}\jc b) && \text{\eqref{eq:localunits}}\\
            & = (a\jc 1)({\Oc{\nc{a}}}\jc b) && \text{\eqref{eq:unitcoh}}\\
            & = \Ic{a}(\nc{\Ic{a}}\jc b) && \text{(by def.~ of local constants)}\\
            & = \Ic{a}b && \text{\eqref{eq:divis}.} \qedhere
    \end{align*}
    \end{proof}  
\begin{remark}\label{rem: 1x0-x0}
    Note that, in the presence of \ref{eq:localdecomp} and \ref{eq:unitcoh}, the identity $\Oc{\Ic{x}}\approx \Oc{x}$ must hold. 
    Indeed, $(x\jc 1)0 \approx x0\jc \nc{x}0\approx  x0\jc  x0 \approx x0$ holds, i.e., $\Oc{\Ic{x}}\approx \Oc{x} $. 
    Dually, so too does $\Ic{\Oc{x}}\approx \Ic{x} $.
\end{remark}
\cadd{
As noted in \cite{Guzman-Squier}, the term-operation $x\mc_{\mathsf{wk}}y:= (x\jc y)(y\jc x)$, when evaluated in $\McA$, coincides with the operation $\land_{\mathsf{wk}}$ in $\mathbf{WK}$, which is commutative. 
So we consider the following identity, which holds for $\McA$ \cite{Guzman-Squier}.
\begin{equation}\label{eq:Guz}
xy \jc yx \approx yx \jc xy \qquad\text{dually,}\qquad (x\jc y)(y\jc x) \approx (y\jc x)(x \jc y)
\tag{\textsf{wk}-commutativity}
\end{equation}
}
\cadd{
\begin{lemma}\label{t:dist-comm-equiv}
    Let $\m A$ be a \hyperref[eq:leftdist]{left-distributive} {\iname} satisfying \ref{eq:localdecomp}. 
    If $\m A$ satisfies any of the identities for \ref{eq:comlocalunits}, \ref{eq:A3}, or \ref{eq:Guz}, then it satisfies all of them.
\end{lemma}
\begin{proof}
    By \cref{t: LR2}, $\m A$ is left-regular and satisfies $(\star)$ $\Ic{x}yx \approx xy \jc yx$. 
    It is immediate from $(\star)$ that $\m A$ satisfies \ref{eq:A3} iff it satisfies \ref{eq:Guz}.
    From \cref{t:comAxs}, we have \ref{eq:A3} implies \ref{eq:comlocalunits}. 
    So suppose $\m A$ satisfies \ref{eq:comlocalunits}. By \cref{t:coherence}, $\m A$ also satisfies \ref{eq:coherence}. 
    Let $a,b\in A$ and observe:
    \begin{align*}
        ab \jc ba 
            = \Ic{a}ba 
            &= \Ic{a}\Ic{b}ba && \text{($\star$ \& \ref{eq:localunits})}\\
            &= \Ic{b}\Ic{a}ba && \eqref{eq:comlocalunits} \\
            &= \Oc{b} \jc \Ic{a}ba && \eqref{eq:coherence}\\
            &= \Oc{b} \jc ab \jc ba && (\star) \\
            &= \Ic{b}  ab \jc ba && \eqref{eq:coherence} \\
            &= ba \jc ab \jc ba &&(\star) \\
            &= ba \jc ab &&\eqref{eq:leftreg}.
    \end{align*}
    Thus \ref{eq:Guz} holds and, by $(\star)$, so too does \ref{eq:A3}.
\end{proof}
}

Lastly, when further in the presence of a weak absorption law, we have the following result.

\begin{lemma}\label{t:LDidens}
    Let $\m A$ be a \hyperref[eq:leftdist]{left-distributive}  \ref{eq:leftann} {\iname} also satisfying \ref{eq:localdecomp} and \ref{eq:comlocalunits}. 
Then $\m A$ satisfies \ref{eq:paradist} and the following identities:
\begin{align}
    x \mc y & \approx (\nc{x} \jc y)\mc x 
     &\mbox{dually, }&& x \jc y & \approx \nc{x} y\jc x 
     \label{eq:paracomm}\tag{left-para\allowbreak commutativity} \\ 
     xy \jc \nc{x}z & \approx (\nc{x}\jc y)(x\jc z) 
     \label{eq:orthocoher}\tag{left-ortho\allowbreak coherence} \\
     xy \jc \nc{x}z & \approx \nc{x}z \jc xy 
     &\mbox{dually, }&& (x \jc y)(\nc{x} \jc z) & \approx (\nc{x}\jc z)(x\jc y) 
     \label{eq:orthocom}\tag{left-ortho\allowbreak commutativity} 
\end{align}
    as well as those for 
    \ref{eq:localunits},
    \ref{eq:leftreg},
    \ref{eq:orthodist},
    \ref{eq:localcomp},
    \ref{eq:divis},
    \ref{eq:unitcoh},
    \ref{eq:coherence},
    \ref{eq:dramcon},
    \ref{eq:A3}, 
    \cadd{\ref{eq:Guz}}
    and \ref{eq:leftabs}.
\end{lemma}

\begin{proof} 
The final list of identities follow from \cref{t:ICabs,t: LR2,t:localdecomp,t:divisibility,t:coherence,t:leftbdd}.
For the remaining identities, fix $a,b,c\in A$. 
    For \eqref{eq:paracomm}, observe:
    \begin{align*}
        ab 
            &= \Ic{a}ab && \text{\eqref{eq:localunits}}\\
            &= \Ic{\nc{a}}ab &&\text{\eqref{eq:unitcoh}}\\
            &= (\nc{a}\jc ab)\mc (a\jc ab) &&\text{\eqref{eq:localdecomp}}\\
            &= (\nc{a}\jc b)\mc (a\jc ab) &&\text{\eqref{eq:divis}}\\
            &= (\nc{a}\jc b)\mc a &&\text{\eqref{eq:leftabs}}
    \end{align*}

    For \eqref{eq:orthocoher}, observe:
    \begin{align*}
        (\nc{a}\jc b)(a\jc c)
            &= (\nc{a}\jc b)a \jc (\nc{a}\jc b)c &&\text{\eqref{eq:leftdist}}\\
            &= ab \jc  (\nc{a}\jc b)c &&\text{\eqref{eq:paracomm}}\\
            &= ab \jc  \nc{(ab)}(\nc{a}\jc b)c &&\text{\eqref{eq:leftabs}}\\
            &= ab \jc  (\nc{a}\jc \nc{b})(\nc{a}\jc b)c &&\text{(De~Morgan)}\\
            &= ab \jc  (\nc{a}\jc \nc{b}b)c &&\text{\eqref{eq:leftdist}}\\
            &= ab \jc  (\nc{a}\jc \nc{b}0)c &&\text{\eqref{eq:localcomp}}\\
            &= ab \jc  (\nc{a}\jc \nc{b})(\nc{a}\jc 0)c &&\text{\eqref{eq:leftdist}}\\
            &= ab \jc  \nc{(ab)}(\nc{a}c)&&\text{(by De~Morgan and $0$ a unit for $\jc$)}\\
            &= ab \jc \nc{a}c&&\text{\eqref{eq:divis}}
    \end{align*}

    For \eqref{eq:orthocom}, first we observe the following:
    \begin{align*}
        ab \jc \nc{a}c 
            &=\nc{(ab)}(\nc{a}c) \jc ab &&\text{\eqref{eq:paracomm}}\\
            &= (\nc{a}\jc\nc{b})\nc{a}c \jc ab &&\text{(De~Morgan)}\\
            &= (\nc{a}\jc\nc{b}\jc 1)\nc{a}c \jc ab &&\text{(\cref{t:ICabs} since $\Ic{\nc{b}}=\Ic{\nc{b}\jc 1}$)}\\
            &= \Ic{\nc{a}\jc \nc{b}}\nc{a}c \jc ab &&\text{(by definition of $\Ic{\nc{a}\jc \nc{b}} = \nc{a}\jc \nc{b} \jc 1$)}\\
            &= \Ic{ab}\nc{a}c \jc ab &&\text{(\ref{eq:unitcoh} and De~Morgan)}\\
            &= \Oc{ab}\jc \nc{a}c \jc ab &&\text{\eqref{eq:coherence}}
    \end{align*}
    Now, since $\nc{}$ is an involution, i.e., $a=\nnc{a}$, by the same argument we have $ \nc{a}c \jc ab =  \Oc{\nc{a}c} \jc ab \jc \nc{a}c$.
    Abbreviating, let $x:= ab$ and $y:=\nc{a}c$; we have shown $x\jc y = \Oc{y} \jc x \jc y$ and $y\jc x = \Oc{x} \jc y \jc x$.
    But \ref{eq:A3} guarantees $\Oc{y} \jc x \jc y = \Oc{x} \jc y \jc x$. 
    Consequently, $x\jc y = y \jc x$, thus completing our claim.
    
    Finally, for \eqref{eq:paradist} we observe:
    \begin{align*}
        ac \jc \nc{a}bc 
            &= (\nc{a}\jc c)\mc(a\jc bc) && \text{\eqref{eq:orthocoher}}\\
            &= (a\jc bc)\mc(\nc{a}\jc c) && \text{\eqref{eq:orthocom}}\\
            &= (a\jc b)\mc(a\jc c)\mc(\nc{a}\jc c) &&\text{\eqref{eq:leftdist}}\\
            &=(a\jc b)\mc(\Oc{a}\jc c) &&\text{\eqref{eq:localdecomp}}\\
            &=(a\jc b)\mc\Ic{a}\mc c &&\text{\eqref{eq:coherence}}\\ 
            &=(a\jc b)\Ic{a}(a\jc b)\mc c &&\text{\eqref{eq:leftreg}}\\ 
            &=(a\jc b)(\Oc{a}\jc a\jc b)\mc c &&\text{\eqref{eq:coherence}}\\ 
            &=(a\jc b)(a\jc b)\mc c &&\text{\eqref{eq:localunits}}\\ 
            &=(a\jc b)\mc c &&\text{\eqref{eq:idem}}\qedhere
    \end{align*}
\end{proof}
\subsection{Focusing on left-divisibility}
Up until now, we have been largely focused on {\iname}s satisfying \ref{eq:leftdist}. 
Here, we will investigate how properties are similarly shared for {\iname}s starting from the setting of \ref{eq:divis} instead.

\begin{lemma}\label{t:bdddivis}
    A \hyperref[eq:divis]{left-divisible} {\iname} satisfies \ref{eq:orthodist} iff it satisfies \ref{eq:localdecomp}, and it satisfies \ref{eq:leftabs} iff it is \ref{eq:leftann}.
\end{lemma}
\begin{proof}
    Recall that \ref{eq:divis} implies  \ref{eq:localcomp} by \cref{t:divisibility}. 
    Hence the first claim holds by \cref{t:localdecomp}.
    For the second claim, the forward direction follows from \cref{t:leftbdd}. 
    On the other hand, if \ref{eq:leftann} holds, then $1\jc y \approx 1$, and hence
    $x\approx x\mc 1 \approx x \mc (1\jc y) \approx x\mc (\nc{x}\jc 1 \jc y) \approx x\mc (\nc{x}\jc x \jc y) \approx x\mc (x\jc y).$
\end{proof}

\begin{lemma}\label{t:diviscoh}
    Let $\m A$ be a \hyperref[eq:divis]{left-divisible} {\iname} satisfying \ref{eq:orthodist}. Then the following hold:
    \begin{enumerate}
        \item $\m A$ satisfies \ref{eq:localunits} and the identity $\Ic{x}\mc \Ic{\nc{x}} \approx \Ic{x}$.
        \item If $\m A$ is additionally \ref{eq:leftann}, then it also satisfies \ref{eq:coherence} and \ref{eq:leftreg}.
    \end{enumerate}
\end{lemma}
\begin{proof}
Recall $\m A$ satisfies \ref{eq:localdecomp} by \cref{t:bdddivis}. 
Let $a\in A$. 
For \ref{eq:localunits}, observe:
    \begin{align*}
        \Ic{a}a 
            &=aa\jc \nc{a}a &&\eqref{eq:localdecomp}\\
            &=a\jc \nc{a}a &&\eqref{eq:idem}\\
            &=a\jc a &&\eqref{eq:divis}\\
            &= a &&\eqref{eq:idem}
    \end{align*}
For the second identity, observe:
\begin{align*}
        \Ic{a}\mc \Ic{\nc{a}} 
            &= a\Ic{\nc{a}} \jc \nc{a}\Ic{\nc{a}} &&\eqref{eq:localdecomp}\\
            &= a \jc \nc{a}\Ic{\nc{a}} &&\eqref{eq:divis}\\
            &= a \jc \Ic{\nc{a}}   &&\eqref{eq:divis}\\
            &= a \jc 1 \jc  1  && (\text{\ref{eq:divis} since $\Ic{\nc{a}}:= \nc{a}\jc 1 =\nc{a}1\jc 1 $})\\
            & = \Ic{a} &&\eqref{eq:idem}
    \end{align*}

For the remaining claims, suppose $\m A$ is also \ref{eq:leftann} (and hence satisfies \ref{eq:leftabs}), and fix also $b\in A$. 
We obtain \ref{eq:coherence} as follows:
    \begin{align*}
        \Oc{a}\jc b 
            &= \Oc{a}\jc \nc{\Oc{a}}b &&\eqref{eq:divis}\\
            &= \Oc{a}b\jc \nc{\Oc{a}}b &&\eqref{eq:leftann}\\
            &= (\Oc{a}\jc 1)b &&\eqref{eq:localdecomp}\\
            &= (a\jc 1)(\nc{a}\jc 1)b &&\eqref{eq:localdecomp}\\
            &=\Ic{a}b &&(\text{since $\Ic{a}\Ic{\nc{a}}=\Ic{a}$})
    \end{align*}
For \ref{eq:leftreg}, we observe
\begin{align*}
        aba
            &= aaba &&\eqref{eq:idem}\\
            &=a(\nc{a}\jc ab)a &&\eqref{eq:divis}\\
            &=a(\nc{a}\jc ab)(a\jc ab)   &&\eqref{eq:leftabs}\\
            &=a(\nc{a}a \jc ab) &&\eqref{eq:orthodist}\\
            &=a(\nc{a} \jc \nc{a}a \jc ab) &&\eqref{eq:divis}\\
            &=a(\nc{a} \jc ab) &&\eqref{eq:leftabs}\\
            &=a\mc ab &&\eqref{eq:divis}\\
            &=ab &&\eqref{eq:idem}\qedhere
    \end{align*}
\end{proof}
Taking these together, we are able to recover \ref{eq:leftdist}.

\begin{lemma}\label{t:divisLdist}
    In a \hyperref[eq:divis]{left-divisible} \ref{eq:leftann} {\iname} $\m A$ satisfying \ref{eq:orthodist}:
    \begin{enumerate}
        \item\label{item: divisLdist_Guz} 
        \cadd{
            \ref{eq:Guz} implies \ref{eq:comlocalunits}
            }
        \item\label{item: divisLdist} \ref{eq:comlocalunits} implies \ref{eq:leftdist}
    \end{enumerate}
    \cadd{
    Consequently, if $\m A$ satisfies any of the identities for \ref{eq:comlocalunits}, \ref{eq:A3}, or \ref{eq:Guz}, then it satisfies all of them.
    }
\end{lemma}
\begin{proof}
    Let $\m A$ be such an {\iname} satisfying 
    \ref{eq:divis}, \ref{eq:leftann}ness, and \ref{eq:orthodist},
    and note that $\m A$ satisfies also \ref{eq:localdecomp}, \ref{eq:leftabs}, \ref{eq:coherence}, and \ref{eq:leftreg} by \cref{t:bdddivis,t:diviscoh}. 

    \cadd{
    For \cref{item: divisLdist_Guz}, suppose $\m A$ additionally satisfies \ref{eq:Guz} and let $a,b\in A$. 
    From \ref{eq:Guz} we have 
    $\Oc{a}\Ic{b} \jc \Ic{b}\Oc{a}  
    = \Ic{b}\Oc{a} \jc \Oc{a}\Ic{b}$,
    and using the fact that \ref{eq:leftann}ness entails $\Oc{x}\mc y \approx \Oc{x}$, we obtain $(\star)$: 
    $
    \Oc{a}\jc \Ic{b}\Oc{a} 
    = \Ic{b}\Oc{a} \jc \Oc{a}$. 
    Observe,
    \begin{align*}
        \Oc{a} \jc \Oc{b} 
            &= \Oc{a} \jc \Oc{b} \jc \Oc{a} && \eqref{eq:leftreg}\\
            &= \Oc{a} \jc \Ic{b} \Oc{a} &&\eqref{eq:coherence}\\
            &=  \Ic{b}\Oc{a} \jc \Oc{a} &&(\star)\\
            &= \Oc{b} \jc \Oc{a}\jc \Oc{a} = \Oc{b} \jc \Oc{a} &&\text{(\ref{eq:coherence} \& \ref{eq:idem})} .
    \end{align*}
    }
    
    Towards establishing \cref{item: divisLdist},
    we first show that $\m A$ satisfies \ref{eq:unitcoh} and \ref{eq:paracomm}. 
    The former is immediate from \cref{t:diviscoh}(1) and \ref{eq:comlocalunits}: $\Ic{x} \approx \Ic{x}\mc \Ic{\nc{x}} \approx \Ic{\nc{x}}\mc \Ic{{x}} \approx \Ic{\nc{x}}.$
    For \ref{eq:paracomm}, we fix $a,b\in A$ and observe:
    \begin{align*}
        (\nc{a}\jc b)a 
            &= (\nc{a}\jc ab)a &&\eqref{eq:divis}\\
            &= (\nc{a}\jc ab)(a\jc ab) &&\eqref{eq:leftabs}\\
            &= \nc{a}0 \jc ab &&\eqref{eq:localdecomp}\\
            &= \Oc{a} \jc ab &&\eqref{eq:unitcoh}\\
            &=\Ic{a}ab &&\eqref{eq:coherence}\\
            &=ab &&\eqref{eq:localunits}
    \end{align*}
    Finally, to verify \ref{eq:leftdist}, fix $a,b,c\in A$ and observe:
    \begin{align*}
        a(b\jc c) 
            &= (\nc{a} \jc b \jc c)a && \eqref{eq:paracomm}\\
            &=(\nc{a}\jc ab \jc ac)a&&\text{(\ref{eq:divis}~\&~\ref{eq:leftreg})}\\
            &=(\nc{a}\jc ab \jc ac)(a\jc ab \jc ac)&&\text{(two applications of \ref{eq:leftabs})}\\
            &=\Oc{\nc{a}} \jc ab \jc ac &&\eqref{eq:localdecomp}\\
            &= \Oc{a}\jc ab \jc ac &&
            \eqref{eq:unitcoh}\\
            &=ab \jc ac &&\text{(\ref{eq:coherence}~\&~\ref{eq:localunits})}
    \end{align*}
\cadd{
    Lastly, the final claim follows from the above, \cref{t:dist-comm-equiv}, and \cref{t:comAxs}.
}
\end{proof}

\section{Axiomatizations for McCarthy algebras}\label{sec:axiomMC}
\subsection{{\MKname}s}\label{sec:MKdef}
In \cite{Kow96}, Konikowska gives a list if identities claimed to hold for the algebra $\McA$. 
Translated into our signature, we call this list \emph{Konikowska's postulates}, shown in the figure below.\footnote{In \cite{Kow96}, the right-most term in axiom $(A4')$ contains a typo, there written [sic] $(\nc{x}\jc y)y$.} 
\begin{figure}[H]
    \centering
    \scriptsize
    $
    \begin{array}{r l r l}
        (A1)& \nnc{x} \approx x
        &
        (A2)& \nc{1}\approx 0
        \\
        (A3)& \fc{y}\jc x\jc y \approx \fc{x}\jc y\jc x 
        &
        (A3')& \tc{y}\mc x y \approx \tc{x}\mc y x 
        \\
        (A4)& x \jc y \approx x \jc \nc{x}y \approx \nc{x}y \jc x
        &
        (A4')& x y \approx x (\nc{x}\jc y) \approx (\nc{x}\jc y) x
        \\
        (A5)& x\jc y \jc x \approx x \jc y
        &
        (A5')& xyx \approx xy
        \\
        (A6)& x \approx x\jc xy
        &
        (A6')& x \approx x(x\jc y)
        \\
        (A7)& x\jc x \approx x
        &
        (A7')& x x\approx x
        \\
        (A8)& 1\jc x \approx 1
        &
        (A8')& 0\mc x \approx 0
        \\
        (A9)& 0\jc x \approx x\jc 0 \approx x
        &
        (A9')& 1\mc x \approx x\mc 1 \approx x
        \\
        (A10)& \tc{x} \approx \nc{x}\jc x
        &
        (A10')& \fc{x} \approx \nc{x}\mc x
        \\
        (A11)& \tc{x} \jc 1 \approx \tc{x}
        &
        (A11')& \fc{x}\mc 0 \approx \fc{x}
        \\
        (A12)& x\approx \tc{x}\mc x
        &
        (A12')& x\approx \fc{x}\jc x
        \\
        (A13)& \nc{(x\jc y)} \approx \nc{x}\mc\nc{y}
        &
        (A13')& \nc{(x\mc y)} \approx \nc{x}\jc\nc{y}
        \\
        (A14)& x\jc (y\jc z) \approx (x\jc y)\jc z
        &
        (A14')& x(yz) \approx (xy)z
        \\
        (A15)&  x\jc yz \approx (x\jc y)(x\jc z)
        &
        (A15')& x (y\jc z) \approx xy\jc xz
        \\
        (A16)& xy \jc z \approx (x\jc z)(\nc{x}\jc y\jc z)
        &
        (A16')& (x\jc y)z \approx xy \jc \nc{x}yz
    \end{array}
    $
    \caption{Konikowska's postulates \cite[p. 169]{Kow96} translated into our signature of {\iname}s, recalling $\tc{x}:=x\jc\nc{x}$ and $\fc{x}:=x\mc\nc{x}$.}
    \label{f:KonAx}
\end{figure}
We observe that Konikowska's postulates include those for {\iname}s: the identities $(A7/A7')$, $(A9/A9')$, and $(A14/A14')$ specify that the reducts $\str{\jc,0}$/$\str{\mc,1}$ are idempotent monoids; $(A1)$ specifies that $\nc{}$ is an involution; and finally $(A2)$ and $(A13/A13')$ specify that the reducts $\str{\jc,\nc{},0}$ and $\str{\mc,\nc{},1}$ are De~Morgan dual, i.e., $\str{\jc,\nc{},0}=\str{\mc,\nc{},1}^\partial$. 
Consequently, any identity $(A\#')$ is the De~Morgan dual of $(A\#)$.

This leads us to the following definition.
\begin{definition}
    By a \defem{\McKname} (\MKname) we refer to any {\iname} satisfying Konikowska's postulates in \Cref{f:KonAx}.
    By $\Mk$ we denote the variety of {\MKname}.
\end{definition}
Utilizing the results of the previous section, we provide the following reduced and equivalent axiomatizations for the variety $\Mk$ of {\MKname}s.

\begin{theorem}\label{t:equivAxioms}
    The variety of {\MKname}s \crev{has an equational basis}, relative to {\iname}s, given by: 
    \begin{enumerate}
        \item\label{eqAx:wd} \emph{[Weak Distributivity]} At least one of the following items:
            \begin{enumerate}
                \item\label{wd:dist} \ref{eq:leftdist} and \ref{eq:localdecomp}, i.e., $x(y\jc z)\approx xy \jc xz$ and $\Ic{x}y\approx xy\jc \nc{x}y$;
                \item\label{wd:div} \ref{eq:divis} and \ref{eq:orthodist}, i.e., $xy\approx x(\nc{x}\jc y)$ and $(x\jc \nc{x})y\approx xy\jc \nc{x}y$;
                \item\label{wd:pdist} \ref{eq:paradist}, i.e., $(x\jc y)z\approx xz \jc \nc{x}yz$.
            \end{enumerate}  
        \item\label{eqAx:wabs} \emph{[Weak Absorption]} At least one of the following items:
            \begin{enumerate}
                \item\label{wabs:lann} \ref{eq:leftann}, i.e, $0\mc x \approx 0$;
                \item\label{wabs:labs} \ref{eq:leftabs}, i.e., $x(x\jc y)\approx x$.
            \end{enumerate}  
        \item\label{eqAx:wc} \emph{[Weak Commutativity]} At least one of the following items:
            \begin{enumerate}
                \item\label{wc:lu} \ref{eq:comlocalunits}, i.e., $\Ic{x}\mc \Ic{y} \approx \Ic{y} \mc \Ic{x}$;
                \item\label{wc:A3} \ref{eq:A3}, i.e., $\Ic{x}\mc yx \approx \Ic{y} \mc xy$;
                \item\label{wc:t=1} Any identity obtained from \ref{wc:lu} or \ref{wc:A3} by replacing any instance of a term $\Ic{z}$ by the term $\tc{z}$;
                \item\label{wc: Guz} \cadd{\ref{eq:Guz}, i.e., $xy \jc yx \approx yx \jc xy$.}
            \end{enumerate}
    \end{enumerate}
    Moreover, and in addition to each of the above, the following list of identities hold for {\MKname}s: 
    \ref{eq:leftreg},
    \ref{eq:localcomp},
    \ref{eq:localunits},
    \ref{eq:unitcoh},
    \ref{eq:coherence},
    \ref{eq:orthocoher},
    \ref{eq:orthocom},
    \ref{eq:paracomm}, and
    \ref{eq:dramcon}.
\end{theorem}
\begin{proof}
    We begin by showing that any combination of the items above specify the same variety.
    First, observe that any {\iname} satisfying item \ref{wd:pdist} also satisfies item \ref{wd:div} by \cadd{\cref{t:paradist}}. 
    Thus it suffices to only verify the claim for items \ref{wd:dist} and \ref{wd:div}.
    Also, observe that \ref{wabs:lann} and \ref{wabs:labs} are both satisfied in any \hyperref[eq:leftdist]{left-distributive}  or \hyperref[eq:divis]{left-divisible} {\iname} satisfying either of them via \cref{t:leftbdd} and \cref{t:bdddivis}, respectively. 
    That is, the assumption of either one of them specify the same variety, so we may indicate an assumption of either one of them simply by writing \ref{eqAx:wabs}. 
\cadd{
    Similarly, any {\iname} satisfying items \ref{wd:dist} or \ref{wd:div}, all of the items \ref{wc:lu}, \ref{wc:A3}, and \ref{wc: Guz} hold iff any one of them hold by \cref{t:dist-comm-equiv,t:divisLdist}, respectively.
    Moreover, as any one of items \ref{wd:dist} and \ref{wd:div} entail \ref{eq:localcomp}, i.e., $\Ic{z}\approx \tc{z}$, any instance of item \ref{wc:t=1} is equivalent either \ref{wc:lu} or \ref{wc:A3}; we may indicate as assumption of any these items by writing \ref{eqAx:wc}. 
}

\cadd{
    On the one hand, assuming items \ref{wd:dist}, \ref{eqAx:wabs}, and \ref{eqAx:wc} entails all the remaining identities (in particular, item \ref{wd:div} and \ref{wd:dist}) via \cref{t:LDidens}.
    On the other hand, to establish that items \ref{wd:div}, \ref{eqAx:wabs}, and \ref{eqAx:wc} suffice, we need only verify the identities in \ref{wd:dist}, which is immediate from \cref{t:leftbdd,t:divisLdist}. 
    This completes all the claims. That is, each collection specifies the same variety; call it $\mathcal{V}$.
}
    
    Next, we show that $\mathcal{V}$ and $\Mk$ coincide. We begin by showing $\mathcal{V}\subseteq \Mk$, i.e., each one of Konikowska's postulates are satisfied in $\mathcal{V}$. Recall that $(A1), (A2), (A7/A7'), (A9/A9'), (A14/A14')$, and $(A13/A13')$ correspond to those of an {\iname}.  
    Moreover, notice that $(A4/A4')$ is the combination of \ref{eq:divis} and \ref{eq:paracomm}, $(A5/A5')$ is \ref{eq:leftreg}, $(A6/A6')$ is \ref{eq:leftabs}, $(A8/A8')$ is \ref{eq:leftann}, $(A15/A15')$ is \ref{eq:leftdist}, and $(A16/A16')$ is \ref{eq:paradist}.
    So we need only check the validity of $(A3), (A10)$, and $(A11)$. 
  
   Since $\mathcal{V}$ is \ref{eq:localcomp}, the terms $\tc{x},\fc{x}$ can be equivalently replaced by $\Ic{x},\Oc{x}$, respectively, in each identity. 
   So $(A3)$ follows from \ref{eq:A3}, $(A11)$ from idempotency, and $(A10)$ from \ref{eq:unitcoh} (since $\tc{\nc{x}}\approx \nc{x}\jc x$ by involutivity). Together, this establishes $\mathcal{V}\subseteq \Mk$.

   For the reverse inclusion, recall $(A16)$ is \ref{eq:paradist} and $(A6)$ is \ref{eq:leftabs}. Finally, $(A3)$ is an instance of item \ref{wc:A3}; i.e, replacing the term $\Ic{x}$ by $\tc{x}$ on the left of \ref{eq:A3}, and replacing $\Ic{y}$ by $\tc{y}$ on the right. 
   Hence $\Mk \subseteq \mathcal{V}$. Therefore these varieties coincide. The fact that the remaining identities hold follows from, e.g., \cref{t:LDidens}.
    \end{proof}
\begin{remark}
    For many of the axiomatizations above, the assumption of being  an {\iname} can be weakened to just a monoid with involution, i.e., idempotency may be dropped. 
    For instance, it follows simply from the \ref{eq:leftabs} laws:
    $x\jc x \approx x \jc x(x\jc x) \approx x$. 
    Also it is garnered from \ref{eq:divis}, and hence also \ref{eq:paradist}; indeed, on the one hand we have $\nc{x} \jc 1 \approx \nc{x} \jc x1 \approx \nc{x} \jc x$ by (the dual form of) left-divisibility and involutivity, and hence $x\mc x \approx x (\nc{x} \jc x) \approx x (\nc{x} \jc 1)\approx x\mc 1 \approx x $ by a few applications of divisibility. 
    More surprisingly for \ref{eq:paradist}, presented as $(x\jc y)z\approx xz \jc \nc{x}(yz)$, experiments with \texttt{Prover9} \cite{P9M4} reveal that even the assumption of associativity (as well as having a two-sided unit) can be recovered when also in presence of \ref{eq:leftabs} and either identity from item \ref{eqAx:wc} (i.e., the base condition is that of a left (or, instead, right) unital groupoid with involution). 
    However, verifying such results lie outside the scope of this paper.
\end{remark}

\begin{proposition}\label{t:MinMK}
    The algebra $\McA$ is an {\MKname}. 
    Consequently, the variety of {\MKname}s subsumes the variety of McCarthy algebras, i.e., $\MC \subseteq \Mk$.
\end{proposition}
\begin{proof}
    The algebra $\McA$ is an {\iname} (cf. \Cref{f:3elem}), is \hyperref[eq:leftdist]{left-distributive}  by \cref{t:leftsemiring}, \ref{eq:leftann} by \cref{t:M3leftbdd}, satisfies \ref{eq:comlocalunits} by \cref{t:M3localunitcom}, and satisfies \ref{eq:localdecomp} by \cref{t:M3localdecomp}. Hence $\McA\in \Mk$, and consequently $\MC := \mathsf{V}(\McA) \subseteq \Mk$. 
\end{proof}

The variety of {\MKname}s generalizes that of Boolean algebras. 
Indeed, as $\m 2$ generates the variety of Boolean algebras, and $\m 2$ is a subalgebra of $\McA$, from \cref{t:MinMK} it follows that $\BA \subseteq \Mk$. We now proceed to establish some equivalent conditions for when an {\MKname} is actually Boolean.

\begin{theorem}\label{t:BAeq}
    The following are equivalent for any {\MKname} $\m A = \str{A,\jc,\mc,\nc{},0,1}$:
    \begin{enumerate}
        \item $\m A$ is a Boolean Algebra.
        \item\label{BAeq:comm} $\m A$ is commutative.
        \item\label{BAeq:rightdist} $\m A$ satisfies right-distributivity; i.e., $\m A\models (x \jc y)\mc z \approx xz \jc yz$ \lbrak dually, $(x \mc y)\jc z \approx (x\jc z) \mc(y\jc z)$\rbrak.
        \item\label{BAeq:bounded} $\m A$ is right-bounded i.e., $\m A\models x \mc 0 \approx 0$ \lbrak dually, $x \jc 1 \approx 1 $\rbrak.
        \item\label{BAeq:orthocomp} $\m A$ is orthocomplemented; i.e., $\m A\models x \mc \nc{x} \approx 0$ \lbrak dually, $x \jc \nc{x} \approx 1 $\rbrak.
        \item\label{BAeq:rightabs} $\m A$ satisfies right-absorption; i.e., $\m A\models (x \jc y)\mc x \approx x$ \lbrak dually, $ xy\jc x \approx x$\rbrak.
    \end{enumerate}
\end{theorem}
\begin{proof}
    For the forward direction, first note that clearly $1{\Rightarrow}\ref{BAeq:comm}$ holds.
    Also, $\ref{BAeq:comm}{\Rightarrow}\ref{BAeq:rightdist}$ holds via \ref{eq:leftdist}, $\ref{BAeq:bounded} {\Leftrightarrow} \ref{BAeq:orthocomp}$ holds since $\m A$ is \ref{eq:localcomp}, and $\ref{BAeq:bounded}{\Rightarrow}\ref{BAeq:rightabs}$ follows since 
    $$(a\jc b)a = aa \jc \nc{a}ba = a\jc \Oc{\nc{a}b}$$
    via \ref{eq:paradist}, \ref{eq:idem}, and \ref{eq:dramcon}, so $(a\jc b)a = a$ follows from the assumption $x\mc 0 \approx 0$ since $0$ is a unit for $\jc$. 
    To complete the forward direction, it suffices to verify $\ref{BAeq:rightdist}{\Rightarrow}\ref{BAeq:bounded}$. Indeed, 
    \begin{align*}
        0 = 1\mc 0 
        &= (1\jc a)0 && \eqref{eq:leftann}\\
        &=(1\mc 0)\jc (a\mc 0) &&\text{(\ref{BAeq:rightdist}: right-distributivity)}\\
        &= a\mc 0&& \text{(def. of units)},  
    \end{align*}
    thus $\m A$ is right-bounded. 

    Towards the reverse implications, first observe that $\ref{BAeq:rightabs}{\Rightarrow}\ref{BAeq:bounded}$ since then $0 = (0\jc a)0$ and $(0\jc a) = a0$ since $0$ is a unit for $\jc$. 
    The implication $\ref{BAeq:bounded}{\Rightarrow}\ref{BAeq:comm}$ is a direct consequence of \ref{eq:A3}, as right-bounded simply means $\Ic{x}\approx 1$ holds. 
    Therefore items \ref{BAeq:comm}--\ref{BAeq:rightabs} are all equivalent. 
    So any one of them imply all of them, ensuring that $\m A$ is a distributive ortholattice, i.e., a Boolean algebra. 
    This establishes all the equivalences.
\end{proof}
Recall the relation $\leq_\jc$ is defined via $a\leq_\jc b$ iff $a\jc b = b$. 
While we may use the relation corresponding to either operation, we prefer to work with the one for $\jc$. 
Henceforth, we will abbreviate $\leq_\jc$ simply by $\leq$. Recall from \cref{t:leftregEquiv}, $\leq$ is a partial order for {\MKname}s since they are left-regular.

\begin{remark}\label{rem: order with left-absorption}
Observe that for any $a,b\in A$, by \ref{eq:leftabs}, $a\leq b$ implies $a\mc b = a$ (indeed $a \mc b = a\mc (a + b) = a$). 
The converse is not true in general (e.g., in $\McA$, $\varepsilon \mc x = \varepsilon$, with $x\in\{0,1\}$ but $\varepsilon\nleq x$).     
\end{remark}

\begin{lemma}\label{t:BelowZeros}
    The following quasi-identities hold for {\MKname}s:
    $$
    x\leq \Oc{y} \Rightarrow x = \Oc{x}
    \qquad\text{and}\qquad
    x\leq \nc{x} \Rightarrow x = \Oc{x}.
    $$
    Consequently, an {\MKname} can have at most one involution fixed-point $e=\nc{e}$. 
    Moreover, if $e$ is a involution fixed-point then $e=\Oc{e}=\Ic{e}$ and, for all $x$, $\Oc{x} \leq e$ and $ex=e = e\jc x$.
\end{lemma}
\begin{proof} 
    Let $\m A$ be an {\MKname} and fix $a,b\in A$. 
    On the one hand, if $a\leq \Oc{b}$ then $a\mc \Oc{b} = a$ by \cref{rem: order with left-absorption}. 
    Hence, since $\Oc{b}\mc 0:=b0\mc 0 = b0 = \Oc{b}$ by \ref{eq:idem}, we have 
    $$\Oc{a}:= a\mc 0
    =  (a\mc \Oc{b})\mc 0
    =a\mc (\Oc{b}\mc 0) 
    = a\mc \Oc{b}
    =a.$$
    On the other hand, if $a\leq \nc{a}$ then again by \cref{rem: order with left-absorption}, $a\mc \nc{a} = a$, but $a\mc \nc{a} = \Oc{a}$ as $\m A$ is \ref{eq:localcomp}, hence $a = \Oc{a}$. 
    Thus we have established both quasi-identities.

    For the final claims: if $e=\nc{e}$, then $e = e\jc e = e\jc \nc{e} = e\jc 1=\Ic{e}$ since it is \ref{eq:localcomp}, and hence also $e= \nc{e}=\nc{\Ic{{e}}}=\nc{e}0=e0=\Oc{e}$ as it is a fixed-point; 
    so $ex= e0x= e0= e$, and dually $e\jc x = e\jc 1 \jc x = e\jc 1 = e$, for any $x$ as a consequence of being \ref{eq:leftann}. 
    By invoking \ref{eq:comlocalunits}, we find $\Oc{x}\jc e = \Oc{x}\jc \Oc{e} = \Oc{e}\jc \Oc{x} = e \jc \Oc{x} = e$. 
    That $e$ is the unique such fixed-point, suppose $f = \nc{f}$. Then also $f = \Ic{f}$ and $fx=f$. The former yields $ef=fe$ by \ref{eq:comlocalunits}, and thus from the latter  $e=ef = fe =f$.  
\end{proof}

\begin{remark}
Observe that the final claims in \cref{t:BelowZeros} hold, more generally, for any \ref{eq:leftann} \ref{eq:localcomp} {\iname} satisfying \ref{eq:comlocalunits}.  
\end{remark}

\begin{corollary}\label{t:trivialalg}
    The trivial (one-element) {\iname} is the only {\MKname} in which $1=0$. 
\end{corollary}
\begin{proof}
    If $1=0:=\nc{1}$ in some $\m A \in \Mk$, then \cref{t:BelowZeros} yields $a\jc 1 = 1$ for all $a\in A$, i.e., $\m A$ is right-bounded. So from \cref{t:BAeq}\eqref{BAeq:bounded}, $\m A$ must be Boolean, but the only Boolean algebra with $1=0$ is the trivial one.
\end{proof}

\begin{corollary}\label{t:M3uniq}
    The algebra $\McA$ is the unique $3$-element {\MKname}.
\end{corollary}
\begin{proof}
    Suppose $\m A\in \Mk$ has cardinality $3$, and thus has  $1\neq 0$ by \cref{t:trivialalg}. 
    As $\m A$ is an {\iname} with $1\neq 0$, it must be one of the four algebras mentioned in \cref{t:3elemAlg}. 
    But the algebras $\m{WK}$, $\m{SK}$, and $\McA^\mathsf{op}$ each satisfy some identity from \cref{t:BAeq} (e.g., they all satisfy right-distributivity), but are not Boolean. Therefore $\m A$ must be isomorphic to $\McA$.
\end{proof}

\subsection{Structure of {\MKname}s} 

For an algebra $\m A$ and set of pairs $X\subseteq A\times A$, by $\Cg{\m A}(X)$ we denote the \defem{congruence on $\m A$ generated by $X$}, i.e., the least congruence ${\cng}$ on $\m A$ containing each pair from $X$. For singleton sets $X=\{(a,b)\}$, we simply write $\Cg{\m A}(a,b)$ to denote $\Cg{\m A}(X)$. 

\crev{Call a relation $R\subseteq A\times A$ \emph{term-defined} if there exists a unary term $t(x)$ in the language of $\m A$ so that, for each $a,b\in A$, $(a,b)\in R$ iff $t(a) = t(b)$ holds in $\m A$. 
It is clear then that term-defined relations are equivalence relations, thus a term-defined relation $R$ is a congruence on $\m A$ if and only if $R$ is compatible with the operations of $\m A$. }

For the remainder of this (sub)section, let us fix an {\MKname} $\m A = \str{A,\jc,\mc,\nc{},0,1}$.

\begin{proposition}\label{t:poly1}
    For fixed $a\in A$, let $\sim_a$ denote the congruence on $\m A$ generated by the pair $(a,1)$, i.e., ${\sim_a}:=\Cg{\m A}(a,1)$. Then for all $x,y\in A,$
    $$x\sim_a y \iff a\mc x = a\mc y.$$
    Moreover, $\sim_a$ coincides with the congruence $\Cg{\m A}(\nc{a},0)$.
\end{proposition}
\begin{proof}
    On the one hand, it is immediate that $ax=ay$ yields $x\sim_a y$, as $x = 1\mc x \sim_a ax = ay \sim_a 1\mc y = y$, where $a\sim_a 1$ by definition. 
    For the other implication, let us first show that the relation $\cng_a$ over $\m A$, defined via $x\cng_a y$ iff $ax = ay$, is a congruence. 
    That $\cng_a$ is an equivalence relation is immediate since $\cng_a$ is \crev{term-defined}. 
    The fact that $\cng_a$ is compatible with the operations is a consequence of \ref{eq:divis} and \ref{eq:leftreg}. Indeed, if $ax= ay$, then the De~Morgan laws give $\nc{a}\jc \nc{x} = \nc{a}\jc \nc{y}$, whence \ref{eq:divis} yields $a\nc{x} = a(\nc{a}\jc \nc{x}) = a(\nc{a}\jc\nc{y}) = a\nc{y}$; on the other hand, if also $au = av$, then \ref{eq:leftreg} yields $aux =auax = avay =avy$.
    So it suffices to verify ${\sim_a}\subseteq {\cng_a}$. 
    By definition, $\sim_a $ is the congruence generated from the pair $(a,1)$, and as $\m A$ is an idempotent monoid, $aa = a = a1$ giving $a\cng_a 1$; hence ${\sim_a} \subseteq {\cng_a}$. 
    The final claim follows since $\nc{}$ is an involution and $0=\nc{1}$.
\end{proof}

Recall that $a\leq b$ iff $a\jc b = b$. 
As a consequence of \ref{eq:comlocalunits}, the following is immediate.
\begin{proposition}\label{t:SL}
    The relation $\leq$ restricted to the subset $A\mc 0:=\{\Oc{a}: a\in A \}$ is a join-semilattice order with least element $0$, where the join $\vee$ is the restriction of $\jc$ to the set $A\mc 0$.
\end{proposition}
\begin{proof}
    The operation $\vee$ is idempotent since $\jc$ is; commutative since $\m A$ satisfies \ref{eq:comlocalunits}, i.e., $\Oc{x}\jc \Oc{y} \approx \Oc{y}\jc \Oc{x}$; and $0$ is the least element since it is the unit for $\jc$.  
\end{proof}
In light of the above \cref{t:SL}, we present the following definitions.
\begin{definition}\label{d:SL}
    By $\SL_{\m A}$ we denote the semilattice $\str{A \mc 0,\vee,0}$, and call it the \defem{semilattice skeleton} of $\m A$.
\end{definition}
\crev{Note that, by definition, $i\in \SL_{\m A}$ iff $i = \Oc{a}$ for some $a \in A$; in particular, $i = \Oc{i}$.
Often we will use the notation $i$ when referring to it as a member of a semilattice, while when referring to it as a member of the algebra $\m A$, we will use the notation $\Oc{i}$. In this way, we use $\Ic{i}$ to denote the element $\Oc{i}\jc 1$ (see \cref{rem: 1x0-x0}). 
} 

\begin{remark}\label{rem: 0ij=0i+0j}
    Using this notation, we note that $\Oc{i\vee j}=\Oc{i}\jc \Oc{j}$ and $\Ic{i\vee j}=\Ic{i}\Ic{j}$, for any $i,j\in \SL_{\m A}$.
\end{remark}
\begin{definition}\label{d:kerI}
For each $i\in \SL_{\m A}$, we define
the congruence ${\kerI_i}=\Cg{\m A}(\Ic{i},1)$. 
As a consequence of \cref{t:poly1}, note also that ${\kerI_i}:=\Cg{\m A}(\Oc{i},0)$ and for each $x,y\in A$:
$$ x \kerI_i y \iff \Ic{i}x=\Ic{i}y \iff \Oc{i}\jc x = \Oc{i}\jc y $$
\end{definition}
\begin{remark}\label{rem: inckerns}
    ${\kerI_i}\subseteq {\kerI_j}$ whenever $i\leq j$; indeed, $\Oc{j}\jc 0= \Oc{j}=\Oc{j\vee i} = \Oc{j} \jc \Oc{i}$, hence $(\Oc{i},0)\in {\kerI_j}$.
\end{remark}

For a partially ordered set $\str{P,\leq}$, a \defem{principal upset} is a subset $\{x\in P: a\leq x \}$ for some fixed element $a\in P$, here denoted by $\upset{a}$. 
The following is immediate from the definitions above and the Homomorphism Theorems (cf., \cite[Thm.~1.16]{mckenzie}).

\begin{theorem}\label{t:Ai}
    For each $i\in \SL_{\m A}$, the structure (with operations $\jc,\mc,\nc{}$ the restriction of those from $\m A$)
    $$\m A_i:=\str{\upset\Oc{i}, \jc,\mc,\nc{},\Oc{i},\Ic{i}}$$
    is an {\MKname} with $\m A_i \cong \m A/{\kerI_i}$, and $h_i:x\mapsto \Oc{i}\jc x$ is the homomorphism from $\m A$ onto $\m A_i$ with $\ker h_i = {\kerI_i}$.
\end{theorem}
\begin{corollary}\label{t:trivKerI}
    The congruence $\kerI_i$ is the identity only when $i=0$, i.e., ${\kerI_i}=  \Delta_{\m A}$ iff $i=0$.
\end{corollary}

Next, we show that multiplication is generally compatible with the order $\leq$.
\begin{proposition}\label{t:CompatibleOrder}
    The relation $\leq$ is a partial order for any {\MKname}. 
    Moreover, multiplication is order-preserving, i.e., {\MKname}s satisfy the quasi-identity $x\leq u \;\;\& \;\; y\leq v\;\Rightarrow\; xy\leq uv$.
\end{proposition}
\begin{proof}
    That the relation ${\leq}$ is a partial order follows from \ref{eq:leftreg} (see \cref{t:equivAxioms}) and \cref{t:leftregEquiv}. 
    To verify that multiplication is compatible with the order, let us suppose $x\leq u$ and $y\leq v$, i.e., $x\jc u = u$ and $y\jc v = v$, in some {\MKname}. Utilizing \cref{t:equivAxioms}, we observe:
    \begin{align*}
    xy\jc uv &= xy \jc (x\jc u)v &&(x\leq u)\\
            &= xy \jc xv \jc \nc{x}uv &&\eqref{eq:paradist}\\
            &=x(y\jc v) \jc \nc{x}uv &&\eqref{eq:leftdist}\\
            &=xv \jc \nc{x}uv &&(y\leq v)\\
            &=(x\jc u)v &&\eqref{eq:paradist}\\
            &=uv &&(x\leq u). \qedhere
    \end{align*}
\end{proof}
\cadd{
\begin{remark}\label{rem:a1j<1ja}
    The identity $x\Ic{y}\leq\Ic{y}x$ holds in every {\MKname}.  
    Indeed, since $x\leq \Ic{x}$, using the order preservation of multiplication, we have $x\Ic{y}= x\mc \Ic{y}x\leq \Ic{y}\mc \Ic{y}x=\Ic{y}\Ic{x}x = \Ic{y}x$ by \ref{eq:leftreg}, \ref{eq:comlocalunits}, and \ref{eq:localunits}.
\end{remark}
}
For a partially ordered set $\str{P,\leq}$, a \defem{principal downset} is a subset $\{x\in P: x\leq b \}$ for fixed $b\in P$, here denoted by $\downset{b}$. For $a,b\in P$, the set $[a,b]:=\upset{a}\cap\downset{b}$ is called an \defem{interval}. 
Note that, if $P$ has a least element $0$, then $\downset{b}=[0,b]$.

\begin{proposition}\label{t:MKinterval}
    {\MKname}s satisfy the following quasi-identity:  $x\leq u \;\& \; y\leq u\Rightarrow x\jc y = y\jc x$.
    Consequently, any (nonempty) interval from an {\MKname} forms a bounded semilattice with respect to addition and, moreover, is multiplicatively closed with its greatest element a multiplicative right-unit.
\end{proposition}
\begin{proof}
    Fix $b\in A$. 
    By \cref{rem: order with left-absorption} $x\leq b$ implies $xb = x(x\jc b) = x$, that is, $b$ is a multiplicative right-unit in $\downset{b}$.
    To verify the quasi-identity, suppose $x,y\in \downset{b}$. By \cref{rem: order with left-absorption}, $xb=x$ and $yb=y$. 
    First, we observe:
    \begin{align*}
    x\jc y &= x'y \jc x &&\eqref{eq:paracomm}\\
    &=x'y \jc xb &&(\text{\cref{rem: order with left-absorption}})\\
    &=x'y \jc x(y\jc b) &&(y\leq b)\\
    &=x'y \jc xy\jc xb &&\eqref{eq:leftdist}\\
    &=\Ic{x}y\jc xb &&\eqref{eq:localdecomp}\\
    &=\Ic{x}y\jc x && (\text{\cref{rem: order with left-absorption}})\\
    &=\Oc{x} \jc y  \jc x &&\eqref{eq:coherence}.
\end{align*}
Hence (i) $x\jc y = \Oc{x} \jc y \jc x$, and thus symmetrically (ii) $y\jc x = \Oc{y} \jc x \jc y$. 
Putting these together,
\begin{align*}
    x \jc y &= \Oc{x} \jc y \jc x&&\text{(i)} \\
            &= \Oc{x} \jc \Oc{y} \jc x\jc y&&\text{(ii)}\\
            &= \Oc{y} \jc \Oc{x} \jc x \jc y &&\eqref{eq:comlocalunits}\\
            &= \Oc{y}  \jc x \jc y &&\eqref{eq:localunits}\\
            & = y \jc x&&\text{(ii)}.  
\end{align*}
    Hence the quasi-identity holds. 
    Consequently, the operation $\jc$ is commutative in any principle downset, and hence, hereditarily, in any nonempty subinterval $[a,b]$. 
    That $[a,b]$ is closed under $\jc$ is generally true for any idempotent operation $\jc$ with relation $\leq_\jc$. 
    That $[a,b]$ is closed under multiplication is immediate from \cref{t:CompatibleOrder} and idempotency. That $b$ is a right-unit was already established.
\end{proof}
\begin{corollary}\label{t:BoolSubAlg}
    There is a largest Boolean subalgebra of $\m A$ and its universe is the interval $[0,1]$.
\end{corollary}
\begin{proof}
    Note that any subalgebra $\m B$ of an {\MKname} $\m A$ must contain $1$, and as $\m B$ is itself an {\MKname}, it is Boolean iff it is right-bounded by \cref{t:BAeq}, i.e., $1$ is the greatest element with respect to $\leq$. 
    The largest such subset of $A$ is the interval $[0,1]$, which is equivalent to the downset $\downset{1}$.
    That $[0,1]$ is closed under multiplication follows \cref{t:MKinterval}. 
    That $[0,1]$ is closed under $\nc{}$ follows from \ref{eq:unitcoh}; i.e., $x\in [0,1]$ iff $x\leq 1$ iff $\Ic{x}:=x\jc 1 = 1$, so $\nc{x}\jc 1 = \Ic{\nc{x}}=\Ic{x}=1$. 
\end{proof}
Let us call the subalgebra $[0,1]$ the \defem{principal Boolean component} of an {\MKname} $\m A$. 
\begin{definition}\label{d:fiber}
    By a \defem{fiber} of $\m A$ we refer to a structure 
    $$\m B_i=\str{B_i, \jc,\mc,\nc{},\Oc{i},\Ic{i}},$$ 
    where $i\in \SL_{\m A}$, $B_i$ is the interval $[\Oc{i},\Ic{i}]$, and whose operations $\{\jc,\mc,\nc{}\}$ are those of $\m A$ restricted to $B_i$.
    In this way, the principal Boolean component of $\m A$ is the fiber $\m B_0$.
\end{definition}
\begin{theorem}\label{t:fiberdecomp}
    The following hold:
    \begin{enumerate}
        \item For each $i\in \SL_{\m A}$, the fiber $\m B_i$ is the principal Boolean component of $\m A_i$.
        \item For each $i\in \SL_{\m A}$, $B_i = \{a\in A: \Oc{a} = \Oc{i}\}$.
        \item $A$ is a disjoint union of the intervals $B_i$, i.e., $A = \bigcup\limits_{i\in \SL_{\m A}} B_i$, and $B_i\cap B_j = \varnothing$ for distinct $i,j\in \SL_{\m A}$.
    \end{enumerate}
\end{theorem}
\begin{proof}
    The first item is a direct consequence of \cref{t:Ai} and \cref{t:BoolSubAlg}. 
    For the second item, for one inclusion fix $a\in B_i$. So $\Oc{i}\leq a\leq \Ic{i}$ which implies $\Oc{i}\mc 0\leq a \mc 0\leq \Ic{i} \mc 0$ via \cref{t:CompatibleOrder}, but this yields $\Oc{i}\leq \Oc{a}\leq \Oc{i}$ by \eqref{eq:leftann} and \cref{rem: 1x0-x0} (i.e., the $\Oc{\Ic{x}} \approx \Oc{x}$ holds). 
    Since $\leq$ is antisymmetric, it follows that $\Oc{a} = \Oc{i}$. 
    The reverse inclusion is immediate from the dual version \cref{rem: 1x0-x0} (i.e., $\Ic{\Oc{x}} \approx \Ic{x}$ holds). 
    The last claim is immediate from the second; i.e., each element is contained in the set $B_i$ where $\Oc{a}=\Oc{i}$, and if $a\in B_i\cap B_j$ then $i=j$ as $\Oc{i}=\Oc{j}$ would follow. 
\end{proof}

\subsection{\cadd{Completeness with McCarthy algebras}}
\cadd{
In this section, we prove that the variety $\Mk$ coincides with the variety $\MC$ of McCarthy algebras, therefore providing a finite equational basis for McCarthy algebras via \cref{t:equivAxioms}. 
We note that an equational basis for the variety $\MC$ is provided in \cite[Cor.~2.7]{Guzman-Squier}. 
In fact, their basis is subsumed by the identities listed in \cref{t:equivAxioms}, and hence it follows that $\Mk \subseteq \MC$ by the results of \cite{Guzman-Squier}. 
Therefore, by \cref{t:MinMK}, it is immediate that $\Mk = \MC$. 
However, for the sake of being self-contained, we prove this fact directly.
}
\begin{lemma}\label{t:SIbottomfiber}
    If $\m A$ is a subdirectly irreducible {\MKname}, then $\m B_0 \cong \mathbf{2}$. 
\end{lemma}
\begin{proof}
    First note that $1\neq 0$ as otherwise $\m A$ would be trivial (\cref{t:trivialalg}) and thus not subdirectly irreducible. 
    Fix $a\in B_0$ and consider the congruence ${\cng} := {\sim_a} \cap {\sim_{\nc{a}}}$. 
    We claim ${\cng} = \Delta_{\m A}$. 
    Indeed, if $x\cng y$ then $ax = ay$ and $\nc{a}x = \nc{a}y$ by \cref{t:poly1}. 
    Since $a\in B_0$, it must be that $\Ic{a}=1$, so by \eqref{eq:localdecomp}
    $$x = 1 \mc x = \Ic{a}x = ax\jc \nc{a}x = ay \jc \nc{a}y = \Ic{a}y=1\mc  y = y.$$ 
    Since $\m A$ is subdirectly irreducible, either ${\sim_a} = \Delta_{\m A}$ or ${\sim_{\nc{a}}} = \Delta_{\m A}$. 
    Consequently, either $a = 1$ or $a = 0$. 
    Since $a$ is an arbitrary member of $\m B_0$ and $1\neq 0$, it follows that $\m B_0 \cong \mathbf{2}$.
\end{proof}

\begin{lemma}\label{t:map2M3}
    Suppose $\m A$ is an {\MKname} such that $\m B_0  \cong \mathbf{2}$. Then the map $f:A\to \McSet$ defined via
    $$f(x):=\left\{\begin{array}{cc}
        x & \mbox{if }x\in B_0, \\
        \ee & \mbox{otherwise}.
    \end{array} \right. $$
is a homomorphism from $\m A$ to $\McA$. 
\cadd{
Consequently, $\ker f = \Delta_A \cup (E\times E) = \mathrm{Cg}_{\m A}(E\times E)$, where $E:=A\setminus B_0$.
}
\end{lemma}
\begin{proof}
Let $a,b\in A$ with $a\in B_i$ and $b\in B_j$ for some $i,j\in \SL_{\m A}$. 
As $\nc{a}\in B_i$ holds from \cref{d:fiber}(1), it follows that $f(\nc{a})=\nc{f(a)}$. 
For multiplication, note that the claim obviously holds if $a=1$ as $f(1)=1$ and $1\mc b = b \mapsto f(b)=1\mc f(b) = f(1)\mc f(b)$; so we will assume $a\neq 1$.
Also, the claim is easily verified when $a=0$ as $0\mc b = 0 \mapsto 0 = 0\mc f(b) = f(0)\mc f(b)$; so we will assume $a\neq 0$.
Since $\m B_0\cong\mathbf{2}$ and we have assumed $a\not\in \{0,1\}$, it must be that $i\neq 0$, and hence $f(a) = \ee$ by definition of the function $f$. 
But $\ee\mc x = \ee$ for any $x\in \McSet$, hence $f(a)\mc f(b) = \ee \mc f(b) = \ee$, so it suffices to show that $ab \not\in B_0$. 
Indeed, multiplication is order-preserving by \cref{t:CompatibleOrder}, and $\Oc{a}=\Oc{i}$ by \cref{t:fiberdecomp}(2), so 
$$\Oc{i} = a\mc 0 \leq a\mc b, $$
and $ab \in A_i$ by \cref{t:Ai}. 
Since $i>0$, it follows that $ab\not\in B_0$, and hence $f(ab) = \ee$. 
Therefore $f$ is a homomorphism. 
\cadd{
The final claim is immediate by the definition of $f$.
}
\end{proof}

\begin{theorem}\label{t:SIchar}
    The only subdirectly irreducible {\MKname}s are $\m 2$ and $\McA$.
\end{theorem}
\begin{proof}
It is easily verified that $\m 2$ and $\McA$ are simple (i.e., contain exactly two distinct congruences) and thus subdirectly irreducible, and they are members of $\Mk$ by \cref{t:MinMK}. 
So let $\m A$ be subdirectly irreducible and suppose $\m A \not\cong \m 2$. 
In order to show that $\m A \cong \McA$, by \cref{t:M3uniq} it is enough to verify $|A|=3$, or equivalently, that $\m B_0 \cong \mathbf{2}$ and the set $E:=A\setminus B_0$ is a singleton. 
As $\m A$ is subdirectly irreducible, $\m B_0 \cong \m 2$ is immediate from \cref{t:SIbottomfiber}. 
Moreover, as $\m A \not\cong \m 2$ by assumption, it must be that $E\neq \varnothing$.
So, $\m A\cong \McA $ iff $|E|=1$ iff $|E\times E| = 1$ iff the congruence $\equiv_E:= \Cg{\m A}(E\times E) = \Delta_{\m A}$. 
\cadd{
Since $\m A$ is subdirectly irreducible, we have that ${\equiv_E} = {\Delta_{\m A}}$ iff there exist a congruence ${\kerI}\neq \Delta_{\m A}$ with ${\kerI}\cap {\equiv_E} = \Delta_{\m A}$.
So we may focus on establishing the latter condition.
}

For ${\kerI}$ we take the congruence $\bigcap\{{\kerI_i}: i\in I_E\}$, where $I_E:= \SL_{\m A}\setminus \{0\}$ and ${\kerI_i}:= \Cg{\m A}(0,\Oc{i})$. 
By \cref{t:trivKerI}, ${\kerI_i}\neq \Delta_{\m A}$ for each $i\in I_E$. 
Since $\m A$ is subdirectly irreducible, it therefore follows that ${\kerI} \neq \Delta_{\m A}$, as desired. 
Also, note that as a consequence of \cref{d:kerI} (i.e., \cref{t:poly1}) we have: 
\begin{equation}
\label{eq:kerI}
x\kerI y 
\qquad\iff \qquad 
\Ic{e}\mc x = \Ic{e}\mc y \quad \text{for all $e\in E$.} 
\end{equation}

So it suffices to verify ${\kerI}\cap {\equiv_E} = \Delta_{\m A}$. 
\cadd{
Let $a,b\in A$ be such that $a\kerI b$ and $a\equiv_E b$; we must show $a=b$.
Since $\m B_0 \cong \m 2$, from \cref{t:map2M3} we have $a\equiv_E b$ iff $a = b$ or $a,b\in E$, so we may assume $a,b\in E$. 
As $\kerI$ is a congruence, $\Ic{x}:=x\jc 1$, and $a\kerI b$, we have $\Ic{a}\kerI \Ic{b}$. 
So from \eqref{eq:kerI} and idempotency, $a\in E$ implies $\Ic{a}=\Ic{a}\mc \Ic{a} = \Ic{a}\mc \Ic{b}$. 
By the same argument, $b\in E$ implies $\Ic{b}=\Ic{b}\Ic{a}$. 
Thus $a,b\in E$ yields $\Ic{a}=\Ic{b}$ by \ref{eq:comlocalunits}, and so $a = \Ic{a} a = \Ic{a} b = \Ic{b}b = b$ by \ref{eq:localunits} and \eqref{eq:kerI}. 
}
Consequently, ${\kerI} \cap {\equiv_E} = \Delta_{\m A}$. Therefore $\m A \cong \McA$.
\end{proof}
As every variety of algebras is generated by its subdirectly irreducible members, and $\m 2$ is a subalgebra of $\McA$, we immediately obtain the following as a corollary to \cref{t:SIchar}.
\begin{corollary}\label{t:M3genMK}
    The variety $\Mk$ of {\MKname}s is generated by the algebra $\McA$. 
\end{corollary}

\begin{theorem}[Completeness]\label{t:MisMK}
    The variety $\MC$ of McCarthy algebras coincides with $\Mk$, and is therefore equationally based by any of the equivalent presentations in \cref{t:equivAxioms}.
\end{theorem}
\noindent Consequently, every proposition established above for {\MKname}s holds for McCarthy algebras.
In particular, the only McCarthy algebra in which $1=\nc{1}$ is the trivial algebra from \cref{t:trivialalg}, and from \cref{t:SIchar} the only proper and non-trivial subvariety of $\MC$ is $\BA$. 
\cadd{\begin{corollary}\label{t:subclass}\label{t:McoversBA}
    McCarthy algebras form a subclassical variety of {\iname}s, and cover Boolean algebras in the lattice of subvarieties of {\iname}s
\end{corollary}}
As the algebra $\McA^\mathsf{op}$ is simply the mirror of $\McA$, we obtain the following for the variety $\MC^\mathsf{op}:=\mathsf{V}(\McA^\mathsf{op})$.
\begin{corollary}\label{t:MopChar}
    The variety $\MC^\mathsf{op}$ is equationally based by swapping any ``left/right'' identity with the corresponding ``right/left'' identity in \cref{t:equivAxioms}, and where $\Ic{x}:= 1\jc x$ and $\Oc{x}:= 0\mc x$. 
    Consequently, this variety is also subclassical and forms a cover of Boolean algebras in the lattice of subvarieties of {\iname}s.
\end{corollary}

\section{\cadd{The semilattice decomposition of McCarthy algebras}}\label{sec:decomp}
We now aim to establish a semilattice decomposition theorem for McCarthy algebras. 
It is well known that several classes (not necessarily varieties) of semigroups admit semilattice decomposition theorems: the most prominent examples include Clifford semigroups (which are semilattices of groups) \cite{Clifford1941,petrich1984Book}, bands (semilattices of rectangular bands) \cite{Clifford1941} and commutative semigroups (semilattices of archimedean semigroups) \cite{TamuraKimura}. 
We refer the interested reader to \cite{RS91,RomanowskaPlonka92} for a general (universal) algebraic overview of this kind of decomposition theorem. Despite providing a semilattice decomposition theorem below, our result turns out to be slightly different to the cases mentioned above, as we will briefly explain. In order to introduce the decomposition, recall that a \defem{semilattice direct system} of algebras of type $\tau$ is a triple $\str{\SL,\mathcal{A}, \mathcal{P}}$ where (i)~$\SL = \str{I,\leq}$ is a semilattice; (ii)~$\mathcal{A} =\{\m A_i\}_{i\in I}$ is a family of algebras with disjoint universes and similar type $\tau$; and (iii)~$\mathcal{P}$ is a family of homomorphisms $p_{ij}:\m A_i\to \m A_j$, for any $i\leq j$, such that $p_{ii}$ is the identity map on $\m A_i$ and $p_{jk}\circ p_{ij} = p_{ik}$ whenever $i\leq j\leq k$.
\cadd{
We say an algebra $\m A$ of type $\tau$ \emph{decomposes} into an SDS $\str{\SL,\mathcal{B},\mathcal{P}}$ of type $\tau$ if $A=\bigcup \mathcal{B}$ and for each $n$-ary operation $f\in \tau$, $f^{\m A}{\restriction} B_i^n = f^{\m B_i}$. 
}
\begin{theorem}[Semilattice Decomposition]\label{t:Decomp}
Each McCarthy algebra $\m A$ decomposes into an SDS of Boolean algebras $\str{\SL_{\m A}, \mathcal{B}_{\m A},  \mathcal{P}_{\m A}}$, where
\begin{enumerate}
    \item $\SL_{\m A}$ is the semilattice skeleton of $\m A$ (see \cref{d:SL});
    \item $\mathcal{B}_{\m A}$ is the set of fibers of $\m A$ (see \cref{d:fiber}); 
    \item $\mathcal{P}_{\m A}$ is the set of homomorphisms $p_{ij}:= h_j{\restriction}B_i:\m B_i\to \m B_j$  for $i\leq j$ in $\SL_{\m A}$ (see \cref{t:Ai}).  
\end{enumerate}
\end{theorem}
\begin{proof}
    The set $\SL_{\m A}$ is a semilattice by \cref{t:SL} and the universe of $\m A$ is the disjoint union of its Boolean fibers by \cref{t:fiberdecomp}. 
    Lastly, we must establish  that $p_{ii} = \mathrm{id}_{\m B_i}$, $p_{jk}\circ p_{ij} = p_{ik}$ whenever $i\leq j\leq k$, and the map $p_{ij}$ is a Boolean algebra homomorphism from $\m B_i$ to $\m B_j$.    
    For the first claim, let $x\in B_i$. Then $p_{ii} = h_i(x)=\Ic{i}x = \Ic{x}x = x$ by \cref{t:Ai}, \cref{t:fiberdecomp}(2), and \ref{eq:localunits}. So $p_{ii} = \mathrm{id}_{\m B_i}$. 
    Next, suppose $i\leq j\leq k$. 
    Then $p_{jk}\circ p_{ij} = p_{ik}$, as for each $x\in B_i$,
    $$(p_{jk}\circ p_{ij})(x)=p_{jk}(\Oc{j}\jc x) = \Oc{k} \jc \Oc{j} \jc x = \Oc{k\vee j}\jc x = \Oc{k} \jc x = \hc_k(x) = p_{ik}(x). $$
    
    Lastly, suppose $i\leq j$.
    The fact that $p_{ij}$ is a homomorphism from $\m B_i$ to $\m B_j$ essentially follows from \cref{t:Ai} and the Second Isomorphism Theorem (cf. \cite[Thm.~4.10]{mckenzie}). 
    Indeed, we have that $h_i$ maps $\m A$ surjectively onto $\m A_i$ and, as $i\leq j$, we have $\ker h_i = {\kerI_i} \subseteq {\kerI_j} = \ker h_j$ (see \cref{rem: inckerns}). So there exists a unique homomorphism $g:\m A_i\to \m A_j$ satisfying $h_j = g\circ h_i$.
    \[
    \begin{tikzcd}
    \m A \arrow[r,"h_j"] \arrow[d,tail, twoheadrightarrow,swap,"h_i"]  & \m A_j\\
    \m A_i \arrow[ru,dashed,"g"] &
    \end{tikzcd}
    \qquad (\ker h_i \subseteq \ker h_j)
    \]
    So $p_{ij}:= h_{j}{\restriction}B_i = (g \circ h_i){\restriction}B_i= g \circ (h_i{\restriction}B_i) = g\circ \mathrm{id}_{\m B_i}$. 
    Thus $p_{ij}$ is a homomorphism as it is a composition of them. 
    That the range $p_{ij}[B_i]\subseteq B_j$ follows from the fact that homomorphisms must map constants to constants (i.e., $\Ic{i}\mapsto\Ic{j}$) and, \crev{since $B_{i} = [0_i , 1_i]$ then, for every $x\in B_{i} $, we have that $p_{ij}(x)\in [0_{j},1_{j}] = B_j$.} 
    Hence $p_{ij}$ is a homomorphism from $\m B_i$ to $\m B_j$.
    \cadd{It is clear then that $\m A$ decomposes into $\str{\SL_{\m A},\mathcal{B}_{\m A}, \mathcal{P}_{\m A} }$.}
\end{proof}

The semilattice decomposition result stated above essentially differs from the previously recalled ones (for Clifford semigroups, bands and commutative semigroups): indeed, the latter are examples of a general construction known as \emph{semilattice sum} (see \cite{RS91,RomanowskaPlonka92}), where the indexing semilattice, called the semilattice replica, can be obtained as a quotient of the starting algebra modulo an opportune congruence. 
However, it is not difficult to check that the semilattice skeleton of any McCarthy algebra $\mathbf{A}$ is, in general, not a quotient of $\mathbf{A}$ (another example of a semilattice decomposition which is not a semilattice sum can be found in \cite[Prop. 9]{BalRes}).\footnote{\cadd{Using \cref{t:trivialalg}, it is readily verified that the semilattice skeleton $\SL_{\m A}$ is obtained as a quotient of a McCarthy algebra $\m A$, in the sense of \cite{RS91}, iff $\SL_{\m A}$ is the trivial semilattice, i.e., that $\m A$ is a Boolean algebra.}} 

\begin{remark}
 Semilattice direct systems are usually related to the theory of P{\l}onka sums. 
 Despite the fact that the introduction of this theory goes beyond the purpose of the present work (standard references include \cite{RomanowskaPlonka92,Plonka1984nullary,Bonziobook}), we can simply say that, roughly speaking, the P{\l}onka sum allows one to (re)construct an algebra starting from a semilattice direct system of (similar) algebras (by taking the disjoint union of the universes of the fibers and opportunely define operations via homomorphisms connecting the fibers). 
 It is not difficult to see that the P{\l}onka sum cannot be used to obtain a McCarthy algebra starting from a direct system of Boolean algebras (one reason is that the variety whose members are P{\l}onka sums of Boolean algebras is the variety of involutive bisemilattices \cite{Bonzio16SL}, another is that the variety of $\mathsf{M}$ of McCarthy algebra satisfy strongly irregular identities, e.g., \ref{eq:leftabs}). 
 Moreover, also the reconstruction of algebras out of a semilattice sum holds in general for \emph{regular} varieties \cite{RS91}, i.e., varieties satisfying regular identities only ($\varphi\approx\psi$ where $\varphi$ and $\psi$ are terms where exactly the same variables appear). 
 For these reasons, it makes sense to introduce a suitable construction in order to recover a McCarthy algebra starting from a semilattice direct system of Boolean algebra. We address this problem in future work.      
\end{remark}

\subsection{\cadd{McCarthy algebras with 2-element fibers}}
In this section, we show how to obtain a McCarthy algebra from of an arbitrary $\bot$-semilattice (i.e., a semilattice with least element $\bot$).

Fix $\m I = \str{\SL, \vee, \bot}$ a $\bot$-semilattice.
For each $i\in \SL$, let $\m B_i$ be a copy of the $2$-element Boolean algebra with universe $B_i=\{\Oc{i},\Ic{i}\}$.  
Let $\SL[\m 2]:=\bigcup_{i\in \SL} B_i$ be their disjoint union. 
We will show that $\SL[\m 2]$ is the universe of some McCarthy algebra $\m I[\m 2]$ whose semilattice skeleton is isomorphic to $\m I$.

Before defining the operations of $\m I[\m 2]$, we adopt a notational convenience. 
For each $i\in \SL$, let $(\_)_i$ denote the isomorphism from $\m 2 \to \m B_i$. 
In this way, it obvious then that $z\in \SL[\m 2]$ iff $z \in B_i$ for some $i\in \SL$ iff $z = x_i := (x)_i$ for some $x\in \{0,1\}$ and $i\in \SL[\m 2]$.
Additionally, let us define the function $\acts: \SL[\m 2]\times \SL\to \SL$ via 
\begin{equation*}\label{eq:I2act}
x_i \acts j =\left\{ 
\begin{array}{c l}
    i & \text{if $x=0$} \\
    i\vee j & \text{if $x = 1$}
\end{array} \right.
\end{equation*}
For the sake of discernment in what follows, we denote the signature of $\m 2$ using the symbols $\vee,\land,\neg,0,1$. 

By $\m I[\m 2]$ let us denote the algebra $\str{\SL[\m 2], \mc,\nc{},\Ic{\bot}}$, where operations $\nc{}$ and $\mc$ are defined via 
\begin{align*}\label{eq:I2ops}
    \nc{x_i} &:= (\neg x)_i\\
    x_i\mc y_j&:= (x\land y)_{x_i \acts j}
\end{align*} 
where $x_i,y_j\in \SL[\m 2]$. 
In other words, $x_i\mc y_j$ takes value $\Oc{i}$ if $x=0$, else $y_{i\vee j}$ if $x=1$.

It is not difficult to see that the following holds:
\begin{equation}\label{eq:I2ass}
    x_i\acts (y_j\acts k) = (x_i\mc y_j)\acts k.
\end{equation}
Indeed, if $x=0$ then both expressions take value $i$ (as $\Oc{i}y_j = \Oc{i}$); if $x=1$ and $y=0$ then both expressions take value $i\vee j$ (as $\Ic{i}\Oc{j}=\Oc{i\vee j}$); else $x=y=1$ and both expressions take value $i\vee j \vee k$ (as $\Ic{i}\Ic{j}=\Ic{i\vee j}$).

It is also easy to see that $\m I[\m 2]$ is a unital i-groupoid. 
Indeed, that $\Ic{\bot}$ is a two-sided unit for $\mc$ is immediate since $\bot$ is the least element in $\SL$ and $1$ is a unit for $\land$, while $\nc{}$ is an involution over $\SL[\m 2]$ since $\neg$ is an involution.
Moreover, $\Oc{\bot} = \nc{\Ic{\bot}}$ and the De Morgan dual $\jc$ for $\mc$ is computed via
$$x_i\jc y_j = (x\vee y)_{\nc{x_i}\acts j} $$
i.e., $x_i\jc y_j$ takes value $y_{i\vee j}$ if $x=0$, else $\Ic{i}$ if $x=1$. 
Hence $\Ic{x_i}:= x_i \jc \Ic{\bot} = \Ic{i}$ and $\Oc{x_i} = x_i\mc \Oc{\bot} = \Oc{i}$.

\begin{proposition}\label{t:Mc from SL}
    $\m I[\m 2]$ is a McCarthy algebra for any $\bot$-semilattice $\mathbf{I}=\str{\SL,\vee,\bot}$.
\end{proposition}
\begin{proof}
    For what follows, let $x_i,y_j,z_k\in \SL[\m 2]$; i.e., fix $i,j,k\in \SL$ and $x,y,z\in \{0,1\}$. 
    First we show that $\m I[\m 2]$ is an {\iname}. 
    As described above, it is clear that $\m I[\m 2]$ is a unital i-groupoid. 
    That $\mc$ is associative is nearly immediate from \cref{eq:I2ass}. Indeed, from the definition of $\mc$, observe that
    $$x_i \mc y_jz_k = x_i \mc (y\land z)_{y_j\acts k} = (x\land y\land z)_{x_i\acts (y_j\acts k)}
    \quad\text{and}\quad
    x_iy_j\mc z_k = (x\land y)_{x_i\acts j}\mc z_k =(x\land y\land z)_{x_iy_j\acts k}
    $$
    and thus $x_i \mc y_jz_k = x_iy_j\mc z_k$ since $x_i\acts (y_j\acts k) = x_i y_j\acts k$.
    That $\mc$ is idempotent is immediate since $x_i\acts i = i$ and $\land$ is idempotent. 
    Hence $\m I[\m 2]$ is an {\iname}.

    To show $\m I[\m 2]$ is a McCarthy algebra, we show it satisfies \cref{wd:div,wabs:lann,wc:lu} from \cref{t:equivAxioms}.
    \Cref{wc:lu} \eqref{eq:comlocalunits} holds since $\Ic{x_i}\Ic{y_j} = \Ic{i}\Ic{j}=\Ic{i\vee j} = \Ic{j}\Ic{i} = \Ic{y_j}\Ic{x_i}.$
    \Cref{wabs:lann} \eqref{eq:leftann} holds since $\Oc{\bot}\mc x_i = \Oc{\bot}$ by definition.
    For \cref{wd:div}, we must verify that \ref{eq:divis} and \ref{eq:localdecomp} hold.
    The former is nearly immediate from \cref{eq:I2ass}, idempotency, and the fact that it holds in Boolean algebras: 
    $$
    x_i \mc (\nc{x_i}\jc y_j) 
    = x_i \mc (\neg x \vee y)_{x_i\acts j} 
    = (x\land (\neg x \vee y))_{x_i\acts (x_i\acts j)}
    = (x\land y)_{x_i \acts j}
    = x_i\mc y_j.
    $$
    Lastly, for \ref{eq:localdecomp}, first observe that $(x_i\jc \nc{x_i})y_j = \Ic{i}\mc y_j = y_{i\vee j}$, while
    $$ x_iy_j \jc \nc{x_i}y_j =
    \left\{
    \begin{array}{ll}
       (\text{for $x=0$}) & \Oc{i}y_j \jc \Ic{i}y_j = \Oc{i} \jc y_{i\vee j} = y_{i\vee (i\vee j)} = y_{i\vee j}  \\
       (\text{for $x=1$}) &\Ic{i}y_j \jc \Oc{i}y_j = y_{i\vee j}\jc \Oc{i} = y_{(i\vee j)\vee i} = y_{i\vee j}  
    \end{array}
    \right.
    $$
    completing our claim. Therefore $\m{I}[\m 2]$ is a McCarthy algebra.
    \end{proof}

\begin{corollary}\label{t:every SL appears}
    Every $\bot$-semilattice appears as the semilattice skeleton of some McCarthy algebra.
\end{corollary}

For a $\bot$-semilattice $\SL$, let $\SL_{\top}$ be the semilattice obtain by adding a greatest element $\top$. 
It is not difficult to see that the relation ${\equiv_\top}:=\Delta_{\m I_\top}\cup (B_\top\times B_\top)$ is a congruence over $\m I_\top[\m 2]$. 
Indeed, $\equiv_\top$ is clearly an equivalence relation that is compatible with the involution $\nc{}$. 
That it is compatible with multiplication holds since, for each $i\leq \top $ and $x,y\in \{0,1\}$, $x_\top \mc y\in B_\top$, $\Oc{i}\mc x_\top = \Oc{i}$, and $\Ic{i}\mc x_\top = x_\top \in B_\top$. 
Defining $\m I[\m 2]_\ee: = \m I[\m 2]/{\equiv_\top}$ and $\ee:=[\Oc{\top}]_{\equiv_\top} = \{\Oc{\top},\Ic{\top} \}$, we obtain the following.
\begin{proposition}
    For any $\bot$-semilattice $\SL$, $\m I[\m 2]_\ee$ is a McCarthy algebra with semilattice skeleton $\SL_\top$, where $B_\top = \{\ee\}$ and $\m B_i\cong \m 2$ for each $i\in \SL$.
\end{proposition}

\subsection{\cadd{Characterizing McCarthy algebras as decorated posets}} 
Fix a McCarthy algebra $\m A$. 
Recall that $\mathcal{A}:=\str{A,\leq}$ is a poset, where ${\leq}:={\leq_\jc}$; call $\mathcal{A}$ the \defem{induced poset} of $\m A$. 
We will call $\str{\mathcal{A},\SL}$ the \defem{$s\ell$-decorated poset} of $\m A$, where $\SL:=\SL_{\m A}=\{x\in A: x = x\mc 0 \}$ is the semilattice skeleton of $\m A$.
In this section, we show that the operations of a McCarthy algebra are fully determined by its $s\ell$-decorated poset, in the sense that the operations are \emph{expressible} in the first-order theory of posets with an additional unary predicate (whose interpretation is membership in the set $\SL$).

We first verify that the least upper bound of two elements $i,j\in\SL$ exists in $\mathcal{A}$, and coincides with the element $i\vee j\in \SL$ (where we recall that the operation $\vee$ is the restriction of $\jc$ to $\SL$).
\begin{lemma}\label{t:SLjoins}
For $i,j\in \SL$, the element $i\vee j\in \SL$ is the least upper bound of $i$ and $j$ in $\mathcal{A}$.
\end{lemma}
\begin{proof}
    By \cref{t:BelowZeros}, the set $\SL$ is a convex set in $\mathcal{A}$; i.e., if $i,j\in \SL$ with $i\leq j$, then $[i,j]:=\{x\in A: i\leq x\leq j \} \subseteq \SL$. 
    Since $\str{\SL,\vee}$ is a semilattice, where $\vee$ is the restriction of $\jc$ to $\SL$, and ${\leq}:={\leq_{\jc}}$, the claim follows. 
\end{proof}

In this way, there is no harm in using the same symbol $\vee$ (join) for least upper bounds in $\mathcal{A}$, when they exist. 
Similarly, we use the symbol $\land$ (meet) to denote the greatest lower bound in $\mathcal{A}$, when they exist.
We recall from \cref{t:fiberdecomp} that an interval $[\Oc{i},\Ic{i}]$ is the universe of a Boolean algebra whose additive/join operation is the restriction of $\jc$, and hence the subposet $\str{[\Oc{i},\Ic{j}],\leq}$ is a Boolean lattice. 
The following is therefore immediate.
\begin{lemma}\label{t:Bsublattice}
    Meets and joins between elements in a fiber $B_i$ exist in $\mathcal{A}$ and coincide with the products and sums, respectively, in $\m B_i$.
\end{lemma}

First we show that the involution $\nc{}$ is expressible in $\str{\mathcal{A},\SL}$ by showing that the term operations $\Oc{x}:=x\mc 0$ and $\Ic{x}:= x\jc 1$ are, and thus also membership in a Boolean fiber.
 
\begin{lemma}\label{t:ExpFiber}
    Let $x\in A$. 
    Then $\Oc{x} = \max\{i\in \SL: i\leq x \}$ and $\Ic{x}=\max\{y\in A: \Oc{y} = \Oc{x}\}$.
    Consequently, $\nc{x} = \max\{ y\in [\Oc{x},\Ic{x}]: x\land y = \Oc{x}\}$.
\end{lemma}
\begin{proof}
    Let $X = \{i\in \SL: i\leq x \}$. 
    That $\Oc{x}\in X$ follows by \ref{eq:localunits}, and if $i\in X$ then $i = i\mc 0\leq x\mc 0 = \Oc{x}$. 
    Hence $\Oc{x}=\max X$. 
    For the second claim, let $Y=\{y\in A: \Oc{y} = \Oc{x}\}$. 
    By \cref{t:fiberdecomp}(2), $Y = B_i=[\Oc{i},\Ic{i}]$ where $i=\Oc{x}$, and $\m B_i$ is a Boolean algebra. 
    Since $\Ic{i} = \Ic{x}$, it follows that $\Ic{x} = \max{Y}$.
    Lastly, as the negation of an element $a$ in a Boolean lattice is the largest element whose meet with $a$ is the bottom, we have $\nc{x} = \max\{ y\in B_i: x\mc y = \Oc{i}\}$, and by \cref{t:Bsublattice}, the final claim follows.
\end{proof}

Towards establishing $\mc$ is expressible in $\str{\mathcal{A},\SL}$, we adopt the following notational convenience.
Define $\acts: A\times \SL_{\m A}\to \SL_{\m A}$ via $a\acts i:= a\mc \Oc{i}$. 
As we see below, the function $\acts$ behaves as an ``action'' of $\m A$ on the semilattice $\SL_{\m A}$.  
We remind the reader that multiplication is order-preserving (\cref{t:CompatibleOrder}).
\begin{lemma}\label{t:Action}
    For $x,y\in A$ and $i, j\in \SL$, with $i:=\Oc{x}$ and $j:=\Oc{y}$, $xy\in B_{x\acts j}$ and $x\acts i=i\leq x\acts j \leq i\vee j$. Moreover, for $k\in \SL$, $x\acts(y\acts k)=xy\acts k$ and $x\acts(j \vee k)=x\acts j\vee x\acts k$.
\end{lemma}
\begin{proof}
    Recall from \cref{t:fiberdecomp}(2), $x\in B_k$ iff $\Oc{x} = \Oc{k}$. 
    So $xy\in B_{x\acts j}$ since $\Oc{xy} = xy\mc0=x\cdot \Oc{y} = x\cdot \Oc{j} = \Oc{x\acts j} $, as $j=\Oc{y}$. 
    Since $i=\Oc{x}$, it is clear that $x\acts i = i$ by idempotency. 
    That $x\acts j \in [i,i\vee j]$ is nearly immediate from the order-preservation of $\mc$. 
    Indeed, as $0\leq \Oc{y}$ and $x\leq \Ic{x}$ always hold in {\iname}s, we have
    $\Oc{x}=x\mc 0 \leq x\mc\Oc{y} \leq \Ic{x}\mc\Oc{y} = \Oc{x} \jc \Oc{y}$,
    where the final inequality is an instance of \ref{eq:coherence}; i.e., $i\leq x\acts j\leq i \vee j$.
    The final two claims follow from associativity of $\mc$ and \ref{eq:leftdist}, respectively.
\end{proof}
\begin{lemma}\label{t:actsV}
       Let $x \in A$, $j\in \SL$, and set $i:= \Oc{x}\in \SL$.  
       Then $x\acts j= \max\{k\in \SL:i\leq k\leq i \vee j  \text{ and } \nc{x} \leq \Ic{k} \}$.
\end{lemma}
\begin{proof}
    Set $X: = \{k\in [i,i\vee j]: \nc{x} \leq \Ic{k}\}$. 
    First we show $x\acts j\in X$. 
    From \cref{t:Action}, we have $x\acts j\in [i,i\vee j]$ and $x\Ic{j}\in B_{x\acts j}$.
    As fibers are closed under $\nc{}$, the latter yields also $\nc{x}\jc \Oc{j} = \nc{(x\Ic{j})} \in B_{x\acts j}$, in particular, $\nc{x}\jc \Oc{j} \leq \Ic{x\acts j}$. 
    So $\nc{x} \leq \nc{x}\jc \Oc{j}\leq \Ic{x\acts j}$, thus $x\acts j \in X$.
    Finally, we show $k\leq x\acts j$ for any $k \in X$.  
    Notice that $\nc{x}\leq \Ic{k}$ iff $x\mc \Oc{k} = \Oc{k}$, i.e., iff $x\acts k = k$. 
    So for $k\in X$, from fact that $k\leq i\vee j$, the above observations, and \cref{t:Action} yield
     $k = x\acts k \leq x\acts k \vee x\acts(i\vee j) = x\acts (k\vee i \vee j) = x\acts (i\vee j) = x\acts i \vee x\acts j = i \vee x\acts j = x\acts j$.
\end{proof}

\begin{lemma}\label{t:a1j}
       Let $x,y\in A$, and set $i:= \Oc{x}$ and $j:=\Oc{y}$.  
       Then $\nc{x}\vee \Oc{x\acts j}$ exists and coincides with the element $\nc{x} \jc \Oc{y} \in B_{x\acts j}$.
       Consequently, $x\Ic{y} = \nc{(\nc{x}\vee \Oc{x\acts j})} \in B_{x\acts j}$.
\end{lemma}
\begin{proof}
    Since $\jc$ restricted to $\downset{\Ic{x\acts j}}$ is a join-semilattice, the element $\nc{x}\jc \Oc{x\acts j}$ is the join of $\nc{x}$ and $\Oc{x\acts j}$, since $\nc{x},\Oc{x\acts j}\leq \Ic{x\acts j}$. 
    Thus $\nc{x}\jc \Oc{x\acts j}=\nc{x}\jc \Oc{x\acts j}= \nc{x}\jc x\Oc{j}= \nc{x}\jc \Oc{j}$, where the last equality holds by \ref{eq:divis}, and so we locate $x\Ic{j}$ by taking its negation. 
\end{proof}

\begin{lemma}\label{t:a1j<1ja}
For $i,j\in \SL$ and $x\in B_i$,  $x\Ic{j}\vee \Oc{i\vee j}$ exists and coincides with the element $\Ic{j}x \in B_{i\vee j}$.
\end{lemma}
\begin{proof}
    First, note that $x\Ic{j}\leq \Ic{j}x$ holds via \cref{rem:a1j<1ja}.
    Now, $\Ic{j}x\in B_{i\vee j}$ since $\Ic{j}\mc i = i\vee j$, and thus $\Oc{i\vee j}\leq \Ic{j}a$ as well. 
    Thus $\Ic{j}x$ is a upper bound for $x\Ic{j}$ and $\Oc{i\vee j}$.
    Suppose let $c$ be any other upper bound, and note that $\Oc{j} \leq c$ since $j\leq i \vee j$. 
    Then using \ref{eq:leftreg}, \ref{eq:coherence}, and $x\Ic{j}\leq c$, we find $\Ic{j}x \jc c = \Ic{j}a\Ic{j} \jc c = \Oc{j}\jc x\Ic{j}\jc c = \Oc{j} \jc c = c$.
    Hence $\Ic{j}x \leq c$. 
\end{proof}

\begin{lemma}\label{t:ab as a meet}
    For $a,b\in A$, $a\Ic{b}\land \Ic{b}ab$ and $a\Ic{b}\land \Ic{a}b$ exists and both coincide with $a\mc b$.
\end{lemma}
\begin{proof}
    First we verify $ab\leq a\Ic{j},\Ic{j}ab$. 
    Indeed, $ab\leq a\Ic{j}$ since $b\leq \Ic{j}$, and thus using \ref{eq:unitcoh}, left-regularity, and \cref{rem:a1j<1ja} we find
    $ab = a\Ic{i}b\Ic{j}=a\Ic{i}b\Ic{i}\Ic{j} = ab\Ic{i\vee j} \leq \Ic{i\vee j}ab = \Ic{j}ab$.
    Also, as $a\leq \Ic{i}$, we note that $\Ic{j}ab = \Ic{j}\Ic{i}b = \Ic{i}\Ic{j}b = \Ic{i}b$ using \ref{eq:comlocalunits} and \ref{eq:localunits}.
    Now, suppose $c\leq a\Ic{j},\Ic{i}b$. 
    Observing first that that $ab = a\Ic{i}\Ic{j}b = a\Ic{j}\Ic{i}b$ by \ref{eq:localunits} and \ref{eq:comlocalunits}, we find
    $ c\jc ab  = c\jc  (a\Ic{j}\mc \Ic{i} b) = (c\jc a\Ic{j})(c\jc \Ic{i}b) = a\Ic{j}\mc \Ic{i} b = ab$ by \ref{eq:leftdist}. 
\end{proof}

\begin{remark}\label{rem:ab}
Locating $a\mc b$ in the $s\ell$-decorated poset, as depicted in \Cref{fig:McMult}, therefore proceeds as follows:
\begin{enumerate}
    \item 
    Locate the elements $\Oc{i},\nc{a}\in  B_i$ and $\Oc{j},\nc{b}\in B_j$. (See \cref{t:ExpFiber})
    \item Having found $\Oc{i},\Oc{j}$, locate the element $\Oc{i\vee j}=\Oc{i}\vee\Oc{j}\in B_{i\vee j}$. (See \cref{t:SLjoins}) 
    \item Having found $\Oc{i}$ and $\Oc{i\vee j}$, locate $\Oc{a\acts j} \in B_{a\acts j}$, and similarly find $\Oc{b\acts i} \in B_{b\acts i}$. (See \cref{t:actsV})
    \item From $\nc{a}$ and $\Oc{a\acts j}$, locate $a\Ic{j}=\nc{(\nc{a}\vee\Oc{a\acts j})}\in B_{a\acts j}$, and similarly $b\Ic{i}\in B_{b\acts i}$. (See \cref{t:a1j})
    \item From $b\Ic{i}$ and $\Oc{i\vee j}$, locate $\Ic{i}b=b\Ic{i}\vee\Oc{i\vee j}\in B_{i\vee j}$. (See \cref{t:a1j<1ja}) 
    \item The product $ab\in B_{a\acts j}$ is the meet of $a\Ic{j}\in B_{a\acts j}$ and $\Ic{j}b \in B_{i\vee j}$. (See \cref{t:ab as a meet}).
\end{enumerate}
\begin{figure}[h]
    \centering
    \begin{tikzpicture}[framed, scale = 1.2, every node/.style={inner sep=.5pt}, every label/.style={gray}]
        \node[label = {[xshift=0,yshift = -1em]\scriptsize $\Oc{i}$}] (0i) at (-3,0) {};
        \node[label = {[xshift=0,yshift = -1em] \color{black} $a$}] (a) at (-4,.5) {\scriptsize$\bullet$};
        \node[label = {[xshift=.5em,yshift = -.5em]\scriptsize $\nc{a}$}] (a') at (-3,.5) {\scriptsize$\bullet$};
        \node[label = {\scriptsize $\Ic{i}$}] (1i) at (-4,1) {\scriptsize$\bullet$};
        \draw (-3,0) -- (-3,.5) -- (-4,1) -- (-4,.5) -- (-3,0);

        \node[label = {[xshift=.1em,yshift = -1.2em]\scriptsize $\Oc{a\acts j}$}] (a0j) at (-1.5,.5) {};
        \node[label = {[xshift=-1.2em,yshift = -.5em]\scriptsize $\nc{a}{\jc}\Oc{j}$}] (a'j) at (-2.5,1) {\scriptsize$\bullet$};
        \node[label = {[xshift = -.8em,yshift=-.7em]\scriptsize $a\Ic{j}$}] (aj) at (-1.5,1.5) {\scriptsize$\bullet$};
        
        \node[label = {\scriptsize $\Ic{a\acts j}$}]  (a1j) at (-2.5,2) {\scriptsize$\bullet$};
        \draw (-1.5,.5) -- (-2.5,1) -- (-2.5,2) -- (-1.5,1.5)  -- (-1.5,.5);

        \node[label = {[xshift=0,yshift = -1.2em]\scriptsize $\Oc{j}$}]  (0j) at (3,0) {};
        \node[label = {[xshift=0,yshift = -1.2em] \color{black} $b$}] (b) at (4,.5) {\scriptsize$\bullet$};
        \node[label = {[xshift=-.5em,yshift = -.7em]\scriptsize $\nc{b}$}] (b') at (3,.5) {\scriptsize$\bullet$};
        \node[label = {\scriptsize $\Ic{j}$}] (1j) at (4,1) {\scriptsize$\bullet$};
        \draw (3,0) -- (3,.5) -- (4,1) -- (4,.5) -- (3,0);

        \node[label = {[xshift=0,yshift = -1.2em]\scriptsize $\Oc{b\acts i}$}] (b0i) at (1.5,.5) {};
        \node[label = {[xshift=1.3em,yshift = -.5em]\scriptsize $\nc{b}{\jc}\Oc{i}$}] (b'i) at (2.5,1) {\scriptsize$\bullet$};
        \node[label = {[xshift = .8em,yshift=-.75em]\scriptsize $b\Ic{i}$}] (bi) at (1.5,1.5) {\scriptsize$\bullet$};
        \node (b1i)[label = {\scriptsize $\Ic{b\acts i}$}] at (2.5,2) {\scriptsize$\bullet$};
        \draw (1.5,.5) -- (2.5,1) -- (2.5,2) -- (1.5,1.5)  -- (1.5,.5);

        \draw[dashed] (-1,2) -- (-1,1.5);
        \draw[dashed] (1,2) -- (1,1.5);
        \draw[dashed]  (1,2) -- (1,1.5)--(0,1) -- (-1,1.5) -- (-1,2) ;
        \draw[dashed] (-1,2) -- (0,2.5) -- (1,2);
        \draw[dashed] (-1,1.5) -- (0,2) -- (1,1.5);
        \draw[dashed] (0,2) -- (0,2.5);
        \draw (0,1) -- (0,1.5);
        \draw[preaction={draw, white, line width=2pt}] (-1,2) -- (0,1.5) -- (1,2);
        \node[label = {\scriptsize $\Ic{i\vee j}$}] (1ij) at (0,2.5) {\scriptsize$\bullet$}; 
        \node[label = {[xshift=0,yshift = -1.2em]\scriptsize $\Oc{i\vee j}$}] (0ij) at (0,1) {};%
        \node[label = {\scriptsize $\Ic{j}ab$}] (1ijab) at (0,1.5) {\scriptsize$\bullet$};
        \node[label = {\scriptsize $\Ic{j}a$}] (1ja) at (-1,2) {\scriptsize$\bullet$};
        \node[label = {\scriptsize $\Ic{i}b$}] (1ib) at (1,2) {\scriptsize$\bullet$};

        \draw[gray] (0i) -- (a0j) -- (0ij) -- (b0i) -- (0j);
        \draw[gray] (a') -- (a'j);
        \draw[gray] (aj) -- (1ja);
        \draw[gray] (b') -- (b'i);
        \draw[gray] (bi) -- (1ib);

        \node[label = {[xshift=-.75em,yshift = -1em] \color{black}$ab$}]  (ab) at (-1.5,1) {$\bullet$}; 
        \node (ba) at (1.5,1) {};
        \draw[gray,preaction={draw, white, line width=2pt}] (ab) -- (1ijab);
        \node (ab') at (-2,1) {};
        \node (ba') at (2,1) {};
    \tikzset{SLNode/.style={circle, draw, fill=white, minimum size=3pt, inner sep=0pt}}
        \node[SLNode] at (0i) {}; 
        \node[SLNode] at (0j) {}; 
        \node[SLNode] at (b0i) {}; 
        \node[SLNode] at (a0j) {}; 
        \node[SLNode] at (0ij) {}; 
    \end{tikzpicture}
    \caption{Locating a product $a\mc b$ in the $s\ell$-decorated poset of $\m A$. 
    A node $\circ$ indicates membership in $\SL_{\m A}$.
    }
    \label{fig:McMult}
\end{figure}
\end{remark}

\begin{theorem}\label{t:sl-skeleton-char}
    A McCarthy algebra is fully determined by its $s\ell$-decorated poset. 
\end{theorem}
\begin{proof}
    The involution $\nc{}$ is expressible via \cref{t:ExpFiber}, and as $0$ ($=\nc{1}$) is the least element in any McCarthy algebra, so too is the constant $1$. 
    Multiplication $\mc$ is expressible by the steps in \cref{rem:ab}.
\end{proof}
Call two $s\ell$-decorated posets $\str{\mathcal{A}_1,\SL_1}$ and $\str{\mathcal{A}_2,\SL_2}$ \emph{isomorphic} if there is an order-isomorphism $\phi: \mathcal{A}_1\to \mathcal{A}_2$ such that the restriction $\phi{\restriction}\SL_1$ is a bijection onto $\SL_2$.
\begin{corollary}\label{t:isopo}
Two McCarthy algebras are isomorphic iff their $s\ell$-decorated posets are isomorphic.
\end{corollary}
\begin{proof}
    The forward direction is immediate from the fact the $x\leq y$ iff $x\jc y=y$, and $x\in \SL$ iff $x=x\mc 0$.
    The reverse direction holds since order-isomorphisms preserve existing suprema and infima. 
    Since the operations are expressible via existing suprema/infima and the use of membership in the skeleton, via Lemmas~\ref{t:ExpFiber}--\ref{t:ab as a meet}, any $s\ell$-decorated isomorphism preserves the operations, making it an {\iname} isomorphism. 
\end{proof}
Utilizing \texttt{Mace4} \cite{P9M4}, the \emph{Fine Spectrum}, i.e., number of non-isomorphic models, for McCarthy algebras up to cardinality $14$ is given in \Cref{tab:finespec} below. 
Moreover, exploiting \cref{t:sl-skeleton-char} and \cref{t:isopo}, in \Cref{fig:models10} below we display those models up to size $10$ via their induced $s\ell$-decorated poset.
\begin{table}[ht]
    \centering
    \begin{tabular}{|r|c|c|c|c|c|c|c|c|c|c|c|c|c|c|}
    \hline
        $n=$ &1& 2 & 3 & 4 &5 &6&7&8&9&10&11&12&13&14 \\ \hline
         $\#$ of models& 
         1&1&1&2&1&3&2&6&6&12&16&35&56&111 \\ \hline
    \end{tabular}
    \caption{The Fine Specturum for McCarthy algebras up to size $14$}
    \label{tab:finespec}
\end{table}
\begin{figure}[ht]
    \centering
\tikzset{
  Z/.style={
    circle,
    draw,
    fill=white,
    minimum size=4pt,
    inner sep=0pt
  },
  N/.style={
    circle,
    fill=black,
    minimum size=4pt,
    inner sep=0pt
  }
}

\begin{tikzpicture}[>=latex,line join=bevel,  scale = .2]
\node[N] [Z] (0) at (11.0bp,11.0bp)   {}; 
%
\end{tikzpicture}
\qquad\quad
\begin{tikzpicture}[>=latex,line join=bevel,  scale = .2]
\node[N] [Z] (0) at (11.0bp,11.0bp)   {}; 
  \node[N] [N] (1) at (11.0bp,69.0bp)   {}; 
  \draw [] (1) ..controls (11.0bp,47.921bp) and (11.0bp,32.474bp)  .. (0);
\end{tikzpicture}
\qquad\quad
\begin{tikzpicture}[>=latex,line join=bevel,  scale = .2]
\node[N] [Z]  (0) at (31.0bp,11.0bp)   {}; 
  \node[N] [N] (1) at (11.0bp,69.0bp)   {}; 
  \node[N] [Z] (2) at (51.0bp,69.0bp)   {}; 
  \draw [] (1) ..controls (18.048bp,48.265bp) and (23.836bp,32.06bp)  .. (0);
  \draw [] (2) ..controls (43.952bp,48.265bp) and (38.164bp,32.06bp)  .. (0);
\end{tikzpicture}
\qquad\quad
\begin{tikzpicture}[>=latex,line join=bevel,  scale = .2]
\node[N] [Z]  (0) at (31.0bp,11.0bp)   {}; 
  \node[N] [N] (1) at (31.0bp,127.0bp)   {}; 
  \node[N] [N] (2) at (11.0bp,69.0bp)   {}; 
  \node[N] [N] (3) at (51.0bp,69.0bp)   {}; 
  \draw [] (1) ..controls (23.952bp,106.27bp) and (18.164bp,90.06bp)  .. (2);
  \draw [] (1) ..controls (38.048bp,106.27bp) and (43.836bp,90.06bp)  .. (3);
  \draw [] (2) ..controls (18.048bp,48.265bp) and (23.836bp,32.06bp)  .. (0);
  \draw [] (3) ..controls (43.952bp,48.265bp) and (38.164bp,32.06bp)  .. (0);
\end{tikzpicture}
\quad
\begin{tikzpicture}[>=latex,line join=bevel,  scale = .2]
\node[N] [Z]  (0) at (31.0bp,11.0bp)   {}; 
  \node[N] [N] (1) at (11.0bp,69.0bp)   {}; 
  \node[N] [Z]  (2) at (51.0bp,69.0bp)   {}; 
  \node[N] [N] (3) at (51.0bp,127.0bp)   {}; 
  \draw [] (1) ..controls (18.048bp,48.265bp) and (23.836bp,32.06bp)  .. (0);
  \draw [] (2) ..controls (43.952bp,48.265bp) and (38.164bp,32.06bp)  .. (0);
  \draw [] (3) ..controls (51.0bp,105.92bp) and (51.0bp,90.474bp)  .. (2);
\end{tikzpicture}
\qquad\quad
\begin{tikzpicture}[>=latex,line join=bevel,  scale = .2]
\node[N] [Z]  (0) at (31.0bp,11.0bp)   {}; 
  \node[N] [N] (1) at (11.0bp,69.0bp)   {}; 
  \node[N] [Z]  (2) at (31.0bp,127.0bp)   {}; 
  \node[N] [Z]  (3) at (51.0bp,69.0bp)   {}; 
  \node[N] [N] (4) at (71.0bp,127.0bp)   {}; 
  \draw [] (1) ..controls (18.048bp,48.265bp) and (23.836bp,32.06bp)  .. (0);
  \draw [] (2) ..controls (38.048bp,106.27bp) and (43.836bp,90.06bp)  .. (3);
  \draw [] (3) ..controls (43.952bp,48.265bp) and (38.164bp,32.06bp)  .. (0);
  \draw [] (4) ..controls (63.952bp,106.27bp) and (58.164bp,90.06bp)  .. (3);
\end{tikzpicture}
\qquad\quad
\begin{tikzpicture}[>=latex,line join=bevel,  scale = .2]
\node[N] [Z]  (0) at (51.0bp,11.0bp)   {}; 
  \node[N] [N] (1) at (31.0bp,127.0bp)   {}; 
  \node[N] [N] (2) at (51.0bp,69.0bp)   {}; 
  \node[N] [N] (3) at (11.0bp,69.0bp)   {}; 
  \node[N] [Z]  (4) at (91.0bp,69.0bp)   {}; 
  \node[N] [N] (5) at (81.0bp,127.0bp)   {}; 
  \draw [] (1) ..controls (38.048bp,106.27bp) and (43.836bp,90.06bp)  .. (2);
  \draw [] (1) ..controls (23.952bp,106.27bp) and (18.164bp,90.06bp)  .. (3);
  \draw [] (2) ..controls (51.0bp,47.921bp) and (51.0bp,32.474bp)  .. (0);
  \draw [] (3) ..controls (24.809bp,48.668bp) and (37.446bp,30.976bp)  .. (0);
  \draw [] (4) ..controls (77.191bp,48.668bp) and (64.554bp,30.976bp)  .. (0);
  \draw [] (5) ..controls (70.615bp,106.61bp) and (61.51bp,89.618bp)  .. (2);
  \draw [] (5) ..controls (84.608bp,105.8bp) and (87.413bp,90.087bp)  .. (4);
\end{tikzpicture}
\quad
\begin{tikzpicture}[>=latex,line join=bevel,  scale = .2]
\node[N] [Z]  (0) at (31.0bp,11.0bp)   {}; 
  \node[N] [N] (1) at (11.0bp,69.0bp)   {}; 
  \node[N] [Z]  (2) at (51.0bp,69.0bp)   {}; 
  \node[N] [N] (3) at (51.0bp,185.0bp)   {}; 
  \node[N] [N] (4) at (31.0bp,127.0bp)   {}; 
  \node[N] [N] (5) at (71.0bp,127.0bp)   {}; 
  \draw [] (1) ..controls (18.048bp,48.265bp) and (23.836bp,32.06bp)  .. (0);
  \draw [] (2) ..controls (43.952bp,48.265bp) and (38.164bp,32.06bp)  .. (0);
  \draw [] (3) ..controls (43.952bp,164.27bp) and (38.164bp,148.06bp)  .. (4);
  \draw [] (3) ..controls (58.048bp,164.27bp) and (63.836bp,148.06bp)  .. (5);
  \draw [] (4) ..controls (38.048bp,106.27bp) and (43.836bp,90.06bp)  .. (2);
  \draw [] (5) ..controls (63.952bp,106.27bp) and (58.164bp,90.06bp)  .. (2);
\end{tikzpicture}
\quad
\begin{tikzpicture}[>=latex,line join=bevel,  scale = .2]
\node[N] [Z]  (0) at (31.0bp,11.0bp)   {}; 
  \node[N] [N] (1) at (11.0bp,69.0bp)   {}; 
  \node[N] [Z]  (2) at (31.0bp,127.0bp)   {}; 
  \node[N] [Z]  (4) at (51.0bp,69.0bp)   {}; 
  \node[N] [N] (3) at (31.0bp,185.0bp)   {}; 
  \node[N] [N] (5) at (71.0bp,127.0bp)   {}; 
  \draw [] (1) ..controls (18.048bp,48.265bp) and (23.836bp,32.06bp)  .. (0);
  \draw [] (2) ..controls (38.048bp,106.27bp) and (43.836bp,90.06bp)  .. (4);
  \draw [] (3) ..controls (31.0bp,163.92bp) and (31.0bp,148.47bp)  .. (2);
  \draw [] (4) ..controls (43.952bp,48.265bp) and (38.164bp,32.06bp)  .. (0);
  \draw [] (5) ..controls (63.952bp,106.27bp) and (58.164bp,90.06bp)  .. (4);
\end{tikzpicture}
\qquad\quad
\begin{tikzpicture}[>=latex,line join=bevel,  scale = .2]
\node[N] [Z]  (0) at (51.0bp,11.0bp)   {}; 
  \node[N] [N] (1) at (11.0bp,69.0bp)   {}; 
  \node[N] [Z]  (2) at (71.0bp,127.0bp)   {}; 
  \node[N] [Z]  (3) at (51.0bp,69.0bp)   {}; 
  \node[N] [Z]  (5) at (91.0bp,69.0bp)   {}; 
  \node[N] [N] (4) at (31.0bp,127.0bp)   {}; 
  \node[N] [N] (6) at (111.0bp,127.0bp)   {}; 
  \draw [] (1) ..controls (24.809bp,48.668bp) and (37.446bp,30.976bp)  .. (0);
  \draw [] (2) ..controls (63.952bp,106.27bp) and (58.164bp,90.06bp)  .. (3);
  \draw [] (2) ..controls (78.048bp,106.27bp) and (83.836bp,90.06bp)  .. (5);
  \draw [] (3) ..controls (51.0bp,47.921bp) and (51.0bp,32.474bp)  .. (0);
  \draw [] (4) ..controls (38.048bp,106.27bp) and (43.836bp,90.06bp)  .. (3);
  \draw [] (5) ..controls (77.191bp,48.668bp) and (64.554bp,30.976bp)  .. (0);
  \draw [] (6) ..controls (103.95bp,106.27bp) and (98.164bp,90.06bp)  .. (5);
\end{tikzpicture}
\quad
\begin{tikzpicture}[>=latex,line join=bevel,  scale = .2]
\node[N] [Z]  (0) at (31.0bp,11.0bp)   {}; 
  \node[N] [N] (1) at (11.0bp,69.0bp)   {}; 
  \node[N] [Z]  (2) at (11.0bp,185.0bp)   {}; 
  \node[N] [Z]  (3) at (31.0bp,127.0bp)   {}; 
  \node[N] [Z]  (5) at (51.0bp,69.0bp)   {}; 
  \node[N] [N] (4) at (51.0bp,185.0bp)   {}; 
  \node[N] [N] (6) at (71.0bp,127.0bp)   {}; 
  \draw [] (1) ..controls (18.048bp,48.265bp) and (23.836bp,32.06bp)  .. (0);
  \draw [] (2) ..controls (18.048bp,164.27bp) and (23.836bp,148.06bp)  .. (3);
  \draw [] (3) ..controls (38.048bp,106.27bp) and (43.836bp,90.06bp)  .. (5);
  \draw [] (4) ..controls (43.952bp,164.27bp) and (38.164bp,148.06bp)  .. (3);
  \draw [] (5) ..controls (43.952bp,48.265bp) and (38.164bp,32.06bp)  .. (0);
  \draw [] (6) ..controls (63.952bp,106.27bp) and (58.164bp,90.06bp)  .. (5);
\end{tikzpicture}
\quad

\begin{tikzpicture}[>=latex,line join=bevel,  scale = .2]
\node[N] [Z]  (0) at (51.0bp,11.0bp)   {}; 
  \node[N] [N] (1) at (51.0bp,185.0bp)   {}; 
  \node[N] [N] (3) at (91.0bp,127.0bp)   {}; 
  \node[N] [N] (5) at (51.0bp,127.0bp)   {}; 
  \node[N] [N] (7) at (11.0bp,127.0bp)   {}; 
  \node[N] [N] (2) at (11.0bp,69.0bp)   {}; 
  \node[N] [N] (4) at (51.0bp,69.0bp)   {}; 
  \node[N] [N] (6) at (91.0bp,69.0bp)   {}; 
  \draw [] (1) ..controls (64.809bp,164.67bp) and (77.446bp,146.98bp)  .. (3);
  \draw [] (1) ..controls (51.0bp,163.92bp) and (51.0bp,148.47bp)  .. (5);
  \draw [] (1) ..controls (37.191bp,164.67bp) and (24.554bp,146.98bp)  .. (7);
  \draw [] (2) ..controls (24.809bp,48.668bp) and (37.446bp,30.976bp)  .. (0);
  \draw [] (3) ..controls (77.191bp,106.67bp) and (64.554bp,88.976bp)  .. (4);
  \draw [] (3) ..controls (91.0bp,105.92bp) and (91.0bp,90.474bp)  .. (6);
  \draw [] (4) ..controls (51.0bp,47.921bp) and (51.0bp,32.474bp)  .. (0);
  \draw [] (5) ..controls (37.191bp,106.67bp) and (24.554bp,88.976bp)  .. (2);
  \draw [] (5) ..controls (64.809bp,106.67bp) and (77.446bp,88.976bp)  .. (6);
  \draw [] (6) ..controls (77.191bp,48.668bp) and (64.554bp,30.976bp)  .. (0);
  \draw [] (7) ..controls (11.0bp,105.92bp) and (11.0bp,90.474bp)  .. (2);
  \draw [] (7) ..controls (24.809bp,106.67bp) and (37.446bp,88.976bp)  .. (4);
\end{tikzpicture}
\quad
\begin{tikzpicture}[>=latex,line join=bevel,  scale = .2]
\node[N] [Z]  (0) at (51.0bp,11.0bp)   {}; 
  \node[N] [N] (1) at (11.0bp,127.0bp)   {}; 
  \node[N] [N] (2) at (51.0bp,69.0bp)   {}; 
  \node[N] [N] (3) at (11.0bp,69.0bp)   {}; 
  \node[N] [Z]  (4) at (91.0bp,69.0bp)   {}; 
  \node[N] [N] (5) at (71.0bp,185.0bp)   {}; 
  \node[N] [N] (6) at (51.0bp,127.0bp)   {}; 
  \node[N] [N] (7) at (91.0bp,127.0bp)   {}; 
  \draw [] (1) ..controls (24.809bp,106.67bp) and (37.446bp,88.976bp)  .. (2);
  \draw [] (1) ..controls (11.0bp,105.92bp) and (11.0bp,90.474bp)  .. (3);
  \draw [] (2) ..controls (51.0bp,47.921bp) and (51.0bp,32.474bp)  .. (0);
  \draw [] (3) ..controls (24.809bp,48.668bp) and (37.446bp,30.976bp)  .. (0);
  \draw [] (4) ..controls (77.191bp,48.668bp) and (64.554bp,30.976bp)  .. (0);
  \draw [] (5) ..controls (63.952bp,164.27bp) and (58.164bp,148.06bp)  .. (6);
  \draw [] (5) ..controls (78.048bp,164.27bp) and (83.836bp,148.06bp)  .. (7);
  \draw [] (6) ..controls (51.0bp,105.92bp) and (51.0bp,90.474bp)  .. (2);
  \draw [] (6) ..controls (64.809bp,106.67bp) and (77.446bp,88.976bp)  .. (4);
  \draw [] (7) ..controls (91.0bp,105.92bp) and (91.0bp,90.474bp)  .. (4);
\end{tikzpicture}
\quad
\begin{tikzpicture}[>=latex,line join=bevel,  scale = .2]
\node[N] [Z]  (0) at (31.0bp,11.0bp)   {}; 
  \node[N] [N] (1) at (11.0bp,69.0bp)   {}; 
  \node[N] [Z]  (2) at (31.0bp,127.0bp)   {}; 
  \node[N] [Z]  (6) at (51.0bp,69.0bp)   {}; 
  \node[N] [N] (3) at (31.0bp,243.0bp)   {}; 
  \node[N] [N] (4) at (11.0bp,185.0bp)   {}; 
  \node[N] [N] (5) at (51.0bp,185.0bp)   {}; 
  \node[N] [N] (7) at (71.0bp,127.0bp)   {}; 
  \draw [] (1) ..controls (18.048bp,48.265bp) and (23.836bp,32.06bp)  .. (0);
  \draw [] (2) ..controls (38.048bp,106.27bp) and (43.836bp,90.06bp)  .. (6);
  \draw [] (3) ..controls (23.952bp,222.27bp) and (18.164bp,206.06bp)  .. (4);
  \draw [] (3) ..controls (38.048bp,222.27bp) and (43.836bp,206.06bp)  .. (5);
  \draw [] (4) ..controls (18.048bp,164.27bp) and (23.836bp,148.06bp)  .. (2);
  \draw [] (5) ..controls (43.952bp,164.27bp) and (38.164bp,148.06bp)  .. (2);
  \draw [] (6) ..controls (43.952bp,48.265bp) and (38.164bp,32.06bp)  .. (0);
  \draw [] (7) ..controls (63.952bp,106.27bp) and (58.164bp,90.06bp)  .. (6);
\end{tikzpicture}
\quad
\begin{tikzpicture}[>=latex,line join=bevel,  scale = .2]
\node[N] [Z]  (0) at (31.0bp,11.0bp)   {}; 
  \node[N] [N] (1) at (11.0bp,69.0bp)   {}; 
  \node[N] [Z]  (2) at (51.0bp,69.0bp)   {}; 
  \node[N] [N] (3) at (31.0bp,185.0bp)   {}; 
  \node[N] [N] (4) at (51.0bp,127.0bp)   {}; 
  \node[N] [N] (5) at (11.0bp,127.0bp)   {}; 
  \node[N] [Z]  (6) at (91.0bp,127.0bp)   {}; 
  \node[N] [N] (7) at (71.0bp,185.0bp)   {}; 
  \draw [] (1) ..controls (18.048bp,48.265bp) and (23.836bp,32.06bp)  .. (0);
  \draw [] (2) ..controls (43.952bp,48.265bp) and (38.164bp,32.06bp)  .. (0);
  \draw [] (3) ..controls (38.048bp,164.27bp) and (43.836bp,148.06bp)  .. (4);
  \draw [] (3) ..controls (23.952bp,164.27bp) and (18.164bp,148.06bp)  .. (5);
  \draw [] (4) ..controls (51.0bp,105.92bp) and (51.0bp,90.474bp)  .. (2);
  \draw [] (5) ..controls (24.809bp,106.67bp) and (37.446bp,88.976bp)  .. (2);
  \draw [] (6) ..controls (77.191bp,106.67bp) and (64.554bp,88.976bp)  .. (2);
  \draw [] (7) ..controls (63.952bp,164.27bp) and (58.164bp,148.06bp)  .. (4);
  \draw [] (7) ..controls (78.048bp,164.27bp) and (83.836bp,148.06bp)  .. (6);
\end{tikzpicture}
\quad
\begin{tikzpicture}[>=latex,line join=bevel,  scale = .2]
\node[N] [Z]  (0) at (51.0bp,11.0bp)   {}; 
  \node[N] [N] (1) at (11.0bp,69.0bp)   {}; 
  \node[N] [Z]  (2) at (71.0bp,127.0bp)   {}; 
  \node[N] [Z]  (4) at (91.0bp,69.0bp)   {}; 
  \node[N] [Z]  (6) at (51.0bp,69.0bp)   {}; 
  \node[N] [N] (3) at (71.0bp,185.0bp)   {}; 
  \node[N] [N] (5) at (111.0bp,127.0bp)   {}; 
  \node[N] [N] (7) at (31.0bp,127.0bp)   {}; 
  \draw [] (1) ..controls (24.809bp,48.668bp) and (37.446bp,30.976bp)  .. (0);
  \draw [] (2) ..controls (78.048bp,106.27bp) and (83.836bp,90.06bp)  .. (4);
  \draw [] (2) ..controls (63.952bp,106.27bp) and (58.164bp,90.06bp)  .. (6);
  \draw [] (3) ..controls (71.0bp,163.92bp) and (71.0bp,148.47bp)  .. (2);
  \draw [] (4) ..controls (77.191bp,48.668bp) and (64.554bp,30.976bp)  .. (0);
  \draw [] (5) ..controls (103.95bp,106.27bp) and (98.164bp,90.06bp)  .. (4);
  \draw [] (6) ..controls (51.0bp,47.921bp) and (51.0bp,32.474bp)  .. (0);
  \draw [] (7) ..controls (38.048bp,106.27bp) and (43.836bp,90.06bp)  .. (6);
\end{tikzpicture}
\quad
\begin{tikzpicture}[>=latex,line join=bevel,  scale = .2]
\node[N] [Z]  (0) at (31.0bp,11.0bp)   {}; 
  \node[N] [N] (1) at (11.0bp,69.0bp)   {}; 
  \node[N] [Z]  (2) at (11.0bp,185.0bp)   {}; 
  \node[N] [Z]  (4) at (31.0bp,127.0bp)   {}; 
  \node[N] [N] (3) at (11.0bp,243.0bp)   {}; 
  \node[N] [Z]  (6) at (51.0bp,69.0bp)   {}; 
  \node[N] [N] (5) at (51.0bp,185.0bp)   {}; 
  \node[N] [N] (7) at (71.0bp,127.0bp)   {}; 
  \draw [] (1) ..controls (18.048bp,48.265bp) and (23.836bp,32.06bp)  .. (0);
  \draw [] (2) ..controls (18.048bp,164.27bp) and (23.836bp,148.06bp)  .. (4);
  \draw [] (3) ..controls (11.0bp,221.92bp) and (11.0bp,206.47bp)  .. (2);
  \draw [] (4) ..controls (38.048bp,106.27bp) and (43.836bp,90.06bp)  .. (6);
  \draw [] (5) ..controls (43.952bp,164.27bp) and (38.164bp,148.06bp)  .. (4);
  \draw [] (6) ..controls (43.952bp,48.265bp) and (38.164bp,32.06bp)  .. (0);
  \draw [] (7) ..controls (63.952bp,106.27bp) and (58.164bp,90.06bp)  .. (6);
\end{tikzpicture}
\qquad\qquad
\begin{tikzpicture}[>=latex,line join=bevel,  scale = .2]
\node[N] [Z]  (0) at (71.0bp,11.0bp)   {}; 
  \node[N] [N] (1) at (11.0bp,127.0bp)   {}; 
  \node[N] [N] (3) at (11.0bp,69.0bp)   {}; 
  \node[N] [N] (4) at (51.0bp,69.0bp)   {}; 
  \node[N] [Z]  (2) at (131.0bp,127.0bp)   {}; 
  \node[N] [Z]  (5) at (131.0bp,69.0bp)   {}; 
  \node[N] [Z]  (7) at (91.0bp,69.0bp)   {}; 
  \node[N] [N] (6) at (91.0bp,127.0bp)   {}; 
  \node[N] [N] (8) at (51.0bp,127.0bp)   {}; 
  \draw [] (1) ..controls (11.0bp,105.92bp) and (11.0bp,90.474bp)  .. (3);
  \draw [] (1) ..controls (24.809bp,106.67bp) and (37.446bp,88.976bp)  .. (4);
  \draw [] (2) ..controls (131.0bp,105.92bp) and (131.0bp,90.474bp)  .. (5);
  \draw [] (2) ..controls (117.19bp,106.67bp) and (104.55bp,88.976bp)  .. (7);
  \draw [] (3) ..controls (30.406bp,49.887bp) and (51.89bp,29.836bp)  .. (0);
  \draw [] (4) ..controls (58.048bp,48.265bp) and (63.836bp,32.06bp)  .. (0);
  \draw [] (5) ..controls (111.59bp,49.887bp) and (90.11bp,29.836bp)  .. (0);
  \draw [] (6) ..controls (77.191bp,106.67bp) and (64.554bp,88.976bp)  .. (4);
  \draw [] (6) ..controls (104.81bp,106.67bp) and (117.45bp,88.976bp)  .. (5);
  \draw [] (7) ..controls (83.952bp,48.265bp) and (78.164bp,32.06bp)  .. (0);
  \draw [] (8) ..controls (37.191bp,106.67bp) and (24.554bp,88.976bp)  .. (3);
  \draw [] (8) ..controls (64.809bp,106.67bp) and (77.446bp,88.976bp)  .. (7);
\end{tikzpicture}
\quad
\begin{tikzpicture}[>=latex,line join=bevel,  scale = .2]
\node[N] [Z]  (0) at (71.0bp,11.0bp)   {}; 
  \node[N] [N] (1) at (11.0bp,69.0bp)   {}; 
  \node[N] [Z]  (2) at (71.0bp,127.0bp)   {}; 
  \node[N] [Z]  (3) at (51.0bp,69.0bp)   {}; 
  \node[N] [Z]  (5) at (91.0bp,69.0bp)   {}; 
  \node[N] [Z]  (7) at (131.0bp,69.0bp)   {}; 
  \node[N] [N] (4) at (31.0bp,127.0bp)   {}; 
  \node[N] [N] (6) at (111.0bp,127.0bp)   {}; 
  \node[N] [N] (8) at (151.0bp,127.0bp)   {}; 
  \draw [] (1) ..controls (30.406bp,49.887bp) and (51.89bp,29.836bp)  .. (0);
  \draw [] (2) ..controls (63.952bp,106.27bp) and (58.164bp,90.06bp)  .. (3);
  \draw [] (2) ..controls (78.048bp,106.27bp) and (83.836bp,90.06bp)  .. (5);
  \draw [] (2) ..controls (90.406bp,107.89bp) and (111.89bp,87.836bp)  .. (7);
  \draw [] (3) ..controls (58.048bp,48.265bp) and (63.836bp,32.06bp)  .. (0);
  \draw [] (4) ..controls (38.048bp,106.27bp) and (43.836bp,90.06bp)  .. (3);
  \draw [] (5) ..controls (83.952bp,48.265bp) and (78.164bp,32.06bp)  .. (0);
  \draw [] (6) ..controls (103.95bp,106.27bp) and (98.164bp,90.06bp)  .. (5);
  \draw [] (7) ..controls (111.59bp,49.887bp) and (90.11bp,29.836bp)  .. (0);
  \draw [] (8) ..controls (143.95bp,106.27bp) and (138.16bp,90.06bp)  .. (7);
\end{tikzpicture}
\quad
\begin{tikzpicture}[>=latex,line join=bevel,  scale = .2]
\node[N] [Z]  (0) at (50.0bp,11.0bp)   {}; 
  \node[N] [N] (1) at (11.0bp,69.0bp)   {}; 
  \node[N] [Z]  (2) at (69.0bp,185.0bp)   {}; 
  \node[N] [Z]  (3) at (49.0bp,127.0bp)   {}; 
  \node[N] [Z]  (5) at (89.0bp,127.0bp)   {}; 
  \node[N] [N] (4) at (29.0bp,185.0bp)   {}; 
  \node[N] [Z]  (7) at (89.0bp,69.0bp)   {}; 
  \node[N] [N] (6) at (109.0bp,185.0bp)   {}; 
  \node[N] [N] (8) at (129.0bp,127.0bp)   {}; 
  \draw [] (1) ..controls (24.463bp,48.668bp) and (36.784bp,30.976bp)  .. (0);
  \draw [] (2) ..controls (61.952bp,164.27bp) and (56.164bp,148.06bp)  .. (3);
  \draw [] (2) ..controls (76.048bp,164.27bp) and (81.836bp,148.06bp)  .. (5);
  \draw [] (3) ..controls (49.278bp,94.304bp) and (49.72bp,43.962bp)  .. (0);
  \draw [] (4) ..controls (36.048bp,164.27bp) and (41.836bp,148.06bp)  .. (3);
  \draw [] (5) ..controls (89.0bp,105.92bp) and (89.0bp,90.474bp)  .. (7);
  \draw [] (6) ..controls (101.95bp,164.27bp) and (96.164bp,148.06bp)  .. (5);
  \draw [] (7) ..controls (75.537bp,48.668bp) and (63.216bp,30.976bp)  .. (0);
  \draw [] (8) ..controls (115.19bp,106.67bp) and (102.55bp,88.976bp)  .. (7);
\end{tikzpicture}
\quad
\begin{tikzpicture}[>=latex,line join=bevel,  scale = .2]
\node[N] [Z]  (0) at (51.0bp,11.0bp)   {}; 
  \node[N] [N] (1) at (31.0bp,69.0bp)   {}; 
  \node[N] [Z]  (2) at (51.0bp,185.0bp)   {}; 
  \node[N] [Z]  (3) at (31.0bp,127.0bp)   {}; 
  \node[N] [Z]  (5) at (71.0bp,127.0bp)   {}; 
  \node[N] [Z]  (7) at (71.0bp,69.0bp)   {}; 
  \node[N] [N] (4) at (11.0bp,185.0bp)   {}; 
  \node[N] [N] (6) at (91.0bp,185.0bp)   {}; 
  \node[N] [N] (8) at (111.0bp,127.0bp)   {}; 
  \draw [] (1) ..controls (38.048bp,48.265bp) and (43.836bp,32.06bp)  .. (0);
  \draw [] (2) ..controls (43.952bp,164.27bp) and (38.164bp,148.06bp)  .. (3);
  \draw [] (2) ..controls (58.048bp,164.27bp) and (63.836bp,148.06bp)  .. (5);
  \draw [] (3) ..controls (44.809bp,106.67bp) and (57.446bp,88.976bp)  .. (7);
  \draw [] (4) ..controls (18.048bp,164.27bp) and (23.836bp,148.06bp)  .. (3);
  \draw [] (5) ..controls (71.0bp,105.92bp) and (71.0bp,90.474bp)  .. (7);
  \draw [] (6) ..controls (83.952bp,164.27bp) and (78.164bp,148.06bp)  .. (5);
  \draw [] (7) ..controls (63.952bp,48.265bp) and (58.164bp,32.06bp)  .. (0);
  \draw [] (8) ..controls (97.191bp,106.67bp) and (84.554bp,88.976bp)  .. (7);
\end{tikzpicture}
\quad
\begin{tikzpicture}[>=latex,line join=bevel,  scale = .2]
\node[N] [Z]  (0) at (51.0bp,11.0bp)   {}; 
  \node[N] [N] (1) at (11.0bp,69.0bp)   {}; 
  \node[N] [Z]  (2) at (51.0bp,185.0bp)   {}; 
  \node[N] [Z]  (3) at (71.0bp,127.0bp)   {}; 
  \node[N] [Z]  (5) at (91.0bp,69.0bp)   {}; 
  \node[N] [Z]  (7) at (51.0bp,69.0bp)   {}; 
  \node[N] [N] (4) at (91.0bp,185.0bp)   {}; 
  \node[N] [N] (6) at (111.0bp,127.0bp)   {}; 
  \node[N] [N] (8) at (31.0bp,127.0bp)   {}; 
  \draw [] (1) ..controls (24.809bp,48.668bp) and (37.446bp,30.976bp)  .. (0);
  \draw [] (2) ..controls (58.048bp,164.27bp) and (63.836bp,148.06bp)  .. (3);
  \draw [] (3) ..controls (78.048bp,106.27bp) and (83.836bp,90.06bp)  .. (5);
  \draw [] (3) ..controls (63.952bp,106.27bp) and (58.164bp,90.06bp)  .. (7);
  \draw [] (4) ..controls (83.952bp,164.27bp) and (78.164bp,148.06bp)  .. (3);
  \draw [] (5) ..controls (77.191bp,48.668bp) and (64.554bp,30.976bp)  .. (0);
  \draw [] (6) ..controls (103.95bp,106.27bp) and (98.164bp,90.06bp)  .. (5);
  \draw [] (7) ..controls (51.0bp,47.921bp) and (51.0bp,32.474bp)  .. (0);
  \draw [] (8) ..controls (38.048bp,106.27bp) and (43.836bp,90.06bp)  .. (7);
\end{tikzpicture}
\quad
\begin{tikzpicture}[>=latex,line join=bevel,  scale = .2]
\node[N] [Z]  (0) at (51.0bp,11.0bp)   {}; 
  \node[N] [N] (1) at (31.0bp,69.0bp)   {}; 
  \node[N] [Z]  (2) at (11.0bp,243.0bp)   {}; 
  \node[N] [Z]  (3) at (31.0bp,185.0bp)   {}; 
  \node[N] [Z]  (5) at (51.0bp,127.0bp)   {}; 
  \node[N] [N] (4) at (51.0bp,243.0bp)   {}; 
  \node[N] [Z]  (7) at (71.0bp,69.0bp)   {}; 
  \node[N] [N] (6) at (71.0bp,185.0bp)   {}; 
  \node[N] [N] (8) at (91.0bp,127.0bp)   {}; 
  \draw [] (1) ..controls (38.048bp,48.265bp) and (43.836bp,32.06bp)  .. (0);
  \draw [] (2) ..controls (18.048bp,222.27bp) and (23.836bp,206.06bp)  .. (3);
  \draw [] (3) ..controls (38.048bp,164.27bp) and (43.836bp,148.06bp)  .. (5);
  \draw [] (4) ..controls (43.952bp,222.27bp) and (38.164bp,206.06bp)  .. (3);
  \draw [] (5) ..controls (58.048bp,106.27bp) and (63.836bp,90.06bp)  .. (7);
  \draw [] (6) ..controls (63.952bp,164.27bp) and (58.164bp,148.06bp)  .. (5);
  \draw [] (7) ..controls (63.952bp,48.265bp) and (58.164bp,32.06bp)  .. (0);
  \draw [] (8) ..controls (83.952bp,106.27bp) and (78.164bp,90.06bp)  .. (7);
\end{tikzpicture}
\quad

\begin{tikzpicture}[>=latex,line join=bevel,  scale = .2]
\node[N] [Z] (0) at (51.0bp,11.0bp) [draw=black,circle] {}; 
  \node[N]  (1) at (11.0bp,127.0bp) [draw=black,circle] {}; 
  \node[N]  (2) at (51.0bp,69.0bp) [draw=black,circle] {}; 
  \node[N]  (3) at (11.0bp,69.0bp) [draw=black,circle] {}; 
  \node[N] [Z] (4) at (51.0bp,127.0bp) [draw=black,circle] {}; 
  \node[N] [Z] (6) at (91.0bp,69.0bp) [draw=black,circle] {}; 
  \node[N]  (5) at (61.0bp,185.0bp) [draw=black,circle] {}; 
  \node[N]  (9) at (91.0bp,127.0bp) [draw=black,circle] {}; 
  \node[N]  (7) at (111.0bp,185.0bp) [draw=black,circle] {}; 
  \node[N]  (8) at (131.0bp,127.0bp) [draw=black,circle] {}; 
  \draw [] (1) ..controls (24.809bp,106.67bp) and (37.446bp,88.976bp)  .. (2);
  \draw [] (1) ..controls (11.0bp,105.92bp) and (11.0bp,90.474bp)  .. (3);
  \draw [] (2) ..controls (51.0bp,47.921bp) and (51.0bp,32.474bp)  .. (0);
  \draw [] (3) ..controls (24.809bp,48.668bp) and (37.446bp,30.976bp)  .. (0);
  \draw [] (4) ..controls (64.809bp,106.67bp) and (77.446bp,88.976bp)  .. (6);
  \draw [] (5) ..controls (57.392bp,163.8bp) and (54.587bp,148.09bp)  .. (4);
  \draw [] (5) ..controls (71.385bp,164.61bp) and (80.49bp,147.62bp)  .. (9);
  \draw [] (6) ..controls (77.191bp,48.668bp) and (64.554bp,30.976bp)  .. (0);
  \draw [] (7) ..controls (118.05bp,164.27bp) and (123.84bp,148.06bp)  .. (8);
  \draw [] (7) ..controls (103.95bp,164.27bp) and (98.164bp,148.06bp)  .. (9);
  \draw [] (8) ..controls (117.19bp,106.67bp) and (104.55bp,88.976bp)  .. (6);
  \draw [] (9) ..controls (77.191bp,106.67bp) and (64.554bp,88.976bp)  .. (2);
  \draw [] (9) ..controls (91.0bp,105.92bp) and (91.0bp,90.474bp)  .. (6);
\end{tikzpicture}
\quad
\begin{tikzpicture}[>=latex,line join=bevel,  scale = .2]
\node[N] [Z] (0) at (31.0bp,11.0bp) [draw=black,circle] {}; 
  \node[N]  (1) at (11.0bp,69.0bp) [draw=black,circle] {}; 
  \node[N] [Z] (2) at (51.0bp,69.0bp) [draw=black,circle] {}; 
  \node[N]  (3) at (51.0bp,243.0bp) [draw=black,circle] {}; 
  \node[N]  (5) at (91.0bp,185.0bp) [draw=black,circle] {}; 
  \node[N]  (7) at (51.0bp,185.0bp) [draw=black,circle] {}; 
  \node[N]  (9) at (11.0bp,185.0bp) [draw=black,circle] {}; 
  \node[N]  (4) at (11.0bp,127.0bp) [draw=black,circle] {}; 
  \node[N]  (6) at (51.0bp,127.0bp) [draw=black,circle] {}; 
  \node[N]  (8) at (91.0bp,127.0bp) [draw=black,circle] {}; 
  \draw [] (1) ..controls (18.048bp,48.265bp) and (23.836bp,32.06bp)  .. (0);
  \draw [] (2) ..controls (43.952bp,48.265bp) and (38.164bp,32.06bp)  .. (0);
  \draw [] (3) ..controls (64.809bp,222.67bp) and (77.446bp,204.98bp)  .. (5);
  \draw [] (3) ..controls (51.0bp,221.92bp) and (51.0bp,206.47bp)  .. (7);
  \draw [] (3) ..controls (37.191bp,222.67bp) and (24.554bp,204.98bp)  .. (9);
  \draw [] (4) ..controls (24.809bp,106.67bp) and (37.446bp,88.976bp)  .. (2);
  \draw [] (5) ..controls (77.191bp,164.67bp) and (64.554bp,146.98bp)  .. (6);
  \draw [] (5) ..controls (91.0bp,163.92bp) and (91.0bp,148.47bp)  .. (8);
  \draw [] (6) ..controls (51.0bp,105.92bp) and (51.0bp,90.474bp)  .. (2);
  \draw [] (7) ..controls (37.191bp,164.67bp) and (24.554bp,146.98bp)  .. (4);
  \draw [] (7) ..controls (64.809bp,164.67bp) and (77.446bp,146.98bp)  .. (8);
  \draw [] (8) ..controls (77.191bp,106.67bp) and (64.554bp,88.976bp)  .. (2);
  \draw [] (9) ..controls (11.0bp,163.92bp) and (11.0bp,148.47bp)  .. (4);
  \draw [] (9) ..controls (24.809bp,164.67bp) and (37.446bp,146.98bp)  .. (6);
\end{tikzpicture}
\quad
\begin{tikzpicture}[>=latex,line join=bevel,  scale = .2]
\node[N] [Z] (0) at (31.0bp,11.0bp) [draw=black,circle] {}; 
  \node[N]  (1) at (11.0bp,69.0bp) [draw=black,circle] {}; 
  \node[N] [Z] (2) at (91.0bp,127.0bp) [draw=black,circle] {}; 
  \node[N] [Z] (6) at (51.0bp,69.0bp) [draw=black,circle] {}; 
  \node[N]  (3) at (71.0bp,243.0bp) [draw=black,circle] {}; 
  \node[N]  (4) at (51.0bp,185.0bp) [draw=black,circle] {}; 
  \node[N]  (5) at (91.0bp,185.0bp) [draw=black,circle] {}; 
  \node[N]  (9) at (51.0bp,127.0bp) [draw=black,circle] {}; 
  \node[N]  (7) at (11.0bp,185.0bp) [draw=black,circle] {}; 
  \node[N]  (8) at (11.0bp,127.0bp) [draw=black,circle] {}; 
  \draw [] (1) ..controls (18.048bp,48.265bp) and (23.836bp,32.06bp)  .. (0);
  \draw [] (2) ..controls (77.191bp,106.67bp) and (64.554bp,88.976bp)  .. (6);
  \draw [] (3) ..controls (63.952bp,222.27bp) and (58.164bp,206.06bp)  .. (4);
  \draw [] (3) ..controls (78.048bp,222.27bp) and (83.836bp,206.06bp)  .. (5);
  \draw [] (4) ..controls (64.809bp,164.67bp) and (77.446bp,146.98bp)  .. (2);
  \draw [] (4) ..controls (51.0bp,163.92bp) and (51.0bp,148.47bp)  .. (9);
  \draw [] (5) ..controls (91.0bp,163.92bp) and (91.0bp,148.47bp)  .. (2);
  \draw [] (6) ..controls (43.952bp,48.265bp) and (38.164bp,32.06bp)  .. (0);
  \draw [] (7) ..controls (11.0bp,163.92bp) and (11.0bp,148.47bp)  .. (8);
  \draw [] (7) ..controls (24.809bp,164.67bp) and (37.446bp,146.98bp)  .. (9);
  \draw [] (8) ..controls (24.809bp,106.67bp) and (37.446bp,88.976bp)  .. (6);
  \draw [] (9) ..controls (51.0bp,105.92bp) and (51.0bp,90.474bp)  .. (6);
\end{tikzpicture}
\quad
\begin{tikzpicture}[>=latex,line join=bevel,  scale = .2]
\node[N] [Z] (0) at (51.0bp,11.0bp) [draw=black,circle] {}; 
  \node[N]  (1) at (11.0bp,69.0bp) [draw=black,circle] {}; 
  \node[N] [Z] (2) at (71.0bp,127.0bp) [draw=black,circle] {}; 
  \node[N] [Z] (6) at (91.0bp,69.0bp) [draw=black,circle] {}; 
  \node[N] [Z] (8) at (51.0bp,69.0bp) [draw=black,circle] {}; 
  \node[N]  (3) at (71.0bp,243.0bp) [draw=black,circle] {}; 
  \node[N]  (4) at (51.0bp,185.0bp) [draw=black,circle] {}; 
  \node[N]  (5) at (91.0bp,185.0bp) [draw=black,circle] {}; 
  \node[N]  (7) at (111.0bp,127.0bp) [draw=black,circle] {}; 
  \node[N]  (9) at (31.0bp,127.0bp) [draw=black,circle] {}; 
  \draw [] (1) ..controls (24.809bp,48.668bp) and (37.446bp,30.976bp)  .. (0);
  \draw [] (2) ..controls (78.048bp,106.27bp) and (83.836bp,90.06bp)  .. (6);
  \draw [] (2) ..controls (63.952bp,106.27bp) and (58.164bp,90.06bp)  .. (8);
  \draw [] (3) ..controls (63.952bp,222.27bp) and (58.164bp,206.06bp)  .. (4);
  \draw [] (3) ..controls (78.048bp,222.27bp) and (83.836bp,206.06bp)  .. (5);
  \draw [] (4) ..controls (58.048bp,164.27bp) and (63.836bp,148.06bp)  .. (2);
  \draw [] (5) ..controls (83.952bp,164.27bp) and (78.164bp,148.06bp)  .. (2);
  \draw [] (6) ..controls (77.191bp,48.668bp) and (64.554bp,30.976bp)  .. (0);
  \draw [] (7) ..controls (103.95bp,106.27bp) and (98.164bp,90.06bp)  .. (6);
  \draw [] (8) ..controls (51.0bp,47.921bp) and (51.0bp,32.474bp)  .. (0);
  \draw [] (9) ..controls (38.048bp,106.27bp) and (43.836bp,90.06bp)  .. (8);
\end{tikzpicture}
\quad
\begin{tikzpicture}[>=latex,line join=bevel,  scale = .2]
\node[N] [Z] (0) at (51.0bp,11.0bp) [draw=black,circle] {}; 
  \node[N]  (1) at (31.0bp,69.0bp) [draw=black,circle] {}; 
  \node[N] [Z] (2) at (31.0bp,185.0bp) [draw=black,circle] {}; 
  \node[N] [Z] (6) at (51.0bp,127.0bp) [draw=black,circle] {}; 
  \node[N]  (3) at (31.0bp,301.0bp) [draw=black,circle] {}; 
  \node[N]  (4) at (11.0bp,243.0bp) [draw=black,circle] {}; 
  \node[N]  (5) at (51.0bp,243.0bp) [draw=black,circle] {}; 
  \node[N] [Z] (8) at (71.0bp,69.0bp) [draw=black,circle] {}; 
  \node[N]  (7) at (71.0bp,185.0bp) [draw=black,circle] {}; 
  \node[N]  (9) at (91.0bp,127.0bp) [draw=black,circle] {}; 
  \draw [] (1) ..controls (38.048bp,48.265bp) and (43.836bp,32.06bp)  .. (0);
  \draw [] (2) ..controls (38.048bp,164.27bp) and (43.836bp,148.06bp)  .. (6);
  \draw [] (3) ..controls (23.952bp,280.27bp) and (18.164bp,264.06bp)  .. (4);
  \draw [] (3) ..controls (38.048bp,280.27bp) and (43.836bp,264.06bp)  .. (5);
  \draw [] (4) ..controls (18.048bp,222.27bp) and (23.836bp,206.06bp)  .. (2);
  \draw [] (5) ..controls (43.952bp,222.27bp) and (38.164bp,206.06bp)  .. (2);
  \draw [] (6) ..controls (58.048bp,106.27bp) and (63.836bp,90.06bp)  .. (8);
  \draw [] (7) ..controls (63.952bp,164.27bp) and (58.164bp,148.06bp)  .. (6);
  \draw [] (8) ..controls (63.952bp,48.265bp) and (58.164bp,32.06bp)  .. (0);
  \draw [] (9) ..controls (83.952bp,106.27bp) and (78.164bp,90.06bp)  .. (8);
\end{tikzpicture}
\quad
\begin{tikzpicture}[>=latex,line join=bevel,  scale = .2]
\node[N] [Z] (0) at (51.0bp,11.0bp) [draw=black,circle] {}; 
  \node[N]  (1) at (31.0bp,69.0bp) [draw=black,circle] {}; 
  \node[N] [Z] (2) at (51.0bp,127.0bp) [draw=black,circle] {}; 
  \node[N] [Z] (6) at (71.0bp,69.0bp) [draw=black,circle] {}; 
  \node[N]  (3) at (31.0bp,243.0bp) [draw=black,circle] {}; 
  \node[N]  (4) at (51.0bp,185.0bp) [draw=black,circle] {}; 
  \node[N]  (5) at (11.0bp,185.0bp) [draw=black,circle] {}; 
  \node[N]  (7) at (91.0bp,127.0bp) [draw=black,circle] {}; 
  \node[N] [Z] (8) at (91.0bp,185.0bp) [draw=black,circle] {}; 
  \node[N]  (9) at (81.0bp,243.0bp) [draw=black,circle] {}; 
  \draw [] (1) ..controls (38.048bp,48.265bp) and (43.836bp,32.06bp)  .. (0);
  \draw [] (2) ..controls (58.048bp,106.27bp) and (63.836bp,90.06bp)  .. (6);
  \draw [] (3) ..controls (38.048bp,222.27bp) and (43.836bp,206.06bp)  .. (4);
  \draw [] (3) ..controls (23.952bp,222.27bp) and (18.164bp,206.06bp)  .. (5);
  \draw [] (4) ..controls (51.0bp,163.92bp) and (51.0bp,148.47bp)  .. (2);
  \draw [] (5) ..controls (24.809bp,164.67bp) and (37.446bp,146.98bp)  .. (2);
  \draw [] (6) ..controls (63.952bp,48.265bp) and (58.164bp,32.06bp)  .. (0);
  \draw [] (7) ..controls (83.952bp,106.27bp) and (78.164bp,90.06bp)  .. (6);
  \draw [] (8) ..controls (77.191bp,164.67bp) and (64.554bp,146.98bp)  .. (2);
  \draw [] (9) ..controls (70.615bp,222.61bp) and (61.51bp,205.62bp)  .. (4);
  \draw [] (9) ..controls (84.608bp,221.8bp) and (87.413bp,206.09bp)  .. (8);
\end{tikzpicture}
\quad
\begin{tikzpicture}[>=latex,line join=bevel,  scale = .2]
\node[N] [Z] (0) at (51.0bp,11.0bp) [draw=black,circle] {}; 
  \node[N]  (1) at (11.0bp,69.0bp) [draw=black,circle] {}; 
  \node[N] [Z] (2) at (51.0bp,69.0bp) [draw=black,circle] {}; 
  \node[N]  (3) at (31.0bp,185.0bp) [draw=black,circle] {}; 
  \node[N]  (4) at (51.0bp,127.0bp) [draw=black,circle] {}; 
  \node[N]  (5) at (11.0bp,127.0bp) [draw=black,circle] {}; 
  \node[N] [Z] (6) at (91.0bp,127.0bp) [draw=black,circle] {}; 
  \node[N] [Z] (8) at (91.0bp,69.0bp) [draw=black,circle] {}; 
  \node[N]  (7) at (71.0bp,185.0bp) [draw=black,circle] {}; 
  \node[N]  (9) at (131.0bp,127.0bp) [draw=black,circle] {}; 
  \draw [] (1) ..controls (24.809bp,48.668bp) and (37.446bp,30.976bp)  .. (0);
  \draw [] (2) ..controls (51.0bp,47.921bp) and (51.0bp,32.474bp)  .. (0);
  \draw [] (3) ..controls (38.048bp,164.27bp) and (43.836bp,148.06bp)  .. (4);
  \draw [] (3) ..controls (23.952bp,164.27bp) and (18.164bp,148.06bp)  .. (5);
  \draw [] (4) ..controls (51.0bp,105.92bp) and (51.0bp,90.474bp)  .. (2);
  \draw [] (5) ..controls (24.809bp,106.67bp) and (37.446bp,88.976bp)  .. (2);
  \draw [] (6) ..controls (77.191bp,106.67bp) and (64.554bp,88.976bp)  .. (2);
  \draw [] (6) ..controls (91.0bp,105.92bp) and (91.0bp,90.474bp)  .. (8);
  \draw [] (7) ..controls (63.952bp,164.27bp) and (58.164bp,148.06bp)  .. (4);
  \draw [] (7) ..controls (78.048bp,164.27bp) and (83.836bp,148.06bp)  .. (6);
  \draw [] (8) ..controls (77.191bp,48.668bp) and (64.554bp,30.976bp)  .. (0);
  \draw [] (9) ..controls (117.19bp,106.67bp) and (104.55bp,88.976bp)  .. (8);
\end{tikzpicture}
\quad
\begin{tikzpicture}[>=latex,line join=bevel,  scale = .2]
\node[N] [Z] (0) at (71.0bp,11.0bp) [draw=black,circle] {}; 
  \node[N]  (1) at (11.0bp,69.0bp) [draw=black,circle] {}; 
  \node[N] [Z] (2) at (71.0bp,127.0bp) [draw=black,circle] {}; 
  \node[N] [Z] (4) at (51.0bp,69.0bp) [draw=black,circle] {}; 
  \node[N] [Z] (6) at (91.0bp,69.0bp) [draw=black,circle] {}; 
  \node[N] [Z] (8) at (131.0bp,69.0bp) [draw=black,circle] {}; 
  \node[N]  (3) at (71.0bp,185.0bp) [draw=black,circle] {}; 
  \node[N]  (5) at (31.0bp,127.0bp) [draw=black,circle] {}; 
  \node[N]  (7) at (111.0bp,127.0bp) [draw=black,circle] {}; 
  \node[N]  (9) at (151.0bp,127.0bp) [draw=black,circle] {}; 
  \draw [] (1) ..controls (30.406bp,49.887bp) and (51.89bp,29.836bp)  .. (0);
  \draw [] (2) ..controls (63.952bp,106.27bp) and (58.164bp,90.06bp)  .. (4);
  \draw [] (2) ..controls (78.048bp,106.27bp) and (83.836bp,90.06bp)  .. (6);
  \draw [] (2) ..controls (90.406bp,107.89bp) and (111.89bp,87.836bp)  .. (8);
  \draw [] (3) ..controls (71.0bp,163.92bp) and (71.0bp,148.47bp)  .. (2);
  \draw [] (4) ..controls (58.048bp,48.265bp) and (63.836bp,32.06bp)  .. (0);
  \draw [] (5) ..controls (38.048bp,106.27bp) and (43.836bp,90.06bp)  .. (4);
  \draw [] (6) ..controls (83.952bp,48.265bp) and (78.164bp,32.06bp)  .. (0);
  \draw [] (7) ..controls (103.95bp,106.27bp) and (98.164bp,90.06bp)  .. (6);
  \draw [] (8) ..controls (111.59bp,49.887bp) and (90.11bp,29.836bp)  .. (0);
  \draw [] (9) ..controls (143.95bp,106.27bp) and (138.16bp,90.06bp)  .. (8);
\end{tikzpicture}
\quad
\begin{tikzpicture}[>=latex,line join=bevel,  scale = .2]
\node[N] [Z] (0) at (51.0bp,11.0bp) [draw=black,circle] {}; 
  \node[N]  (1) at (11.0bp,69.0bp) [draw=black,circle] {}; 
  \node[N] [Z] (2) at (71.0bp,185.0bp) [draw=black,circle] {}; 
  \node[N] [Z] (4) at (91.0bp,127.0bp) [draw=black,circle] {}; 
  \node[N] [Z] (6) at (51.0bp,127.0bp) [draw=black,circle] {}; 
  \node[N]  (3) at (71.0bp,243.0bp) [draw=black,circle] {}; 
  \node[N]  (5) at (111.0bp,185.0bp) [draw=black,circle] {}; 
  \node[N] [Z] (8) at (51.0bp,69.0bp) [draw=black,circle] {}; 
  \node[N]  (7) at (31.0bp,185.0bp) [draw=black,circle] {}; 
  \node[N]  (9) at (11.0bp,127.0bp) [draw=black,circle] {}; 
  \draw [] (1) ..controls (24.809bp,48.668bp) and (37.446bp,30.976bp)  .. (0);
  \draw [] (2) ..controls (78.048bp,164.27bp) and (83.836bp,148.06bp)  .. (4);
  \draw [] (2) ..controls (63.952bp,164.27bp) and (58.164bp,148.06bp)  .. (6);
  \draw [] (3) ..controls (71.0bp,221.92bp) and (71.0bp,206.47bp)  .. (2);
  \draw [] (4) ..controls (85.07bp,102.86bp) and (78.388bp,78.359bp)  .. (71.0bp,58.0bp) .. controls (66.307bp,45.066bp) and (59.788bp,30.575bp)  .. (0);
  \draw [] (5) ..controls (103.95bp,164.27bp) and (98.164bp,148.06bp)  .. (4);
  \draw [] (6) ..controls (51.0bp,105.92bp) and (51.0bp,90.474bp)  .. (8);
  \draw [] (7) ..controls (38.048bp,164.27bp) and (43.836bp,148.06bp)  .. (6);
  \draw [] (8) ..controls (51.0bp,47.921bp) and (51.0bp,32.474bp)  .. (0);
  \draw [] (9) ..controls (24.809bp,106.67bp) and (37.446bp,88.976bp)  .. (8);
\end{tikzpicture}
\quad
\begin{tikzpicture}[>=latex,line join=bevel,  scale = .2]
\node[N] [Z] (0) at (51.0bp,11.0bp) [draw=black,circle] {}; 
  \node[N]  (1) at (31.0bp,69.0bp) [draw=black,circle] {}; 
  \node[N] [Z] (2) at (51.0bp,185.0bp) [draw=black,circle] {}; 
  \node[N] [Z] (4) at (71.0bp,127.0bp) [draw=black,circle] {}; 
  \node[N] [Z] (6) at (31.0bp,127.0bp) [draw=black,circle] {}; 
  \node[N]  (3) at (51.0bp,243.0bp) [draw=black,circle] {}; 
  \node[N] [Z] (8) at (71.0bp,69.0bp) [draw=black,circle] {}; 
  \node[N]  (5) at (91.0bp,185.0bp) [draw=black,circle] {}; 
  \node[N]  (7) at (11.0bp,185.0bp) [draw=black,circle] {}; 
  \node[N]  (9) at (111.0bp,127.0bp) [draw=black,circle] {}; 
  \draw [] (1) ..controls (38.048bp,48.265bp) and (43.836bp,32.06bp)  .. (0);
  \draw [] (2) ..controls (58.048bp,164.27bp) and (63.836bp,148.06bp)  .. (4);
  \draw [] (2) ..controls (43.952bp,164.27bp) and (38.164bp,148.06bp)  .. (6);
  \draw [] (3) ..controls (51.0bp,221.92bp) and (51.0bp,206.47bp)  .. (2);
  \draw [] (4) ..controls (71.0bp,105.92bp) and (71.0bp,90.474bp)  .. (8);
  \draw [] (5) ..controls (83.952bp,164.27bp) and (78.164bp,148.06bp)  .. (4);
  \draw [] (6) ..controls (44.809bp,106.67bp) and (57.446bp,88.976bp)  .. (8);
  \draw [] (7) ..controls (18.048bp,164.27bp) and (23.836bp,148.06bp)  .. (6);
  \draw [] (8) ..controls (63.952bp,48.265bp) and (58.164bp,32.06bp)  .. (0);
  \draw [] (9) ..controls (97.191bp,106.67bp) and (84.554bp,88.976bp)  .. (8);
\end{tikzpicture}
\quad
\begin{tikzpicture}[>=latex,line join=bevel,  scale = .2]
\node[N] [Z] (0) at (51.0bp,11.0bp) [draw=black,circle] {}; 
  \node[N]  (1) at (11.0bp,69.0bp) [draw=black,circle] {}; 
  \node[N] [Z] (2) at (71.0bp,127.0bp) [draw=black,circle] {}; 
  \node[N] [Z] (4) at (91.0bp,69.0bp) [draw=black,circle] {}; 
  \node[N] [Z] (6) at (51.0bp,69.0bp) [draw=black,circle] {}; 
  \node[N]  (3) at (51.0bp,185.0bp) [draw=black,circle] {}; 
  \node[N]  (5) at (111.0bp,127.0bp) [draw=black,circle] {}; 
  \node[N]  (7) at (31.0bp,127.0bp) [draw=black,circle] {}; 
  \node[N] [Z] (8) at (91.0bp,185.0bp) [draw=black,circle] {}; 
  \node[N]  (9) at (91.0bp,243.0bp) [draw=black,circle] {}; 
  \draw [] (1) ..controls (24.809bp,48.668bp) and (37.446bp,30.976bp)  .. (0);
  \draw [] (2) ..controls (78.048bp,106.27bp) and (83.836bp,90.06bp)  .. (4);
  \draw [] (2) ..controls (63.952bp,106.27bp) and (58.164bp,90.06bp)  .. (6);
  \draw [] (3) ..controls (58.048bp,164.27bp) and (63.836bp,148.06bp)  .. (2);
  \draw [] (4) ..controls (77.191bp,48.668bp) and (64.554bp,30.976bp)  .. (0);
  \draw [] (5) ..controls (103.95bp,106.27bp) and (98.164bp,90.06bp)  .. (4);
  \draw [] (6) ..controls (51.0bp,47.921bp) and (51.0bp,32.474bp)  .. (0);
  \draw [] (7) ..controls (38.048bp,106.27bp) and (43.836bp,90.06bp)  .. (6);
  \draw [] (8) ..controls (83.952bp,164.27bp) and (78.164bp,148.06bp)  .. (2);
  \draw [] (9) ..controls (91.0bp,221.92bp) and (91.0bp,206.47bp)  .. (8);
\end{tikzpicture}
\quad
\begin{tikzpicture}[>=latex,line join=bevel,  scale = .2]
\node[N] [Z] (0) at (51.0bp,11.0bp) [draw=black,circle] {}; 
  \node[N]  (1) at (31.0bp,69.0bp) [draw=black,circle] {}; 
  \node[N] [Z] (2) at (11.0bp,243.0bp) [draw=black,circle] {}; 
  \node[N] [Z] (4) at (31.0bp,185.0bp) [draw=black,circle] {}; 
  \node[N]  (3) at (11.0bp,301.0bp) [draw=black,circle] {}; 
  \node[N] [Z] (6) at (51.0bp,127.0bp) [draw=black,circle] {}; 
  \node[N]  (5) at (51.0bp,243.0bp) [draw=black,circle] {}; 
  \node[N] [Z] (8) at (71.0bp,69.0bp) [draw=black,circle] {}; 
  \node[N]  (7) at (71.0bp,185.0bp) [draw=black,circle] {}; 
  \node[N]  (9) at (91.0bp,127.0bp) [draw=black,circle] {}; 
  \draw [] (1) ..controls (38.048bp,48.265bp) and (43.836bp,32.06bp)  .. (0);
  \draw [] (2) ..controls (18.048bp,222.27bp) and (23.836bp,206.06bp)  .. (4);
  \draw [] (3) ..controls (11.0bp,279.92bp) and (11.0bp,264.47bp)  .. (2);
  \draw [] (4) ..controls (38.048bp,164.27bp) and (43.836bp,148.06bp)  .. (6);
  \draw [] (5) ..controls (43.952bp,222.27bp) and (38.164bp,206.06bp)  .. (4);
  \draw [] (6) ..controls (58.048bp,106.27bp) and (63.836bp,90.06bp)  .. (8);
  \draw [] (7) ..controls (63.952bp,164.27bp) and (58.164bp,148.06bp)  .. (6);
  \draw [] (8) ..controls (63.952bp,48.265bp) and (58.164bp,32.06bp)  .. (0);
  \draw [] (9) ..controls (83.952bp,106.27bp) and (78.164bp,90.06bp)  .. (8);
\end{tikzpicture}
    \caption{All McCarthy algebras up to size $10$. Each model $\m A$ is presented by its induced $s\ell$-decorated poset, where the nodes decorated by $\circ$ indicate membership in $\SL_{\m A}$, and the operations are computed via \cref{rem:ab}.}
    \label{fig:models10}
\end{figure}

A curiosity arises by inspecting \Cref{fig:models10}: no two McCarthy algebras (up to size 10) share the same induced poset; i.e., seemingly, the decoration is superfluous (up to isomorphism of {\iname}s). 
\begin{openproblem}
    If two McCarthy algebras are order-isomorphic, are the algebras isomorphic?
\end{openproblem}

\section{The lattice of subvarieties of {\iname}s}\label{sec: subvarieties}
In this section, we will return to the algebras described in \cref{t:3elem}, the varieties they generate, and their place in the lower-levels within the lattice of subvarieties of all {\iname}s (see \Cref{fig:subvars}).

Let us first recall that the subvariety of {\iname}s axiomatized by the identity $x\approx\nc{x}$ is term-equivalent to the variety of unital bands (equiv., idempotent monoids), which we will denote by $\mathsf{{uBand}}$. 
The lattice of subvarieties of $\mathsf{uBand}$ has been fully characterized in \cite{VarietiesBandMonoids}. The lower-levels of $\mathsf{uBand}$ find the variety of bounded semilattices $\mathsf{SL}$ as the unique atom, which itself has two covers, the varieties of \emph{left-regular} and \emph{right-regular} unital bands, here denoted as $\mathsf{LB}$ and $\mathsf{RB}$, respectively (their join being the variety of \emph{regular} unital bands, i.e., those satisfying $zxzyz \approx xzyz$). 
The following proposition therefore is immediate from this characterization, and from the fact that the $3$-element semilattice $\m C_3$ is commutative, and $\m L_3$ is left-regular and $\m R_3$ is right-regular with neither of the two commutative (see Figure \ref{f:ten-3elem}). 
 \begin{proposition}\label{t:ubands}
     $\mathsf{SL}=\mathsf{V}(\m{C_3})$,
     $\mathsf{LB}=\mathsf{V}(\m{L_3})$, and $\mathsf{RB}=\mathsf{V}(\m{R_3})$.
 \end{proposition}

We now recall and inspect those algebras in which the unit is an involution fixed-point, but $\nc{}$ is not the (redundant) identity map.
Recall that $\m C^\mathsf{s}_3$ is the {\iname}-reduct of the $3$-element Sugihara monoid, and let $\mathsf{C^s_3}$ denote the variety it generates. 
Also recall the left- and right- variants $\m L^\mathsf{s}_3$ and $\m R^\mathsf{s}_3$ (see Figure \ref{f:ten-3elem}); let us denote the respective varieties they generate by $\mathsf{L^s_3}$ and $\mathsf{R^s_3}$. 

First we show that $\mathsf{C^s_3}$ is a cover for $\mathsf{SL}$,
remarking first about some identities holding for $\mathsf{C^s_3}$.

\begin{remark}\label{rem:Cs3 idens}
    Recall that $\m C^\mathsf{s}_3$ is the three-element {\iname} with multiplication $*_{\mathsf{c}}$ in \Cref{f:3elemIdM} and involution $\prime_{\miden}$, hence it is isomorphic to the algebra with operations given below
    $$ 
    \begin{array}{l|l}
        &\prime   \\\hline
         \miden &\miden\\
         c & \nc{c}\\
         \nc{c} & c
    \end{array}
    \qquad
    \begin{array}{c|lll}
    \jc & \miden & c  & \nc{c}  \\\hline
    \miden & \miden  & c  & \nc{c}  \\
    c  & c  & c  & c  \\
    \nc{c}  & \nc{c}  & c  & \nc{c} 
    \end{array}
    \qquad
    \begin{array}{c|lll}
    \mc & \miden & c  & \nc{c}  \\\hline
    \miden & \miden  & c  & \nc{c}  \\
    c  & c  & c  & \nc{c}  \\
    \nc{c}  & \nc{c}  & \nc{c}  & \nc{c} 
    \end{array}    
    $$
    Clearly $\m C^\mathsf{s}_3$ is commutative and satisfies $\miden \approx \nc{\miden}$. 
    Moreover, $\m C^\mathsf{s}_3$ satisfies the following identities:
    \begin{equation}\label{eq: Cs}
        \text{(a)}~
        \fc{x} \approx \fc{\nc{x}}\approx \nc{\tc{\nc{x}}} \approx  \nc{\tc{x}}
        \qquad
        \text{(b)}~
        \tc{x}\mc x \approx x 
        \qquad
        \text{(c)}~
        \fc{x}\mc x \approx \fc{x}
    \end{equation}
    Indeed, it is easy to see that the identities above hold for $x=\miden$ as $\tc{\miden} =\miden = \fc{\miden}$ and $\miden = \nc{\miden}$, the latter of which entailing that $\miden$ is the unit for both $\mc$ and $\jc$.
    That they hold for $x\in \{c,\nc{c}\}$ follows as $\tc{c}:=c\jc \nc{c} = c $ and $\fc{c}=c\mc\nc{c}= \nc{c}$, so (a) holds by commutativity and (b,c) hold since $\mc$ is idempotent, $c\mc \nc{c} = \nc{c}$, and $c\jc \nc{c} = c$. 
\end{remark}

\begin{proposition}\label{t:C3cover}
    The variety $\mathsf{C^s_3}$ forms a cover for $\mathsf{SL}$ in the lattice of subvarieties of {\iname}s.
\end{proposition}
\begin{proof}
    That $\mathsf{SL}$ is a subvariety of $\mathsf{C^s_3}$; it suffices to verify that $\m C_2$ is a member of $\in \mathsf{C^s_3}$, as it generates the variety $\mathsf{SL}$. 
    Indeed, the map $f: C^\mathsf{s}_3\to C_2$ defined via $\miden \mapsto \top$ and $c,\nc{c}\mapsto \bot$ is easily verified to be a homomorphism from $\m C^\mathsf{s}_3$ onto $\m C_2$. So $\m C_2\in \mathsf{H}(\m C^\mathsf{s}_3)\subseteq \mathsf{C^s_3}$, thus establishing our claim.

    To verify it is a cover, suppose $\mathcal{V}$ is a subvariety of $\mathsf{C^s_3}$ different from $\mathsf{SL}$; we will show $\mathcal{V}=\mathsf{C^s_3}$. 
    Since $\mathsf{SL}$ is properly contained in $\mathcal{V}$, there must exist a member $\m A$ in which $\nc{}$ is not the identity map. 
    So let us fix $a\in A$ such that $\nc{a}\neq a$.
    By \cref{rem:Cs3 idens}, $\m A$ satisfies $\miden \approx \nc{\miden}$ and the identities \ref{eq: Cs}(a,b,c) (as by the assumption $\mathcal{V}\subseteq \mathsf{C^s_3}$).
    We consider the set $C:=\{\miden, \tc{a},\fc{a} \}\subseteq A$: we will show that $C$ is (the universe of) a subalgebra of $\m A$ isomorphic with $\m C^\mathsf{s}_3$. 
    We first observe that $|C|=3$. 
    One the one hand, if $\tc{a}=\fc{a}$ were the case then, from the identities \ref{eq: Cs}(b,c), we find $a =\tc{a}\mc a = \fc{a} \mc a = \fc{a} $, so $a=\fc{a}$ and thus $\nc{a} = \nc{\fc{a}}$, but \ref{eq: Cs}(a) gives $\nc{\fc{a}} = \tc{a} = \fc{a} = a$ which yields $a=\nc{a}$, a contradiction. 
    On the other hand, $\miden\in\{\tc{a},\fc{a}\}$ would entail $\tc{a}=\fc{a}$ since $\miden = \nc{\miden}$ and $\fc{a}=\nc{\tc{a}}$, which we have just shown contradictory. 
    Hence the elements of $C$ are pairwise distinct. 
    That $C$ is closed under the operations follows from the fact that constant $\miden \in C$; it is closed under $\nc{}$ since $\nc{\miden}=\miden$ and $\tc{a} = \nc{\fc{a}}$ (by \ref{eq: Cs}a); and it is multiplicatively closed since $\miden$ is the identity, $\mc$ is idempotent, and $\tc{a}\mc\fc{a} = \fc{a} = \fc{a}\mc\tc{a}$ (by \ref{eq: Cs}(b,c)).
    Finally, the map defined from $\tc{a}\mapsto c$ (hence $\fc{a}\mapsto \nc{c}$) is an isomorphism from $\m A$ onto $\m C^\mathsf{s}_3$. So $\m C^\mathsf{s}_3 \in \mathsf{IS}(\m A)\subseteq \mathcal{V}$. 
    Therefore $\mathsf{C^s_3}$ covers $\mathsf{SL}$.
\end{proof}
\begin{remark}
    In fact, the above proof shows that the variety $\mathsf{C^s_3}$ is subsumed by any variety properly extending $\mathsf{SL}$ and satisfying $\miden\approx\nc{\miden}$ and the identities from Eq.~\ref{eq: Cs} in \Cref{rem:Cs3 idens}.
\end{remark}

\begin{proposition}\label{t:leftsug}
    $\mathsf{L^s_3}$ covers $\mathsf{LB}$, and similarly, $\mathsf{R^s_3}$ covers $\mathsf{RB}$, in the lattice of subvarieties of {\iname}s.
\end{proposition}
\begin{proof}
Due to symmetry the of the operations, i.e., $\jc_{\m R^\mathsf{s}_3} = \jc^\mathsf{op}_{\m L^\mathsf{s}_3}$, it suffices to establish the claims for one of these cases. 
First we show that $\mathsf{LB}\subseteq \mathsf{L^s_3}$. 
On the one hand, we consider the direct product $ \m L^\mathsf{s}_3\times \m L^\mathsf{s}_3$ and its subalgebra $\m B$ generated by the set $\{ (a,a),(a, b)\}$. It is easily verified that this subalgebra has five elements, namely $B= \{(\miden,\miden), (a,a),(b,b), (a,b), (b,a) \}$ and that the map $(x,y)\mapsto x$ is a homomorphism from $\m B$ onto $\m L_3$. This establishes $\mathsf{LB}\subseteq\mathsf{L^s_3}$, and by the same argument, $\mathsf{RB}\subseteq\mathsf{R^s_3}$.

Now, suppose $\mathcal{V}$ is given such that $\mathsf{LB}\subsetneq\mathcal{V}\subseteq\mathsf{L^s_3}$, and let $\m A$ be any member of $\mathcal{V}$ not contained in $\mathsf{LB}$. 
So $\m A$ contains an element $a$ such that $\nc{a}\neq a$. 
Consider the subalgebra $\m L\leq \m A$ generated by $\tc{a}$. By assumption, $\m A \in \mathsf{L^s_3}$, and so $\m A$ must satisfy \ref{eq:leftreg} and the identity $xy\approx x\jc y$; the latter of which entailing $\tc{x}\approx \fc{x}$ holds as well. 
Now, as $\fc{a} = \tc{a} \in L$, its negation $\nc{\fc{a}} = \tc{\nc{a}} $ is a member of $L$. 
It follows that
$$\tc{a}\mc\nc{\tc{a}} = \tc{a}\mc \fc{\nc{a}} = \fc{a}\mc \fc{\nc{a}} =  a\nc{a} \mc \nc{a}a = a\nc{a} = \fc{a} = \tc{a}, $$
where the last equality follows from left-regularity. 
Similarly, $\nc{\tc{a}}\tc{a} = \nc{\tc{a}}$. 
Since $\nc{\miden} = \miden$ holds, it follows that the subalgebra  $\m L$ consists of exactly three element, $\miden$, $\tc{a}$, and $\nc{\tc{a}}$. 
It is easily verified that $\m L$ is isomorphic with $\m L^\mathsf{s}_3$ via homomorphism generated by the map $\tc{a}\mapsto a$. So $\m L^\mathsf{s}_3$ is a member of the variety $\mathcal{V}$, hence $\mathcal{V}=\mathsf{L^s_3}$. 
Therefore $\mathsf{L^s_3}$ covers $\mathsf{LB}$. 
By essentially the same argument (utilizing, instead, right-regularity), $\mathsf{R^s_3}$ covers $\mathsf{RB}$.
\end{proof}
\begin{remark}
The above proof shows that the variety $\mathsf{L^s_3}$ is subsumed by any variety properly extending $\mathsf{LB}$ that satisfies left-regularity and $x\mc y\approx x\jc y$.
\end{remark}

Lastly, as we have already discussed, the algebra $\m{WK}$ generates the variety of involutive bisemilattices, here denoted $\mathsf{WK}$, and the algebra $\m{SK}$ generates the variety of Kleene lattices, here denoted $\mathsf{SK}$. It is well known that $\mathsf{WK}$ contains $\m 2$ and $\m C_2$, and is a cover for the varieties of Boolean algebras and semilattices (see e.g., \cite{Bonzio16SL,Bonziobook}). It is also clear that $\mathsf{SK}$ covers $\mathsf{BA}$ as well. The fact that McCarthy algebras cover the variety of Boolean algebras was shown in \cref{t:McoversBA}, as well as $\MC^\mathsf{op}$ via \cref{t:MopChar}. We therefore have the following snapshot of the lower-levels for the lattice of subvarieties of {\iname}s, shown in \Cref{fig:subvars}.
\begin{figure}[ht]
    \centering
    \vspace{-1cm}
    \begin{tikzpicture}
        \node[label = {[xshift=0,yshift = -.1cm]\scriptsize \color{gray} $x\approx y$}] (T) at (0,0) {$\bullet$};
        \node (SL) at (-2,.5) {$\mathsf{SL}$};
        \node (BA)  at (2,.5)  {$\BA$};
        \node (LR) at (-6,1.5) {$\mathsf{LB}$};
        \node (RR) at (-4,1.5) {$\mathsf{RB}$};
        \node[label = {[xshift=.5cm, yshift=-.5cm]\scriptsize \color{gray}$\mathsf{uBands}$}](B) at (-7,2.5) {};
        \node (Ls) at (-6,3) {$\mathsf{L_3^s}$};
        \node (Rs) at (-4,3) {$\mathsf{R_3^s}$};
        \node (Cs) at (-1.75, 3) {$\mathsf{C_3^s}$};
        \node (WK) at (0, 3) {$\mathsf{WK}$};
        \node (SK) at (1.75, 3) {$\mathsf{SK}$};
        \node (Ml) at (4,3) {$\MC$};
        \node (Mr) at (6,3) {$\MC^\mathsf{op}$};

        \draw[thick] (T) -- (SL);
        \draw[thick] (SL) -- (LR);
        \draw[thick] (SL) -- (RR);        
        \draw[thick] (SL) -- (Cs);       
        \draw[thick] (SL) -- (WK);
        \draw[thick] (T) -- (BA);
        \draw[thick] (BA) -- (WK);
        \draw[thick] (BA) -- (SK);
        \draw[thick] (BA) -- (Ml);
        \draw[thick] (BA) -- (Mr);
        \draw[dashed,thick, gray,bend left=120]
            (SL) .. controls (-4,5.5) and (4,5.5) .. (BA) node[midway, below] {\scriptsize $xy\approx yx$};
        \draw[dashed,thick, gray,bend left=120]
            (-7,2.5) .. controls (-6,5) and (1,5) .. (SL) node[midway, below left] {\scriptsize $\miden\approx \nc{\miden}$};
        \draw[dashed,thick, gray,bend left=120]
            (-7,2.5) .. controls (-6,4.5) and (-2,4.5) .. (SL) node[midway, below left ] {\scriptsize $xy\approx x\jc y$};
        \draw[dashed, thick, black,bend left=120]
            (-7,2.5) .. controls (-7.4,2.5) and (-3,3) .. (SL) node[midway, below] {\scriptsize\color{gray} $x\approx \nc{x}$};
        \draw[dashed, thick, black]
            (-7,2.5) .. controls (-7.4,.5) and (-3,.5) .. (SL);
        \draw[black, thick, preaction={draw, white, line width=5pt}] (LR) -- (Ls);
        \draw[black, thick, preaction={draw, white, line width=5pt}]  (RR) -- (Rs);
        \draw[black, dashed, thick] (BA) .. controls (-1,4.5) and (3.4,4.5) .. (BA) node[midway, below] {\scriptsize\color{gray} $\mathsf{bIL}$};
    \end{tikzpicture}
    \caption{Some covering relations for the lattice of subvarieties of {\iname}s. The areas enclosed by the dashed lines indicate certain subvarieties; the ones with an identity represent the subvariety of {\iname}s relatively axiomatized by the corresponding identity; while $\mathsf{bIL}$ and $\mathsf{uBands}$ represent, respectively, the variety of bounded involutive lattices and (the term-equivalent) variety of unital bands.}
    \label{fig:subvars}
\end{figure}

With the exception of the varieties $\mathsf{C^s_3}$, $\mathsf{L^s_3}$, and $\mathsf{R^s_3}$, a finite equational basis has been found for each variety of {\iname}s generated by the $3$-element algebras from \cref{t:3elem}.  
To the best of our knowledge, it is an open problem finding an axiomatization for the varieties $\mathsf{C^s_3}$, $\mathsf{L^s_3}$, and $\mathsf{R^s_3}$.
So far, the only interesting identities we have observed, aside from $\miden\approx\nc{\miden}$ (and other than commutativity, left-regularity, and right-regularity, respectively) are those from \cref{t:sIDENS} and \Cref{rem:Cs3 idens} (and, of course, those identities entailed by them). 
The question remains whether or not these, and possibly some additional finite set, of identities axiomatize their respective varieties relative for {\iname}s. 
Specifically:
\begin{openproblem}
What is a finite basis, if any, for the variety $\mathsf{C^\mathsf{s}_3}$ generated by $\m C^\mathsf{s}_3$?
\end{openproblem}

\begin{openproblem}
What is a finite basis, if any, for the variety $\mathsf{L^\mathsf{s}_3}$ ($\mathsf{R^\mathsf{s}_3}$) generated by $\m L^\mathsf{s}_3$ ($\m R^\mathsf{s}_3$)?
\end{openproblem}

\section*{Acknowledgments}
The authors acknowledge the support by the Italian Ministry of Education, University and Research through the PRIN 2022 project DeKLA (``Developing Kleene Logics and their Applications'', project code: 2022SM4XC8). 
S. Bonzio also acknowledges the suppport of the PRIN Pnrr project ``Quantum Models for Logic, Computation and Natural Processes (Qm4Np)'' (cod. P2022A52CR), the Fondazione di Sardegna for the support received by the project MAPS (grant number F73C23001550007), the University of Cagliari for the support by the StartUp project ``GraphNet''. Finally, he gratefully acknowledges also the support of the INDAM GNSAGA (Gruppo Nazionale per le Strutture Algebriche, Geometriche e loro Applicazioni). 
G. St.\,John would like to thank Peter Jipsen for his assistance with and sharing of his \texttt{Prover9}/\texttt{Mace4} \cite{P9M4} repository (see \texttt{https://github.com/jipsen}), which has been an indispensable tool in this work.
\bibliographystyle{abbrv}
\bibliography{MCref}
\end{document}